\documentclass[aos,preprint]{imsart}

\RequirePackage[OT1]{fontenc}
\RequirePackage{amsthm,amsmath}
\usepackage[authoryear,round,compress,sort]{natbib} 	
\RequirePackage[colorlinks,citecolor=blue,urlcolor=blue]{hyperref}
\usepackage{amsmath,amssymb,amsfonts,mathrsfs,ulem,multirow}
\usepackage{mathtools}
\usepackage{fancyhdr}
\RequirePackage{bm}

\startlocaldefs
\numberwithin{equation}{section}
\theoremstyle{plain}

\endlocaldefs

\newcommand{\bfm}[1]{\ensuremath{\mathbf{#1}}}
\def\ba{\bfm a}     \def\bA{\bfm A}          
\def\bb{\bm b}     \def\bB{\bfm B}          
               
     \def\bD{\bfm D}          
               
\def\bff{\bm f}              
     \def\bG{\bm G}          
               
     \def\bI{\bfm I}

     \def\bR{\bfm R}          
     \def\bS{\bfm S}          
     \def\bT{\bfm T}          
\def\bu{\bm u}               
               
     \def\bW{\bfm W}          
     \def\bX{\bm{X}}         
               
\def\bz{\bm z}     \def\bZ{\bm Z}

\newcommand{\bfsym}[1]{\ensuremath{\boldsymbol{#1}}}
       \def \bbeta    {\bfsym{\beta}}
\def \bgamma   {\bfsym{\gamma}}       \def \bdelta   {\bfsym{\delta}}
     \def \bzeta    {\bfsym{\zeta}}
         \def \btheta   {\bfsym{\theta}}
        
      \def \bmu      {\bfsym{\mu}}
\def \bnu      {\bfsym{\nu}}          \def \bxi      {\bfsym{\xi}}

       \def \bDelta   {\bfsym{\Delta}}

\def \bSigma   {\bfsym{\Sigma}}

\renewcommand{\hat}{\widehat}

\def \heps     {\hat{\heps}}

\DeclareMathOperator*{\argmin}{argmin}
\DeclareMathOperator{\corr}{corr}
\DeclareMathOperator{\cov}{cov}

\DeclareMathOperator{\diag}{diag}

\DeclareMathOperator{\sgn}{sgn}

\def \var   {\mbox{var}}
\def \corr  {\mbox{corr}}

\def \sgn   {\mbox{sgn}}

\def\today{\ifcase\month\or
  January\or February\or March\or April\or May\or June\or
  July\or August\or September\or October\or November\or December\fi
  \space\number\day, \number\year}

\def \wt      {\widetilde}

\def \newpage {\vfill\eject}

\newdimen\biblioindent\biblioindent=30pt
\newcommand{\beq}  {\begin{equation}}
\newcommand{\eeq}  {\end{equation}}
\newcommand{\beqn} {\begin{eqnarray}}
\newcommand{\eeqn} {\end{eqnarray}}
\newcommand{\beqnn}{\begin{eqnarray*}}
\newcommand{\eeqnn}{\end{eqnarray*}}

\newcommand{\mo}{\mathbf{0}}
\newcommand{\sn}{\sum_{i=1}^n}

\newcommand{\e}{\mathbb{E}}

\def\bbr{ \mathbb{R}}
\newcommand{\nn}{\nonumber}
\newtheorem{theorem}{Theorem}[section]
\newtheorem{remark}{Remark}[section]
\newtheorem{assumption}{Condition}[section]
\newtheorem{lemma}{Lemma}[section]
\newtheorem{proposition}{Proposition}[section]

\newcounter{CondCounter}

\newcommand{\bee}{\begin{equation}}
\newcommand{\eee}{\end{equation}}

\DeclarePairedDelimiter\ceil{\lceil}{\rceil}
\DeclarePairedDelimiter\floor{\lfloor}{\rfloor}
\newcommand{\FDP}{{\rm FDP}}
\newcommand{\Card}{{\rm Card}}
\newcommand{\Id}{{\rm Id}}
\newcommand{\T}{\intercal}

\begin{document}

\begin{frontmatter}
\title{{A New Perspective on Robust $M$-Estimation: Finite Sample Theory and Applications to Dependence-Adjusted Multiple Testing}\thanksref{T1}}
\runtitle{Robust Estimation: New Perspective and Applications}
\thankstext{T1}{J. Fan and W.-X. Zhou were partially supported by NIH Grant R01--GM072611 and NSF Grant DMS--1206464. K. Bose was supported in part by NSF Grant DGE--1148900. H. Liu was supported by NSF Grants DMS--1454377, IIS--1332109, IIS--1408910 and IIS--1546482, and NIH Grants R01--MH102339 and R01--GM083084.}

\begin{aug}
\author{\fnms{Wen-Xin} \snm{Zhou}\thanksref{m0,m1}\ead[label=e1]{wez243@ucsd.edu}},
\author{\fnms{Koushiki} \snm{Bose}\thanksref{m1}\ead[label=e2]{bose@princeton.edu}},
\author{\fnms{Jianqing} \snm{Fan}\thanksref{m2,m1}\ead[label=e3]{jqfan@princeton.edu}} \\
\and
\author{\fnms{Han} \snm{Liu}\thanksref{m1}
\ead[label=e4]{hanliu@princeton.edu}
\ead[label=u1,url]{http://www.foo.com}}

\runauthor{Zhou, Bose, Fan and Liu}

\affiliation{University of California, San Diego\thanksmark{m0}, Princeton University\thanksmark{m1} \\
and Fudan University\thanksmark{m2}}

\address{W.-X. Zhou \\
Department of Mathematics \\
University of California, San Diego \\
La Jolla, California 92093 \\
USA \\
\printead{e1}}

\address{
K. Bose \\
H. Liu \\
Department of Operations Research \\
~~~and Financial Engineering \\
Princeton University \\
Princeton, New Jersey 08544 \\
USA \\
{E-mail: }\printead*{e2}  \\ \phantom{E-mail:} \printead*{e4}}

\address{
J. Fan\\
School of Data Science \\
Fudan University \\
Shanghai 200433 \\
China \\
and \\
Department of Operations Research \\
~~~and Financial Engineering \\
Princeton University \\
Princeton, New Jersey 08544 \\
USA \\
\printead{e3}}
\end{aug}

\begin{abstract}
Heavy-tailed errors impair the accuracy of the least squares estimate, which can be spoiled by a single grossly outlying observation. As argued in the seminal work of Peter Huber in 1973 [{\it Ann. Statist.} {\bf 1} (1973)  799--821], robust alternatives to the method of least squares are sorely needed. To achieve robustness against heavy-tailed sampling distributions, we revisit the Huber estimator from a new perspective by letting the tuning parameter involved diverge with the sample size. In this paper, we develop nonasymptotic concentration results for such an adaptive Huber estimator, namely, the Huber estimator with the tuning parameter adapted to sample size, dimension, and the variance of the noise. Specifically, we obtain a sub-Gaussian-type deviation inequality and a nonasymptotic Bahadur representation when noise variables only have finite second moments. The nonasymptotic results further yield two conventional normal approximation results that are of independent interest, the Berry-Esseen inequality and Cram\'er-type moderate deviation. As an important application to large-scale simultaneous inference, we apply these robust normal approximation results to analyze a dependence-adjusted multiple testing procedure for moderately heavy-tailed data. It is shown that the robust dependence-adjusted procedure asymptotically controls the overall false discovery proportion at the nominal level under mild moment conditions. Thorough numerical results on both simulated and real datasets are also provided to back up our theory.
\end{abstract}



\end{frontmatter}

\section{Introduction}
\label{sec1}

High dimensional data are often automatically collected with low quality. For each feature, the samples drawn from a moderate-tailed distribution may comprise one or two very large outliers in the measurements. When dealing with thousands or tens of thousands of features simultaneously, the chance of including a fair amount of outliers is high. Therefore, the development of robust procedures is arguably even more important for high dimensional problems. In this paper, we develop a finite sample theory for robust $M$-estimation from a new perspective. Such a finite sample theory is motivated by contemporary statistical problems of simultaneously testing many hypotheses.  In these problems, the goal is either to control the false discovery rate (FDR)/false discovery proportion (FDP), or to control the familywise error rate (FWER).

The main contributions of this paper are described and summarized in the following two subsections.

\subsection{A finite sample theory of robust $M$-estimation}

Consider a linear regression model $Y=\mu^* + \bX^\T \bbeta^* + \varepsilon$, where $\mu^* \in \bbr$ is the intercept, $\bbeta^* \in \bbr^d$ is the vector of regression coefficients, $\bX\in \bbr^d$ is the vector of covariates and $\varepsilon \in \bbr$ is the random noise variable with mean zero and finite variance. Assuming that  $\varepsilon$ follows a normal distribution, statistical properties of the ordinary least squares (OLS) estimator of $\mu^*$ and $\bbeta^*$ have been well studied. When the normality assumption is violated, quoting from \cite{H1973}, ``a single grossly outlying observation may spoil the the least squares estimate'' and therefore robust alternatives to the method of least squares, typified by the Huber estimator, are sorely needed. However, unlike the OLS estimator, all the existing theoretical results for the Huber estimator are asymptotic, including asymptotic normality [\cite{H1973}, \cite{YM1979}, \cite{P1985}, \cite{M1989}] and the Bahadur representation [\cite{HS1996, HS2000}]. The main reason for the lack of nonasymptotic results is  that the Huber estimator does not have an explicit closed-form expression, while most existing nonasymptotic analyses of the OLS estimator rely on its closed-form expression.

The first contribution of this paper is to develop a new finite sample theory for the Huber estimator. Recall the Huber loss [\cite{H1964}]:
\begin{align} \label{huber.loss}
	\ell_\tau(u) = \begin{cases}
		\frac{1}{2} u^2  &~\mbox{ if } |u| \leq \tau \\
		\tau |u| - \frac{1}{2} \tau^2  &~\mbox{ if } |u| > \tau
	\end{cases},
\end{align}
a hybrid of squared loss for relatively small errors and absolute loss for large errors, where the degree of hybridization is controlled by the tuning parameter $\tau>0$ that balances robustness and efficiency. In line with this notation, we use $\ell_\infty$ to denote the quadratic loss $\ell_\infty(u) = u^2/2$, $u\in \bbr$. Let $\{ (Y_i, \bX_i  )  \}_{i=1}^n$  be independent random samples from $(Y,\bX)$. We define the robust  $M$-estimator of $\btheta^*  := (\mu^*, \bbeta^{*\T})^\T$ by
\begin{align}
	\hat{\btheta}   := ( \hat{\mu} , \hat{\bbeta}^\T )^\T \in  \argmin_{ \mu \in \bbr, \, \bbeta \in \bbr^{d}  } \sn \ell_\tau( Y_i    -  \mu -  \bX_i^\T \bbeta ) . \label{robust.est}
\end{align}
The dependence of $\hat{\btheta}$ or $(\hat{\mu}, \hat{\bbeta})$ on $\tau$ will be assumed without displaying. It is worth noticing that our robustness concern is rather different from the conventional sense. In Huber's robust location estimation [\cite{H1964}], it is assumed that the error distribution lies in the neighborhood of a normal distribution, which gives the possibility to replace the mean by some location parameter. Unless the shape of the distribution is constrained (e.g., symmetry), in general this location parameter is different from the mean. Our interest, however, is focused on mean estimation in the heavy-tailed case where the error distribution is allowed to be asymmetric and to exhibit heavy tails. Therefore, unlike the classical Huber estimator [\cite{H1973}] which requires $\tau$ to be fixed, we allow $\tau$ to diverge  with the sample size $n$ such that $\ell_\tau(\cdot)$ can be viewed as a robust approximate quadratic loss function. As in \cite{FLW2014}, this is needed to reduce the bias for estimating the (conditional) mean function when the (conditional) distribution of $\varepsilon$ is asymmetric and heavy-tailed. In particular, by taking $\tau=\infty$, $\hat{\btheta}$ coincides with the OLS estimator of $\btheta^*$ and by shrinking $\tau$ toward 0, the resulting estimator approaches the least absolute deviation (LAD) estimator.

For every $\tau >0 $, by definition $\hat{\btheta} $ is an  $M$-estimator of
\bee \label{def.thetatau}
   \btheta^*_\tau := ( \mu_\tau , \bbeta_\tau^\T )^\T = \argmin_{  \mu \in \bbr, \, \bbeta \in \bbr^{d}  } \e \{ \ell_\tau( Y  -  \mu -    \bX^\T \bbeta ) \} ,
\eee
which typically differs from the target parameter 
$$
	\btheta^*  = \argmin_{  \mu \in \bbr, \, \bbeta \in \bbr^{d}   } \e \{\ell_\infty( Y   -  \mu  - \bX^\T \bbeta  )\}. 
$$
Note that the total error $\| \hat{\btheta} - \btheta^*\|  $ can be divided into two parts:
$$
	\underbrace{ \| \hat{\btheta }  - \btheta^* \| }_{{\rm Total\,error}} \leq  \underbrace{ \| \hat{\btheta} - \btheta^*_\tau \| }_{{\rm Estimation\,error}} + \underbrace{ \| \btheta^*_\tau - \btheta^* \| }_{{\rm Approximation\,error}} ,
$$
where $\| \cdot\|$ denotes the Euclidean norm.  We define $\text{Bias}(\tau): = \|  \btheta^*_\tau - \btheta^* \|$ to be the approximation error brought by the Huber loss. Proposition~\ref{RAlem1} in the supplemental material [\cite{ZBFL2017}] shows that $\text{Bias}(\tau)$ scales at the rate $\tau^{-1}$, which decays as $\tau$ grows.  A large $\tau$ reduces the approximation bias but jeopardizes the degree of robustness. Hence, the tuning parameter $\tau$ controls this bias and robustness trade-off of the estimator. Our main theorem (Theorem \ref{RAthm1}) reveals the concentration property of $\hat{\btheta}$ and provides a nonasymptotic Bahadur representation for the difference $\hat{\btheta} - \btheta^*$.  Such a Bahadur representation gives a finite sample approximation of $\hat{\btheta}$ by a sum of independent variables with a higher-order remainder. More specifically, let  $\bSigma = \e ( \bX \bX^\T ) \in \bbr^{d\times d}$ and $\ell'_\tau(\cdot)$ be the derivative function of $\ell_\tau(\cdot)$. Then, with properly chosen $\tau =\tau_n$ we have
\bee
 	 \bigg\|   \left[\begin{array}{c}
\hat{\mu} - \mu^* \\
\bSigma^{1/2}(\hat{\bbeta} - \bbeta^* )
\end{array}\right]   - \frac{1}{n} \sn \ell'_\tau(\varepsilon_i)  \left[\begin{array}{c}
1 \\
\bSigma^{-1/2}\bX_i
\end{array}\right]  \bigg\| \leq  R_n(\tau) \simeq  \frac{d}{n},  \label{na.bound}
\eee
 where $R_n(\tau)$ is a finite sample error bound which characterizes the accuracy of such a linear approximation, and $d$ is the number of covariates that may grow with $n$. We refer to Theorem \ref{RAthm1} for a rigorous description of the result \eqref{na.bound}, where we obtain an exponential-type deviation inequality for this Bahadur representation.

Many asymptotic Bahadur-type representations for robust  $M$-estimators have been obtained in the literature; see, for example, \cite{P1985}, \cite{M1989} and \cite{HS1996, HS2000}, among others. Our result \eqref{na.bound}, however, is nonasymptotic and provides an explicit tail bound for the remainder term $R_n(\tau)$. To obtain such a result, we first derive a sub-Gaussian-type deviation bound for $\hat{\btheta}$, and then conduct a careful analysis on the higher-order remainder term using this bound and techniques from empirical process theory. The expansion \eqref{na.bound} further yields two classical normal approximation results, the Berry-Esseen inequality and Cram\'er-type moderate deviation. These results have important applications to large-scale inference [\citet{FHY2007}, \citet{DHJ2011}, \cite{LS2014}, \cite{CSZ2016}]. Consider the statistical problems of simultaneously testing many hypotheses with FDR/FDP control or globally inferring a high dimensional parameter. For multiple testing, the obtained Berry-Esseen bound and Cram\'er-type moderate deviation result can be used to investigate the robustness and accuracy of the $P$-values and critical values.  For globally testing a high dimensional parameter, the expansion \eqref{na.bound}, combined with the parametric bootstrap, can be used to construct a valid test. In this paper, we only focus on the large-scale multiple testing problem and leave the global testing in high dimensions for future research.

\subsection{FDP control for robust dependent tests}

We apply the Bahadur representation \eqref{na.bound} to construct robust dependence-adjusted test statistics for simultaneous inference.  Conventional tasks of large-scale multiple testing, including controlling the FDR/FDP or FWER,  have been extensively explored and are now well understood when the test statistics are independent [\cite{BH1995}, \cite{S2002}, \cite{GW2004}, \cite{ LR2005}]. It is becoming increasingly important to understand and incorporate the dependence information among multiple test statistics. Under the positive regression dependence condition, the FDR control can be conducted in the same manner as that for the independent case [\cite{BY2001}], which provides a conservative upper bound. For more general dependence, directly applying standard FDR control procedures developed for independent $P$-values may lead to inaccurate false discovery rate control and include too many spurious discoveries [\cite{E2004, E2007}, \cite{SC2009}, \cite{CH2009}, \cite{SL2011}, \cite{FHG2012}]. In this more challenging situation, various multi-factor models have been used to investigate the dependence structure in high dimensional data; see, for example, \cite{LS2008}, \cite{FKC2009},  \cite{DS2012} and \cite{FHG2012}.

The multi-factor model relies on the identification of a linear space of random vectors capturing the dependence structure of the data. \cite{FKC2009} and \cite{DS2012} assume that the data are drawn from a strict factor model with independent idiosyncratic errors. They use the expectation-maximization algorithm to estimate the factor loadings and realized factors in the model, and then obtain an estimator for the FDP by subtracting out realized common factors. These methods, however, depend on stringent model assumptions, including the independence of idiosyncratic errors and joint normality of the factor and noise. In contrast, \cite{FHG2012} and \cite{FH2017} use a more general approximate factor model that allows dependent noise.

Let $\bX = (X_1, \ldots, X_p)^\T $ be a $p$-dimensional random vector with mean $\bmu = (\mu_1,\ldots,\mu_p)^\T $ and covariance matrix $\bSigma = (\sigma_{jk})_{1\leq j,k\leq p}$. We aim to simultaneously test
\begin{align}  \label{multiple.test}
	H_{0j} :  \mu_j = 0 \ \ \mbox{ versus } \ \ H_{1j}: \mu_j \neq 0 , \ \ \mbox{ for } j=1,\ldots , p.
\end{align}
We are interested in the case where there is strong dependence across the components of $\bX$. The approximate factor model assumes that the dependence of a high dimensional random vector $\bX$ can be captured by a few factors, that is,
\begin{align}
	\bX = \bmu + \bB \bff + \bSigma(\bff) \bu ,   \label{factor.model}
\end{align}
from which we observe independent random samples $(\bX_1 , \bff_1 ) , \ldots, (\bX_n , \bff_n )$ satisfying
$$
	\bX_i = (X_{i1}, \ldots, X_{ip})^\T = \bmu + \bB   \bff_i + \bSigma(\bff_i) \bu_i, \ \ i=1,\ldots, n.
$$
Here, $\bB = (\bb_1 , \ldots, \bb_p )^\T \in \bbr^{p \times K}$ represents the factor loading matrix, $\bff_i$ is the $K$-dimensional common factor to the $i$th observation and is independent of the idiosyncratic noise $\bu_i \in \bbr^p$, and $\bSigma(\bff)=\diag( \sigma_1(\bff), \ldots, \sigma_p(\bff))$ with $\sigma_j(\cdot): \bbr^K \mapsto (0,\infty)$ as unknown variance heteroscedasticity functions. We allow the components of $\bu_i$ to be dependent. To fully understand the influence of heavy tailedness, in this paper we restrict our attention to such an observable factor model.

For testing the hypotheses in \eqref{multiple.test} under model \eqref{factor.model}, when the factor $\bff$ is unobserved, a popular and natural approach is based on (marginal) sample averages of $\{\bX_i\}$ with a focus on the valid control of FDR [\cite{BY2001}, \cite{GW2004}, \cite{E2007}, \cite{SC2009}, \cite{SL2011}, \cite{FH2017}].  As pointed out in \cite{FHG2012}, the power of such an approach is dominated by the factor-adjusted approach that produces alternative rankings of statistical significance from those of the marginal statistics.  They focus on a Gaussian model where both $\bff$ and $\bu$ follow multivariate normal distributions, while statistical properties of the corresponding test procedure on FDR/FDP control remain unclear. The normality assumption, however, is really an idealization which provides insights into the key issues underlying the problems.
Data subject to heavy-tailed and asymmetric errors are repeatedly observed in many fields of research [\cite{FR2003}].
For example, it is known that financial returns typically exhibit heavy tails. The important papers by \cite{M1963} and \cite{F1963} provide evidence of power-law behavior in asset prices in the early 1960s.  Since then, the non-Gaussian character of the distribution of price changes has been widely observed in various market data. \cite{C2001} provides further evidence showing that a Student's $t$-distribution with four degrees of freedom displays a tail behavior similar to many asset returns.

For multiple testing with heavy-tailed data, the least squares based test statistics are sensitive to outliers and thus lack robustness. This issue is amplified further by high dimensionality: When the dimension is large,  even moderate tails may lead to significant false discoveries. This motivates us to develop new test statistics that are robust to the tails of error distributions. Also, since the multiple testing problem is more complicated with dependent data, theoretical guarantees of the FDP control for the existing dependence-adjusted methods remain unclear.

To illustrate the impact of heavy tailedness, we generate independent and identically distributed (i.i.d.) random variables $\{ X_{i j}, \, i= 1, \ldots , 30 , \, j = 1, \ldots, 10000 \}$ from a normalized $t$-distribution with $2.5$ degrees of freedom. 
In Figure~\ref{fig:histogram}, we compare the histogram of the empirical means $\overline{X}_j$ with that of the robust mean estimates constructed using \eqref{def.thetatau} without covariates, after rescaling both estimators by $\sqrt{30}$. For a standard normal distribution, we expect $99.73\%$ data points to lie within three standard deviations of the mean or inside $[-3,3]$. Hence for this experiment, if the distribution of the estimator is indeed approximately normal, we would expect about $27$ out of $10000$ realizations to lie outside $[-3,3]$. From Figure~\ref{fig:histogram} we see that the robust procedure gives 28 points that fall outside this interval, whereas the sample average gives a much larger number, 79, many of which would surely be falsely regarded as signals. We see that in the presence of heavy tails and high dimensions, the robust method leads to a more accurate normal tail approximation than using a nonrobust one. In fact, for the empirical means, many of them even fall outside $[-6,6]$. This inaccuracy in tail approximation for the nonrobust estimator gives rise to false discoveries. In summary, outliers from the test statistics $\overline{X}_j$ can be so large that they are mistakenly regarded as discoveries, whereas the robust approach results in fewer outliers.

\begin{figure}        \centering
                \includegraphics[width=.85\textwidth,]{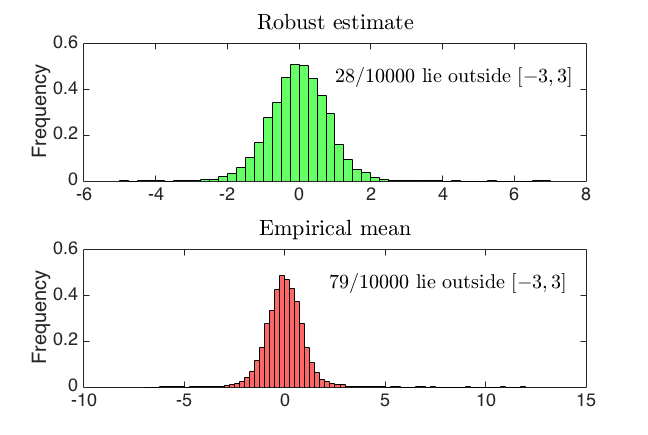}
                \caption{Histograms of 10000 robust mean estimates and empirical means based on 30 i.i.d. samples drawn from a normalized $t$-distribution with $2.5$ degrees of freedom. Both the mean estimates are rescaled by $\sqrt{30}.$ 
                }
                \label{fig:histogram}
\end{figure}

In Section~\ref{sec3}, we develop robust dependence-adjusted multiple testing procedures with solid theoretical guarantees. We use the approximate factor model \eqref{factor.model} with an observable factor and relatively heavy-tailed errors to characterize the dependence structure in high dimensional data. Assuming such a model, we construct robust test statistics based on the Huber estimator with a diverging tuning parameter, denoted by $T_1,\ldots, T_p$, for testing the individual hypotheses. At a prespecified level $0< \alpha <1$, we apply a family of FDP controlling procedures to the dependence-adjusted $P$-values $\{ P_j = 2\Phi(-|T_j|) \}_{j=1}^p$ to decide which null hypotheses are rejected, where $\Phi$ is the standard normal distribution function. To justify the validity of the resulting procedure on FDP control, a delicate analysis of the impact of dependence-adjustment on the distribution of the $P$-values is required. We show that, under mild moment and regularity conditions, the robust multiple testing procedure controls the FDP at any prespecified level asymptotically. Specifically, applying Storey's procedure [\cite{S2002}] to the above $P$-values gives a data-driven rejection threshold $\hat{z}_{{\rm N}}$ such that $H_{0j}$ is rejected whenever $|T_j| \geq \hat{z}_{{\rm N}}$. Let $\FDP(z) = V(z)/ \max\{ 1 , R(z)\}$ be the FDP at threshold $z\geq 0$, where $V(z) = \sum_{j=1}^p 1(|T_j| \geq z, \mu_j=0)$ and $R(z) = \sum_{j=1}^p 1(|T_j| \geq z)$ are the number of false discoveries and the number of total discoveries, respectively. In the ultra-high dimensional setting that $p$ can be as large as $e^{n^c}$ for some $0<c< 1$, we prove that
\bee
	{p \over p_0}\; \FDP(\hat{z}_{{\rm N}})   \to \alpha ~\mbox{ in probability}  \label{FDP.convergence}
\eee
as $(n,p) \to \infty$, where $p_0 = \sum_{j=1}^p 1(\mu_j = 0)$ is the number of true null hypotheses. We also illustrate the usefulness of the robust techniques by contrasting the performances of robust and least squares based inference procedures through synthetic numerical experiments.

Key technical tools in proving \eqref{FDP.convergence} are the Berry-Esseen bound and Cram\'er-type moderate deviation for the marginal statistic $T_j$. These results are built upon the nonasymptotic Bahadur representation \eqref{na.bound}, and may be of independent interest for other statistical applications. For example, \cite{DHJ2011} explore moderate and large deviations of the $t$-statistic in a variety of high dimensional settings.

\subsection{Organization of the paper}

The lay-out of the paper is as follows. In Section~\ref{sec2}, we develop a general finite sample theory for Huber's robust  $M$-estimator from a new perspective where a diverging tuning parameter is involved. In Section~\ref{sec3}, we propose a robust dependence-adjusted multiple testing procedure with rigorous theoretical guarantees. Section~\ref{sec4} consists of numerical studies and real data analysis. The simulation study provides empirical evidence that the proposed robust inference procedure improves performance in the presence of asymmetric and heavy-tailed errors, and maintains efficiency under light-tailed situations.   A discussion is given in Section~\ref{sec5}.
Proofs of the theoretical results in Sections~\ref{sec2}  and \ref{sec3} are provided in the supplemental material [\cite{ZBFL2017}].

\bigskip
\noindent
{\sc Notation.} For a vector $\bu = (u_1,\ldots, u_p)^\T \in \bbr^p$ ($p\geq 2$), we use $\| \bu \| =  ( \sum_{j=1}^p u_j^2 )^{1/2}$ to denote its $\ell_2$-norm. Let $\mathbb{S}^{p-1}=\{ \bu \in \bbr^p : \| \bu \| = 1 \}$ represent the unit sphere in $\bbr^p$. For a matrix $\bA \in \bbr^{p\times p}$, $\| \bA \| = \sup_{\bu \in \mathbb{S}^{p-1} } \| \bA \bu \|$ denotes the spectral norm of $\bA$. For any two sequences $\{ a_n \}$ and $\{ b_n \}$ of positive numbers, denote by $a_n \asymp b_n$ when $c b_n \leq a_n \leq C b_n$ for some absolute constants $C\geq c >0$, denote by $a_n \sim b_n$ if $a_n / b_n \to 1$ as $n\to \infty$. Moreover, we write $a_n = O(b_n)$ if $a_n \leq C b_n$ for some absolute constant $C>0$, write $a_n = o(b_n)$ if $a_n / b_n \to 0$ as $n\to \infty$, and write $a_n = o_\mathbb{P}(b_n)$ if $a_n / b_n \to 0$ in probability as $n\to \infty$. For a set $S$, we use $S^{{\rm c}}$ to denote its complement and $\Card(S)$ for its cardinality. For $x\in \bbr$, denote by $\floor{x}$ the largest integer not greater than $x$ and $\ceil{x}$ the smallest integer not less than $x$. For any two real numbers $a$ and $ b$, we write $a\vee b = \max(a,b)$ and $a\wedge b = \min(a,b)$.

\section{Robust  $M$-estimation: A finite sample theory}
\label{sec2}

Consider a heteroscedastic linear regression model $Y = \mu^* +\bX^\T \bbeta^* + \sigma(\bX) \varepsilon$, from which we observe independent samples $\{(Y_i, \bX_i)\}_{i=1}^n$ satisfying
\bee
 Y_i = \mu^* + \bX_i^\T \bbeta^* + \sigma(\bX_i) \varepsilon_i ,   \ \  i =1,\ldots , n,  \label{linear.model}
\eee
where $\mu^*$ is the intercept, $\bX \in \bbr^d$ is the vector of covariates, $\bbeta^* \in \bbr^d$ is the vector of regression coefficients, $\varepsilon$ is the random error independent of $\bX$, and $\sigma(\cdot):\bbr^d \mapsto (0, \infty)$ is an unknown variance function. We assume both $\bX$ and $\varepsilon$ have zero means. Under this assumption, $\mu^*$ and $\bbeta^*$ together are related to the conditional mean effect of $Y$ given $\bX$, and $\mu^*$ is the unconditional mean of $Y$ that is of independent interest in many applications. For simplicity, we introduce the following notations:
\begin{align}
    \btheta^* = (\mu^*,  \bbeta^{* \T})^\T \in \bbr^{d+1}  , \quad
  \bZ = (1, \bX^\T)^\T  \in \bbr^{d+1}, \quad  \nu =\sigma(\bX) \varepsilon ,  \nn \\
     \mbox{ and }~ \bZ_i = (1,\bX_i^\T)^\T ,  \quad  \nu_{i} = \sigma(\bX_i) \varepsilon_i , \ \   i =1,\ldots, n. \nn
\end{align}

In this section, we study the robust estimator of $\btheta^*$ defined in \eqref{robust.est}. In particular, we show that  it admits an exponential-type deviation bound even for heavy-tailed error distributions. Note that, under the heteroscedastic model \eqref{linear.model}, $\btheta^*$ differs from the median effect of $Y$ conditioning on $\bX$, so that the LAD-based methods are not applicable to estimate $\btheta^*$. Instead, we focus on Huber's robust estimator $\hat{\btheta}$ given in \eqref{robust.est} with a diverging tuning parameter  $\tau=\tau_n$ that balances the approximation error and robustness of the estimator. To begin with, we make the following conditions on the linear model \eqref{linear.model}.

\begin{assumption} \label{cond2.1}
{\rm (i) The random vector $\bX \in \bbr^d$ satisfies $\e(\bX) = \mo$, $\e(\bX \bX^\T) = \bSigma $ for some positive definite matrix $\bSigma$ and $K_0 :=  \| \bSigma^{-1/2} \bX \|_{\psi_2} <\infty$, where $\| \cdot \|_{\psi_2}$ denotes the vector sub-Gaussian norm [\cite{V2012}]. (ii) Independent of $\bX$, the error variable $\varepsilon$ satisfies $\e ( \varepsilon) =0$ and $\e( \varepsilon^2 )  =1$. (iii) $\sigma(\cdot): \bbr^d \mapsto (0,\infty)$ is a positive function and $\sigma^2 := \e\{\sigma^2(\bX) \}$ is finite. }
\end{assumption}

Condition~\ref{cond2.1} allows a family of conditional heteroscedastic models with heavy-tailed error $\varepsilon$. Specifically, it only requires the second moment of $\nu = \sigma(\bX) \varepsilon$ to be finite. Under this condition, our first result, Theorem~\ref{RAthm1}, provides an exponential-type deviation bound and a nonasymptotic Bahadur representation for the robust estimator $\hat{\btheta}=(\hat{\mu}, \hat{\bbeta}^\T)^\T$ defined in \eqref{robust.est}.

\begin{theorem} \label{RAthm1}
Under the linear  model \eqref{linear.model} with Condition~\ref{cond2.1} satisfied, we have for any $w >0$ that, the robust estimator $\hat{\btheta}$ in \eqref{robust.est} with $ \tau = \tau_n = \tau_0  \sqrt{n} (d+1+w)^{-1/2}$ and $\tau_0 \geq \sigma$ satisfies
\begin{align}
	\mathbb{P}\big\{  \| \bS^{1/2}(\hat{\btheta} - \btheta^*) \|  >  a_1     (d+ w)^{1/2} n^{-1/2} \big\} \leq 7 e^{-w} \ \ \mbox{ and }     \label{concentration.MLE} \\
	\mathbb{P}\bigg\{  \bigg\|   \bS^{1/2}(\hat{\btheta} - \btheta^*)   - \frac{1}{  n} \sn \ell'_\tau( \nu_i  ) \bS^{-1/2} \bZ_i \bigg\|  >  a_2  \frac{d+w}{n} \bigg\} \leq 8 e^{-w}   \label{Bahadur.representation}
\end{align}
as long as $n \geq a_3  (d+w)^{3/2}$, where $\bS = \e (\bZ \bZ^\T)$ and $a_1$--$a_3$ are positive constants depending only on $\tau_0, K_0$ and $\| \bS^{-1/2}\overline{\bS} \bS^{-1/2} \|$ with $\overline{\bS} = \e \{ \sigma^2(\bX)\bZ \bZ^\T \}$.
\end{theorem}

An important message of Theorem~\ref{RAthm1} is that, even for heavy-tailed errors with only finite second moment, the robust estimator $\hat{\btheta}$ with properly chosen $\tau$ has sub-Gaussian tails. See inequality \eqref{concentration.MLE}. To some extent, the tuning parameter $\tau$ plays a similar role as the bandwidth in constructing nonparametric estimators. Furthermore, we show in \eqref{Bahadur.representation} that the remainder of the Bahadur representation for $\hat{\btheta}$ exhibits sub-exponential tails. To the best of our knowledge, no nonasymptotic results of this type exist in the literature, and classical asymptotic results can only be used to derive polynomial-type deviation bounds.

Write $W_n := n^{-1/2} \sn \ell'_\tau(\nu_i )$.  As a direct consequence of \eqref{Bahadur.representation}, $\sqrt{n}\, (\hat{\mu} - \mu^*)$ is close to $W_n$ with  probability approaching one exponentially fast. The next result shows that, under higher moment condition on $\nu = \sigma(\bX) \varepsilon$,  $W_n$ has an asymptotic normal distribution with mean 0 and variance $\sigma^2 = \e (\nu^2)$.

\begin{theorem}  \label{RAthm2}
Assume Condition~\ref{cond2.1} holds and $v_\kappa := \e( |\nu |^\kappa )$ is finite for some $\kappa \geq 3$. Then, there exists an absolute constant $C >0$ such that for any $\tau>0$,
\begin{align}
& \sup_{x\in \bbr}  | \mathbb{P}(  \sigma^{-1} W_n \leq x )  - \Phi(x)   |  \nn \\
& \quad \leq  C  \big(   \sigma^{-3} v_3 \, n^{-1/2} + \sigma^{-2} v_\kappa  \, \tau^{2-\kappa}   + \sigma^2 \tau^{-2}  +  \sigma^{-1} v_\kappa \, \tau^{1-\kappa} \sqrt{n}  \big). \nn
\end{align}
In particular, we have
\begin{align}
\sup_{x\in \bbr}  | \mathbb{P}(  \sigma^{-1} W_n \leq x )      - \Phi(x)  | \leq  C  \big(   \sigma^{-3} v_3 \, n^{-1/2} + \sigma^{-2} v_4  \,\tau^{-2}     +  \sigma^{-1} v_4 \, \tau^{-3} \sqrt{n}    \big). \nn
\end{align}
\end{theorem}

Together, Theorems~\ref{RAthm1} and \ref{RAthm2} lead to a Berry-Esseen type bound for $T:= \sqrt{n}  (\hat{\mu}  - \mu^*  ) /\sigma $ for properly chosen $\tau$. In addition, the following theorem gives a Cram\'er-type moderation deviation result for $T$, which quantifies the relative error of the normal approximation.

\begin{theorem}  \label{RAthm3}
Assume Condition~\ref{cond2.1} is met, $\e (| \nu|^3 )<\infty$ and let $\{ w_n \}_{n\geq 1}$ be an arbitrary sequence of positive numbers satisfying $ w_n \to \infty$ and $w_n = o(\sqrt{n})$. Then, the statistic $T$ with $\tau=  \tau_0 \sqrt{n}   (d+w_n)^{-1/2}$ for some constant $\tau_0 \geq \sigma $ satisfies
\begin{align}
	\mathbb{P}\big( | T | \geq z \big) = \big(1+C_{n,z} \big)\mathbb{P}\big(|G| \geq z \big)  \label{fs.GAR.hatT}
\end{align}
uniformly for $0\leq  z =o\{ \min(  \sqrt{w_n}  ,  \sqrt{n}w_n^{-1} ) \}$ as $n\to \infty$, where $G\sim N(0,1)$,
$$
	|C_{n,z} | \leq  C \big\{  \big( \sqrt{\log n} + z \big)^3 n^{-1/2} + (1+z) \big( n^{-3/10}   +  n^{-1/2} w_n \big) +   e^{-w_n} \big\}
$$
and $C>0$ is a constant independent of $n$. In particular, we have
\begin{align}
\sup_{0\leq  z \leq  o\{ \min(  \sqrt{w_n}, \sqrt{n}w_n^{-1} ) \}} \bigg| \frac{\mathbb{P}(| {T}|\geq z)}{  2 -  2\Phi(z)  } - 1 \bigg| \to 0.  \label{md.hatT}
\end{align}
\end{theorem}

\begin{remark}
{\rm From Theorem~\ref{RAthm3} we see that the ratio $\mathbb{P}(|T | \geq z)/\{2 -  2\Phi(z)\}$ is close to 1 for a wide range of nonnegative $z$-values, whose length depends on both the sample size $n$ and $w_n$. In particular, taking $w_n  \asymp n^{1/3}$ gives the widest possible range $[0, o(n^{1/6}))$, which is also optimal for Cram\'er-type moderate deviation results [\cite{L1961}]. In this case, the tuning parameter $\tau = \tau_n$ is of order $n^{1/3}$. }
\end{remark}

\begin{remark}
{\rm  Motivated by an application to large-scale simultaneous inference considered in Section~\ref{sec3}, we only focus on the robust intercept estimator $\hat{\mu}$ of $\mu^*$ in Theorems~\ref{RAthm2} and \ref{RAthm3}. In fact, similar results can be obtained for $\hat{\bbeta}$ or a specific coordinate of $\hat{\bbeta}$ based on the Bahadur representation \eqref{Bahadur.representation}.
}
\end{remark}

\section{Large-scale multiple testing for heavy-tailed dependent data}
\label{sec3}

In this section, we propose and analyze a robust dependence-adjusted procedure for simultaneously testing the means $\mu_1,\ldots, \mu_p$ in model \eqref{factor.model}, based on independent observations from the population vector $\bX$ which exhibits strong dependence and heavy tails.

\subsection{Robust test statistics}
\label{sec3.1}

Suppose we are given independent random samples $\{ (\bX_i, \bff_i ) \}_{i=1}^n$ from model \eqref{factor.model}. We are interested in the simultaneous testing of mean effects \eqref{multiple.test}. A naive approach under normality is to directly use the information $X_{ij} \sim N(\mu_j, \sigma_{jj})$ for the dependent case as was done in the literature, where $\sigma_{jj} = \var(X_j)$.  Such an approach is very natural and popular when the factors are unobservable and focus is on the valid control of FDR, but is inefficient as noted in \cite{FHG2012}. Indeed, if the loading matrix $\bB$ is known and the factors are observed (otherwise, replaced by their estimates), for each $j$, we can construct the marginal test statistic using dependence-adjusted observations $\{ X_{ij} - \bb_j^\T \bff_i \}_{i=1}^n$ from $\mu_j + \sigma_j(\bff) u_j$ for testing the $j$th hypothesis $H_{0j}: \mu_j =0$.

We consider the approximate factor model \eqref{factor.model} and write
\begin{align}
	\bu = (u_1,\ldots, u_p)^\T, \quad  \bnu = ( \nu_1, \ldots, \nu_p)^\T = \bSigma(\bff) \bu  ,  \nn \\
	 \bnu_i = (\nu_{i1}, \ldots, \nu_{ip})^\T =\bSigma(\bff_i) \bu_i , \, i=1,\ldots , n.  \nn
\end{align}
Let $\bSigma_f $ and $\bSigma_\nu = (\sigma_{\nu ,jk})_{1\leq j,k\leq p}$ denote the covariance matrices of $\bff$ and $\bnu$, respectively. Under certain sparsity condition on $\bSigma_\nu$ (see Section~\ref{secT} for an elaboration), $\nu_{1} , \ldots, \nu_{p}$ are weakly dependent random variables with higher signal-to-noise ratios since $\var( \nu_j) = \sigma_{jj} - \| \bSigma_f^{1/2} \bb_j \|^2 < \sigma_{jj}$. Therefore, subtracting common factors out makes the resulting FDP control procedure more efficient and powerful.  It provides an alternative ranking of the significance of hypothesis from the tests based on marginal~statistics.

For each $j=1,\ldots, p$, we have a linear regression model
\bee
	X_{ij} = \mu_j + \bb_j^\T \bff_i + \nu_{ij} , \ \ i=1,\ldots, n.  \label{marginal.model}
\eee
A natural approach is to estimate $\mu_j$ and $\bb_j$ by the method of least squares. However, the least squares method is sensitive to the tails of the error distributions. Also, as noted in \cite{FLW2014}, the LAD-based methods are not applicable in the presence of asymmetric and heteroscedastic errors. Hence, we suggest a robust method that simultaneously estimates $\mu_j$ and $\bb_j$ by solving
\begin{align}
  (\hat{\mu}_j, \hat{\bb}_j^\T)^\T  \in \argmin_{   \mu \in \bbr, \, \bb    \in  \bbr^{d } } \, \sn \ell_\tau( X_{ij} - \mu  -  \bb^\T \bff_i ) ,  \label{robust.mle}
\end{align}
where $\ell_\tau$ is given in \eqref{huber.loss}. By Theorems~\ref{RAthm1} and \ref{RAthm2}, the adaptive Huber estimator $\hat{\mu}_{j }$ follows a normal distribution asymptotically as $n\to \infty$:
\begin{align}  \label{univariate.CLT}
  \sqrt{n}  \big( \hat{\mu}_{j}  - \mu_j  \big)    \xrightarrow {\mathscr{D}} N(0,\sigma_{\nu , jj}) ~\mbox{ with }~  \sigma_{\nu ,jj} = \var( \nu_j )  .
\end{align}

To construct a test statistic for the individual hypothesis $H_{0j}: \mu_j =0$ with pivotal limiting distribution, we need to estimate $\sigma_{\nu, jj}  =   \sigma_{jj}  - \var(\bb_j^\T  \bff ) $. For $ \var(\bb_j^\T  \bff )$, a natural and simple estimator is $\hat{\bb}_j^\T \hat{\bSigma}_f \hat{\bb}_j$, where $\hat{\bSigma}_f : = n^{-1} \sn \bff_i \bff_i^\T $. Let $\hat{\sigma}_{jj}$ and $\hat{\sigma}_{\nu,jj}$ be generic estimators of $\sigma_{jj}$ and $\sigma_{\nu,jj}$, respectively. To simultaneously infer  all the hypotheses of interest, we need the following uniform convergence results
$$
	\max_{1\leq j\leq p} \bigg|  \frac{ \hat{\sigma}_{jj} }{\sigma_{jj}} -1  \bigg| = o_\mathbb{P}(1)   \ \ \mbox{ and } \ \   \max_{1\leq j\leq p} \bigg|  \frac{  \hat{\sigma}_{\nu,jj} }{\sigma_{\nu,jj}} -1  \bigg|  =  o_\mathbb{P}(1).
$$
 For $\sigma_{jj}=\var(X_j)$, it is known that the sample variance $n^{-1}\sn (X_{ij}-\overline{X}_j)^2$ performs poorly when $X_j$ has heavy tails. Based on the recent developments of robust mean estimation for heavy-tailed data [\cite{C2012}, \cite{JL2015}, \cite{FLW2014}], we consider the following two types of robust variance estimators.

\begin{itemize}
\item[1](Adaptive Huber variance estimator). Write $\theta_j = \e(X_j^2)$ so that $\sigma_{jj} =\theta_j   - \mu_j^2$. Construct the adaptive Huber estimator of $\theta_j $ using the squared data, i.e., $\hat{\theta}_{j} = \argmin_{ \theta > 0 } \sn \ell_\gamma( X_{ij}^2 - \theta )$, where $\gamma=\gamma_n$ is a tuning parameter. Then, we compute the adaptive Huber estimator $(\hat{\mu}_j, \hat{\bb}_j^\T)^\T$ given in \eqref{robust.mle}. The variance estimator is then defined by
\bee \label{RA-var}
	 \hat{\sigma}_{\nu,jj}  =
	 \begin{cases}
	\hat{\theta}_{j} -  \hat{\mu}_{j  }^2 - \hat{\bb}_j^\T \hat{\bSigma}_f \hat{\bb}_j   & \, \mbox{ if }  \hat{\theta}_{j}  >  \hat{\mu}_{j  }^2 + \hat{\bb}_j^\T \hat{\bSigma}_f \hat{\bb}_j , \\
	\hat{\theta}_{j}  & \, \mbox{ otherwise.}
	\end{cases}
\eee

\item[2](Median-of-means variance estimator). The median-of-means technique, which dates back to \cite{NY1983}, robustifies the empirical mean by first dividing the given observations into several blocks, computing the sample mean within each block and  then taking the median of these sample means as the final estimator. Although the sample variance cannot be represented in a simple average form, it is a $U$-statistic with a symmetric kernel $h: \bbr^2 \mapsto \bbr$ given by $h(x,y)=  (x-y)^2/2$. The recent work of \cite{JL2015} extends the median-of-means technique to construct $U$-statistics based sub-Gaussian estimators for heavy-tailed data.

Back to the current problem, we aim to estimate $\sigma_{jj}$ based on independent observations $X_{1j}, \ldots, X_{nj}$. Let $V=V_n < n$ be an integer and decompose $n$ as $n=Vm + r$ for some integer $0\leq r<V$. Let $B_1,\ldots, B_V$ be a partition of $\{1,\ldots, n\}$ defined by
\bee
	B_k = \begin{cases}
		\{ (k -1 )m +1, (k -1 )m +2, \ldots, km \} ,  &\ \ \mbox{ if } 1\leq k \leq V-1 , \\
		\{ (V-1)m+1, (V-1)m+2, \ldots, n  \}, & \ \ \mbox{ if }  k  = V.
	\end{cases} \label{blocks.def}
\eee
For each pair $(k,\ell)$ satisfying $1\leq k < \ell \leq V$, define decoupled $U$-statistic $U_{j,k \ell} = (2|B_k||B_\ell |)^{-1} \sum_{i_1 \in B_k} \sum_{i_2 \in B_\ell} (X_{i_1 j} - X_{i_2 j})^2 $. Then, we estimate $\sigma_{jj}$ by the median of $\{ U_{j, k \ell} :  1\leq k< \ell \leq V \}$, i.e., $\wt \sigma_{jj}(V)   \in  \argmin_{u \in \bbr} \sum_{1\leq k< \ell \leq V} |  U_{j,k\ell} -u  |$. As before, we compute the adaptive Huber estimator $(\hat{\mu}_j, \hat{\bb}_j^\T)^\T$. Finally, our robust variance estimators are
\bee
	 \wt{\sigma}_{\nu,jj}  
= \begin{cases}
	\wt \sigma_{jj}(V) - \hat{\bb}_j^\T \hat{\bSigma}_f \hat{\bb}_j   & \, \mbox{ if } \wt \sigma_{jj}(V) >  \hat{\bb}_j^\T \hat{\bSigma}_f \hat{\bb}_j  , \\
	\wt \sigma_{jj}(V) & \, \mbox{ otherwise,}
	\end{cases}	
 . \label{mom-var}
\eee
\end{itemize}

Given robust mean and variance estimators of each type, we construct dependence-adjusted test statistics
\bee
	T_j = \sqrt{n} \,  \hat{\sigma}_{\nu,jj}^{-1/2} \hat{\mu}_j \ \ \mbox{ and } \ \  S_j = \sqrt{n} \,  \wt{\sigma}_{\nu,jj}^{-1/2} \hat{\mu}_j  \ \  \mbox{ for }  j =1 , \ldots , p .  \label{test.stat}
\eee
In fact, as long as the fourth moment $\e ( X_j^4)$ is finite, the estimators $ \hat{\sigma}_{\nu,jj}$ and $ \wt{\sigma}_{\nu,jj}$ given in \eqref{RA-var} and \eqref{mom-var}, respectively, are concentrated around ${\sigma}_{\nu,jj}$ with high probability. In view of \eqref{univariate.CLT}, under the null hypothesis $H_{0j} : \mu_j = 0$, the adjusted test statistics $T_j$ and $S_j$ satisfy that $T_j  \xrightarrow {\mathscr{D}} N(0,1)$ and $S_j \xrightarrow {\mathscr{D}} N(0,1)$ as $n\to \infty$.

\subsection{Dependence-adjusted FDP control procedure}

To conduct multiple testing of \eqref{multiple.test} using the test statistics $T_j$'s, let $z > 0$ be the critical value to be determined such that $H_{0j}$ is rejected whenever $|T_j| \geq z$. The main object of interest in this paper is the false discovery proportion
\bee
	\FDP(z) = \frac{V(z) }{\max\{R(z) , 1 \} } , \ \ z \geq 0 ,  \label{FDP.def}
\eee
where $ V(z) = \sum_{j \in \mathcal{H}_0} 1(|T_j | \geq z)$ is the number of false discoveries, $R(z) =   \sum_{j =1}^p  1(|T_j | \geq z)$ and $\mathcal{H}_0 = \{ j : 1\leq j\leq p, \mu_j = 0 \}$ represents the set of true null hypotheses.  There is substantial interest in controlling the FDP at a prespecified level $0<\alpha<1$ for which the ideal rejection threshold is $z_{{\rm oracle}} = \inf\{ z \geq 0: \FDP(z) \leq \alpha \}$.

The statistical behavior of  $\FDP(z)$ is the center of interest in multiple testing. However, the realization of $V(z)$ for a given experiment is unknown and thus needs to be estimated. When the sample size is large, it is natural to approximate $V(z)$ by its expectation $2 p_0 \Phi(-z)$, where $p_0 = \Card(\mathcal{H}_0)$. In the high dimensional sparse setting, both $p$ and $p_0$ are large and $p_1 = p-p_0 = o(p)$ is relatively small. Therefore, we can use $p$ as a slightly conservative surrogate for $p_0$, so that $\FDP(z)$ can be approximated by
\bee
  \FDP_{{\rm N}}(z) =\frac{ 2p \Phi(-z) }{\max\{R(z) , 1 \} } , \ \ z \geq 0 . \label{AFDP.def}
\eee
We will prove in Theorem~\ref{thm1} that under mild conditions, $\FDP_{{\rm N}}(z)$ provides a consistent estimate of $\FDP(z)$ uniformly in $0\leq z \leq \Phi^{-1}(1-m_p/(2p))$ for any sequence of positive number $m_p\leq 2p$ satisfying $m_p\to \infty$.

In the non-sparse case where $ \pi_0 = p_0 / p$ is bounded away from 0 and 1 as $p \to \infty$, $\FDP_{{\rm N}}$ given in \eqref{AFDP.def} tends to overestimate the true FDP. Therefore, we need to estimate the proportion $\pi_0$,  which has been studied by \cite{ETST2001}, \cite{S2002}, \cite{GW2004}, \cite{LL2005} and \cite{MR2006}, among others. For simplicity, we focus on Storey's approach. Let $\{ P_j = 2\Phi(-|T_j|)\}_{j=1}^p$ be the approximate $P$-values. For a predetermined $\lambda \in [0,1)$, \cite{S2002} suggests the following conservative estimate of $\pi_0$:
\bee
	\hat{\pi}_0(\lambda) = \frac{1}{(1-\lambda) p} \sum_{j=1}^p 1(P_j > \lambda).
\eee
The intuition of such an estimator is as follows. Since most of the large $P$-values correspond to the null and thus are uniformly distributed, for a sufficiently large $\lambda$, we expect about $(1-\lambda) \pi_0$ of the $P$-values to lie in $(\lambda , 1]$. Hence, the proportion of $P$-values that exceed $\lambda$, $p^{-1}\sum_{j=1}^p 1(P_j > \lambda)$, should be close to $(1-\lambda) \pi_0$.  This gives rise to Storey's procedure.

Incorporating such an estimate of $\pi_0$, we obtain a modified estimate of $\FDP(z)$ by
\bee
	\FDP_{{\rm N},\lambda}(z) =\frac{ 2  p \, \hat{\pi}_0(\lambda)  \Phi(-z) }{\max\{R(z) , 1 \}  } , \ \ z \geq 0 . \label{mFDP.def}
\eee
In view of \eqref{AFDP.def}--\eqref{mFDP.def} and the fact $\hat{\pi}_0(0) =1$, we have $\FDP_{{\rm N},0}(z) = \FDP_{{\rm N}}(z)$.

By replacing the unknown quantity $ \FDP(z)$ by $ \FDP_{{\rm N} , \lambda}(z)$ given in \eqref{mFDP.def} for some $\lambda \in [0,1)$, we reject $H_{0j}$ whenever $|T_j| \geq \hat z_{{\rm N},\lambda}$, where
\begin{align}
 \hat z_{{\rm N},\lambda} = \inf \big\{  z \geq 0: \FDP_{{\rm N},\lambda}(z) \leq \alpha \big\} .  \label{zBH.def}
\end{align}
By Lemmas~1 and 2 in \cite{STS2004}, this procedure is equivalent to a variant of the seminal Benjamini-Hochberg (B-H) procedure [\cite{BH1995}] for selecting $\mathcal{S}=\{ j : 1\leq j\leq p,  P_j \leq P_{(k_p(\lambda)) }   \}$ based on the $P$-values $P_j  = 2\Phi(-|T_j|)$, where $k_p(\lambda) := \max\{ j : 1\leq j\leq p,  P_{(j)} \leq \frac{\alpha j }{ \hat{\pi}_0(\lambda) p } \}$ and $P_{(1)} \leq \cdots \leq P_{(p)}$ are the ordered $P$-values. Theorem~\ref{thm2} shows that under weak moment conditions, the FDP of this dependence-adjusted procedure with $\lambda=0$ converges to  $\alpha$ in the ultra-high dimensional sparse setting.

Note that $\FDP_{{\rm N} ,0}(z)$ is the most conservatively biased estimate of $\FDP(z)$ among all $\lambda\in [0,1)$ using normal calibration. The statistical power of the corresponding procedure can be compromised if $\pi_0$ is much smaller than 1. In general, the procedure requires the choice of a tuning parameter $\lambda$ in the estimate $\hat{\pi}(\lambda)$, which leads to an inherent bias-variance trade-off. We refer to Section~9 in \cite{S2002} and Section~6 in \cite{STS2004} for two data-driven methods for automatically choosing $\lambda$.

\subsection{Bootstrap calibration}

When the sample size is large, it is suitable to use the normal distribution for calibration. Here we consider bootstrap calibration, which has been widely used due to its good numerical performance when the sample size is relatively small. In particular, we focus on the weighted bootstrap [\cite{BB1995}], which perturbs the objective function of an  $M$-estimator with i.i.d. weights. Let $W$ be a random variable with unit mean and variance, i.e., $\e (W)=1$ and $\var(W)=1$. Independent of $\bX_1,\ldots, \bX_n$, generate i.i.d. random weights $\{ W_{ij,b}, 1\leq i\leq n, 1\leq j\leq p, 1\leq b\leq B \}$ from $W$, where $B$ is the number of bootstrap replications. For each $j$, the bootstrap counterparts of $(\hat{\mu}_j , \hat{\bb}_j^\T)^\T$ given in \eqref{robust.mle} are defined by
\bee
	 (\hat{\mu}^*_{j,b}, (\hat{\bb}_{j,b}^{*})^\T)^\T  \in \argmin_{   \mu \in \bbr, \, \bb  \in  \bbr^{d} } \, \sn W_{ij,b} \,\ell_\tau( X_{ij} - \mu  -  \bb^\T \bff_i ) , \ \ b=1,\ldots, B. \nn
\eee
For $j=1,\ldots,  p$, define empirical tail distributions
$$
	G^*_{j,B}(z) =  \frac{1}{ B+1}  \sum_{b=1}^B 1\big( |\hat{\mu}^*_{j,b} - \hat{\mu}_j |  \geq z \big) , \ \ z \geq 0.
$$
The bootstrap $P$-values are thus given by $\{ P^{*}_{j} =  G^*_{j,B}(|\hat{\mu}_j|) \}_{j=1}^p$, to which either the B-H procedure or Storey's procedure can be applied. For the former, we reject $H_{0j}$ whenever $ P^{*}_{j}  \leq P^*_{(k_p^*)}$, where $k_p^* = \max\{ j : 1\leq j\leq p,  P^*_{(j)} \leq j\alpha/p\}$ for a predetermined $0<\alpha<1$ and $P^*_{(1)} \leq \cdots \leq P^*_{(p)}$ are the ordered bootstrap $P$-values. For the distribution of the bootstrap weights, it is common to choose  $W \sim 2{\rm Bernoulli}(0.5)$, $W\sim {\rm exp}(1)$ or $W \sim N(1,1)$ in practice. Using nonnegative random weights has the advantage that the weighted objective function is guaranteed to be convex.

\begin{remark}
{\rm
Weighted bootstrap procedure serves as an alternative method to normal calibration in multiple testing. We refer to \cite{SZ2015} and \cite{Z2016} for the most advanced recent results of weighted bootstrap and a comprehensive literature review.  We leave the theoretical guarantee of this procedure for future research.
}
\end{remark}

\subsection{Theoretical properties}
\label{secT}

First, we impose some conditions on the distribution of $\bX$ and the tuning parameters~$\tau$~and~$\gamma$ that are used in the robust regression and robust estimation of the second moment.

\begin{itemize}
\item[\textbf{(C1).}] (i) $\bX \in \bbr^p$ follows the model \eqref{factor.model} with $\bff$ and $\bu$ being independent; (ii) $\e  ( u_j )  = 0$, $\e ( u_j^2 ) =1$ for $j=1,\ldots,p$, and $c_v \leq \min_{1\leq j\leq p} \sigma_{\nu,jj} \leq  \max_{1\leq j\leq p}  \e (\nu_j^4)  \leq C_v$ for some $C_\nu > c_\nu  >0$; (iii) $\e ( \bff)  = \mo$, $\bSigma_f=\cov(\bff)$ is positive definite and $  \|   \bSigma_f^{-1/2} \bff \|_{\psi_2} \leq C_f$ for some $C_f>0$.

\item[\textbf{(C2).}] $(\tau, \gamma) = (\tau_n, \gamma_n)$ satisfies $\tau = \tau_0 \sqrt{n} w_n^{-1/2}$ and $\gamma=\gamma_0 \sqrt{n} w_n^{-1/2}$ for some constants $\tau_0 \geq \max_{1\leq j\leq p} \sigma_{\nu, jj}^{1/2}$ and $\gamma_0 \geq \max_{1\leq j\leq p} {\rm var}^{1/2}(X_j^2)$, where the sequence $ w_n  $ is such that $w_n \to \infty$ and $w_n = o(\sqrt{n})$.

\end{itemize}

In addition, we need the following assumptions on the covariance structure of $\bnu = \bSigma(\bff) \bu$, and the number and magnitudes of the signals (nonzero coordinates of $\bmu$). Let $\bR_\nu = ( \rho_{\nu,jk})_{1\leq j,k\leq p}$ be the correlation matrix of $\bnu$, where by the independence of $\bff$ and $\bu$, $ \rho_{\nu,jk} =  \frac{  \e \{\sigma_j(\bff)\sigma_k(\bff)\} }{ \sqrt{ \e \sigma_j^2(\bff) \e \sigma_k^2(\bff) } }  \times \corr(u_j, u_k)$.
\begin{itemize}

\item[\textbf{(C3).}] $\max_{1\leq j<k\leq p} | \rho_{\nu,jk}| \leq \rho$ and
$$
	s_p := \max_{1\leq j\leq p} \sum_{k=1}^p 1\big\{ | \rho_{\nu,jk} | > (\log p)^{-2-\kappa} \big\} = O(p^r)
$$
for some $0<\rho<1$, $\kappa>0$ and $0< r < (1-\rho)/(1+\rho)$. As $n,p \to \infty$, $p_0/p \to \pi_0 \in (0,1]$, $\log p= o(n^{1/5})$ and $w_n \asymp n^{1/5}$, where $w_n$ is as in Condition~(C2).

\item[\textbf{(C4).}] ${\rm Card}  \big\{  j: 1\leq j\leq p , \sigma_{\nu,jj}^{-1/2} |\mu_j|    \geq  \lambda \sqrt{(\log p) / n}  \big\} \to \infty$ as $n, p\to \infty$ for some $\lambda>2\sqrt{2}$.

\end{itemize}

Condition (C3) allows weak dependence among $\nu_1,\ldots,\nu_p$ in the sense that each variable is moderately correlated with $s_p$ other variables and weakly correlated with the remaining ones. The technical assumption (C4) imposes a constraint on the number of significant true alternatives, which is slightly stronger than $p_1 \to \infty$. According to Proposition~2.1 in \cite{LS2014}, this condition is nearly optimal for the results on FDP control in the sense that if $p_1$ is fixed, the B-H method fails to control the FDP at any level $0<\beta<1$ with overwhelming probability even if the true $P$-values were known.

For robust test statistics $T_j$'s  given in \eqref{test.stat}, define the null distribution $F_{j,n}(x)=\mathbb{P}(T_j \leq x | H_{0j})$ and the corresponding $P$-value $P_{j}^{{\rm true}}=F_{j,n}(-|T_j|)+1-F_{j,n}(|T_j|)$. In practice, we use $P_j=2\Phi(-|T_j|)$ to estimate the true (unknown) $P$-values. A natural question is on how fast $p$ can  diverge with $n$ so as to maintain valid simultaneous inference. This problem has been studied in \cite{FHY2007}, \cite{KM2007} and \cite{LS2010}. There it is shown that the simple consistency $\max_{1\leq j\leq p} |P_j - P_{j}^{{\rm true}} | = o(1)$ is not enough, and the level of accuracy required must increase with $n$. More precisely, to secure a valid inference, we require
\bee
	 \max_{1\leq j\leq p}\bigg| \frac{ P_j }{ P_{j}^{{\rm true}} } -  1  \bigg| 1\{ \mathcal{S}_j \}   = o(1) \ \ \mbox{ as } n\to \infty,  \label{unif.consistency}
\eee
where $\mathcal{S}_j = \{  P_{j}^{{\rm true}} > \alpha/p \}$, $j=1,\ldots,p$.

\begin{theorem} \label{thm0}
Assume that Conditions~(C1) and (C2) hold and $\log p = o \{ \min(   w_n , n w_n^{-2}) \}$. Then \eqref{unif.consistency} holds.
\end{theorem}

Theorem~\ref{thm0} shows that, to ensure the accuracy of the normal distribution calibration,   the number of simultaneous tests can be as large as $\exp\{o(n^{1/3})\}$, when taking $w_n \asymp n^{1/3}$. We are also interested in estimating FDP in the high dimensional sparse setting, that is, $p$ is large, but the number of $\mu_j \neq 0$ is relatively small. The following result indicates that $\FDP_{{\rm N}}(z)$ given in \eqref{AFDP.def} provides a consistent estimator of the realized FDP in a uniform sense.

\begin{theorem} \label{thm1}
Assume that Conditions (C1)--(C3) hold. Then, for any sequence of positive numbers $m_p \leq p$ satisfying $m_p \to \infty$, we have as $(n,p) \to \infty$,
\bee
	\max_{0\leq z\leq \Phi^{-1}(1-m_p/(2p) )} \bigg| \frac{\FDP(z)}{\FDP_{\rm N}(z)} - 1 \bigg| \to 0 ~\mbox{ in probability.} \label{FDP.AFDP}
\eee
\end{theorem}

Further, Theorem~\ref{thm2} shows that the proposed robust dependence-adjusted inference procedure controls the FDP at a given level $\alpha$ asymptotically with $P$-values estimated from the standard normal distribution.

\begin{theorem}  \label{thm2}
Assume that Conditions (C1)--(C4) hold. Then, for any prespecified $0<\alpha <1$,
\bee
  (p_0 / p)^{-1}  \FDP( \hat z_{{\rm N},0})  \to  \alpha ~\mbox{ in probability}  \label{oracle.FDP.converge}
\eee
as $(n,p) \to \infty$, where $\hat z_{{\rm N},0}$ is defined in \eqref{zBH.def}.
\end{theorem}

The constraint on $p$, as a function of $n$, imposed in Theorems~\ref{thm1} and \ref{thm2} can be relaxed in a strict factor model with independent idiosyncratic errors.

\begin{itemize}
\item[\textbf{(C5).}] $\nu_1, \ldots, \nu_p$ in model \eqref{factor.model} are independent. As $n,p \to \infty$, $p_0/p \to \pi_0 \in ( 0, 1]$, $\log p= o(w_n)$ and $w_n = O(n^{1/3})$, where $w_n$ is as in Condition~(C2).
\end{itemize}

\begin{theorem} \label{thm3}
Assume that Conditions (C1), (C2), (C4) and (C5) hold. Then, for any prespecified $0<\alpha <1$, $(p_0 / p)^{-1}  \FDP( \hat z_{{\rm N},0})  \to  \alpha$ in probability as $(n,p) \to \infty$, where $\hat z_{{\rm N},0}$ is defined in \eqref{zBH.def}.
\end{theorem}

Theorems~\ref{thm1}--\ref{thm3} provide theoretical guarantees on the FDP control for the B-H procedure with dependence-adjusted $P$-values $P_j=2\Phi(-|T_j|)$, $j=1,\ldots, p$. A similar approach can be defined by using the median-of-means approach, namely, replacing $T_j$'s with $S_j$'s in the definition of $\FDP_{{\rm N}}(z)$ in \eqref{AFDP.def}, which is equivalent to the B-H procedure with $P$-values $Q_j=2\Phi(-| S_j |)$, $j=1,\ldots, p$. Under similar conditions, the theoretical results on the FDP control remain valid.

\begin{theorem} \label{thm4}
Let $ \FDP(z)$ and $\hat z_{{\rm N} ,0}$ be defined in \eqref{FDP.def} and \eqref{zBH.def} with $T_j$'s replaced by $S_j$'s, and let $V=V_n $ in \eqref{blocks.def} satisfy $V \asymp w_n$ for $w_n$ as in Condition~(C2). Moreover, let $\tau=\tau_n$ be as in Condition~(C2).
\begin{itemize}
\item[(i).] Under Conditions~(C1), (C3) and (C4), \eqref{oracle.FDP.converge} holds for any prespecified $0<\alpha <1$.

\item[(ii).] Under Conditions~(C1), (C4) and (C5), \eqref{oracle.FDP.converge} holds for any prespecified $0<\alpha <1$.
\end{itemize}
\end{theorem}

\section{Numerical study}
\label{sec4}

\subsection{Implementation}

To implement the proposed procedure, we solve the convex program \eqref{robust.mle} by using our own implementation in Matlab of the traditional method of scoring, which is an iterative method starting at an initial estimate $\hat{\btheta}^0 \in \bbr^{K+1}$. Here, we take $\hat{\btheta}^0 = \mo$; using the current estimate $\hat{\btheta}^t$ at iteration $t=0,1,2,\ldots$, we update the estimate by the Newton-Raphson step:
$$
	\hat{\btheta}^{t+1} = \hat{\btheta}^t +   \bigg\{ \frac{1}{n} \sn \ell''_\tau(\bz_i^t )  \bigg\}^{-1} (\mathbb{G}^\T \mathbb{G} )^{-1} \mathbb{G}^\T (\ell'_\tau(\bz_1^t), \ldots, \ell'_\tau(\bz_n^t))^\T ,
$$
where $\bz^t = (X_{1j}, \ldots, X_{nj})^\T - \mathbb{G}\hat{\btheta}^t$ and $\mathbb{G}=( {\bm g}_{1}, \ldots,  {\bm g}_{ n})^\T \in \bbr^{n\times (K+1)}$ with ${\bm g}_{i} =(1, \bff_i^\T)^\T$.

For each $1\leq j \leq p$, we apply the above algorithm with $\tau =\tau_j := c \,\hat{\sigma}_j \sqrt{n/\log(np)}$ to obtain $(\hat{\mu}_j, \hat \bb_j^\T)^\T$, where $\hat{\sigma}_j^2$ denotes the sample variance of the fitted residuals using OLS and $c >0$ is a control parameter. We take $c = 2$ in all the simulations reported below. In practice, we can use a cross-validation procedure to pick $c$ from only a few candidates, say $\{0.5, 1, 2 \}$.


\subsection{Simulations via a synthetic factor model}
In this section, we perform Monte Carlo simulations to illustrate the performance of the robust test statistic under approximate factor models with general errors. Consider the Fama-French three factor model:
\bee
	X_{ij} = \mu_j + \bb_j^\T \bff_i +  u_{ij} , \ \ i=1,\ldots, n ,
\eee
where $\bu_i = (u_{i1}, \ldots, u_{ip})^\T$ are i.i.d. copies of $\bu = (u_1,\ldots, u_p)^\T$. We simulate $\{\bb_j\}_{j=1}^p$ and $\{\bff_i\}_{ i =1}^n$ independently from $N_3(\bmu_B,\bSigma_B)$ and $N_3( \mo ,\bSigma_f)$, respectively. To make the model more realistic, parameters are calibrated from the daily returns of S\&P 500's top 100 constituents (chosen by market cap), for the period July 1st, 2008 to June 29th, 2012. 

To generate dependent errors, we set $\bSigma_u = \cov(\bu)$ to be a block diagonal matrix where each block is  four-by-four correlation matrix with equal off-diagonal entries generated from Uniform$[0,0.5]$.  The hypothesis testing is carried out under the alternative: $\mu_j= \mu$ for $1\leq j\leq \pi_1p$ and $\mu_j=0$ otherwise. In the simulations reported here, the ambient dimension $p = 2000$, the proportion of true alternatives $\pi_1=0.25$  and the sample size $n$ takes values in $\{80, 120\}$.  For simplicity, we set $\lambda = 0.5$ in our procedure and use the Matlab package \texttt{mafdr} to compute the estimate $\hat{\pi}_0(\lambda)$ of $\pi_0=1-\pi_1$.  
For each test, the empirical false discovery rate (FDR) is calculated based on $500$ replications with FDR level $\alpha$ taking values in $\{ 5\%, 10\%, 20\% \}$. The errors $\{\bu_i\}_{i=1}^n$ are generated independently from the following distributions:
\begin{itemize}
\item \textbf{Model 1}.  $\bu \sim N(\mo, \bSigma_u)$: Centered normal random errors with covariance matrix $\bSigma_u$;
\item \textbf{Model 2}. $\bu \sim (1/\sqrt{5})\, t_{2.5}(\mo, \bSigma_u)$: Symmetric and heavy-tailed errors following a multivariate $t$-distribution with degrees of freedom 2.5 and covariance matrix $\bSigma_u$;
\item \textbf{Model 3}. $\bu = 0.5 \bu_{{\rm N}} + 0.5 (\bu_{{\rm LN}} - \e \bu_{{\rm LN}} )$, where $\bu_{{\rm N}} \sim N(\mo, \bSigma_u)$ and $\bu_{{\rm LN}} \sim \exp\{ N(\mo, \bSigma_u) \}$ is independent of $\bu_{{\rm N}}$. This model admits asymmetric and heavy-tailed errors;

\item \textbf{Model 4}. $\bu = 0.25 \bu_{t} + 0.75 ( \bu_{{\rm W}} - \e \bu_{{\rm W}}  )$, where $\bu_t \sim t_{4}(\mo, \bSigma_u)$ and the $p$ coordinates of $\bu_{{\rm W}}$ are i.i.d. random variables following the Weibull distribution with shape parameter $0.75$ and scale parameter $0.75$.

\end{itemize}

The proposed \uline{R}obust \uline{D}ependence-\uline{A}djusted (RD-A) testing procedure is compared with the \uline{O}rdinary \uline{D}ependence-\uline{A}djusted (OD-A) procedure that uses OLS to estimate unknown parameters in the factor model, and also with the naive procedure where we directly perform multiple marginal $t$-tests ignoring the common factors. We use RD-A$_{{\rm N}}$ and RD-A$_{{\rm B}}$ to denote the RD-A procedure with normal and bootstrap calibration. The number of bootstrap replications is set to be $B=2000$. The signal strength $\mu$ is taken to be $\sqrt{2(\log p)/n}$ for Models 1, 2 and 4, and $\sqrt{3(\log p)/n}$ for Model 3. Define the false negative rate ${\rm FNR} =\e \{ T/( p -  R)\}$,   where $T$ is the number of falsely accepted null hypotheses and $R$ is the number of discoveries. The true positive rate (TPR) is defined as the average ratio between the number of correct rejections and $p_1 = \pi_1 p$. Empirical FDR, FNR and TPR for the RD-A$_{\rm B}$, RD-A$_{{\rm N}}$, OD-A and naive procedures under different scenarios are presented in Tables \ref{tab1} and \ref{tab2}. To save space, we leave the numerical comparison between the RD-A$_{{\rm N}}$ and OD-A procedures under some additional models in the supplementary material [\cite{ZBFL2017}], along with comparisons across a range of sample sizes and signal strengths.

\begin{table}[tbh]
{\begin{tabular}{ccrrrrrr}
\hline \vspace{-0.25cm} \\
&   &  \multicolumn{6}{c}{Normal}      \\
 &  & \multicolumn{3}{c}{$n=80$} & \multicolumn{3}{c}{$n=120$}     \\
  \cline{3-5} \cline{6-8}  \vspace{-0.2cm}   \\
   &  & $\alpha=$5\% & 10\% & 20\%  & 5\%  &  10\%   & 20\%	  \\
 \hline  \vspace{-0.2cm}  \\
\multirow{3}{*}{FDR} & RD-A$_{{\rm B}}$ & 3.66\%  &  7.79\%  &  16.64\% & 4.10\% & 8.45\%  & 17.67\%   \\
& RD-A$_{{\rm N}}$ & 6.22\%  &  11.61\%  &  21.86\% & 5.69\% & 10.93\%  & 21.05\%   \\
 & OD-A &  7.51\%  &  13.70\%  &  24.92\% & 6.49\%  & 12.24\% & 23.01\%    \\
&Naive& 8.35\%  & 13.35\%  &  19.40\% & 7.61\%  & 10.97\%  & 17.67\%   \\
\hline \vspace{-0.2cm} \\
\multirow{3}{*}{FNR} &  RD-A$_{{\rm B}}$  &  3.87\% & 2.17\% & 0.99\% & 3.12\% & 1.72\% & 0.79\%    \\
& RD-A$_{{\rm N}}$ & 2.65\%  &  1.49\%  &  0.68\% & 2.44\% & 1.36\%  & 0.62\%   \\
 & OD-A &  2.16\% & 1.21\% & 0.55\% & 2.13\% & 1.19\%  & 0.53\%     \\
&Naive&  20.10\% & 19.29\% & 18.61\% & 20.25\% & 19.17\%  & 18.26\% \\
\hline \vspace{-0.2cm} \\
\multirow{3}{*}{TPR} &  RD-A$_{{\rm B}}$ & 88.05\% & 93.50\% & 97.18\% & 90.45\%  & 94.89\%  & 97.79\% \\
 & RD-A$_{{\rm N}}$ & 91.99\%  &  95.64\%  &  98.13\% & 92.64\% & 96.03\%  & 98.30\%   \\
& OD-A &  93.53\% & 96.50\% & 98.53\% & 93.59\%  & 96.55\% & 98.57\%   \\
&Naive& 22.71\% & 28.80\% & 37.63\% & 21.96\%  & 28.20\% & 36.79\%   \\
\hline \vspace{-0.25cm} \\
&   &  \multicolumn{6}{c}{Student's $t$}       \\
 &  & \multicolumn{3}{c}{$n=80$} & \multicolumn{3}{c}{$n=120$}    \\
  \cline{3-5}\cline{6-8}    \\
   & & $\alpha=$5\% & 10\% & 20\%  & 5\%  &  10\%   & 20\%	  \\
 \hline  \vspace{-0.2cm}  \\
\multirow{3}{*}{FDR} &  RD-A$_{{\rm B}}$ & 2.98\%  & 6.76\%  &  15.36\% &  3.69\% &  7.89\% & 17.01\%    \\
 & RD-A$_{{\rm N}}$ &4.22\%  & 8.72\%  &  17.99\%  & 4.26\% & 8.83\%  & 18.44\%   \\
& OD-A &  6.43\%  &  12.72\%  &  25.17\% & 5.77\%  & 11.64\% & 23.63\%       \\
&Naive &  10.05\%  &  13.36\%  & 19.45\% &  7.52\%  & 11.07\% & 17.50\%   \\
\hline \vspace{-0.2cm} \\
\multirow{3}{*}{FNR} & RD-A$_{{\rm B}}$  & 1.97\% & 1.31\% & 0.80\% &  1.58\% & 1.07\% & 0.68\%  \\
 & RD-A$_{{\rm N}}$ & 2.17\%  & 1.51\%  &  0.97\% & 2.00\% & 1.40\%  & 0.91\%   \\
& OD-A & 1.93\% & 1.39\% & 1.00\% & 1.86\% & 1.34\%  & 0.93\%    \\
&Naive&  19.75\% & 19.12\% & 18.71\% & 19.93\% & 18.99\% & 18.46\%   \\
\hline \vspace{-0.2cm} \\
\multirow{3}{*}{TPR} &  RD-A$_{{\rm B}}$  &  93.79\% & 95.95\% & 97.60\% & 95.03\%  & 96.72\%  & 97.98\%  \\
 & RD-A$_{{\rm N}}$ & 93.04\%  &  95.23\%  &  97.03\% &93.55\% & 95.54\%  & 97.18\%   \\
& OD-A & 93.91\% & 95.74\% & 97.47\% & 94.09\%  & 95.86\% & 97.43\%      \\
&Naive& 25.45\% & 31.95\% & 41.39\% & 23.63\%  & 29.70\% & 38.81\%  \\\hline
\end{tabular}
}
\vspace{0.2cm}
\caption{Empirical FDR, FNR and TPR based on a factor model with dependent errors following a normal distribution (Model 1) and a $t$-distribution (Model 2).}
\label{tab1}
\end{table}

\begin{table}[tbh]\centering
{\begin{tabular}{ccrrrrrr}
\hline \vspace{-0.25cm} \\
&   &  \multicolumn{6}{c}{Mixture Normal/Lognormal}      \\
 &  & \multicolumn{3}{c}{$n=80$} & \multicolumn{3}{c}{$n=120$}     \\
  \cline{3-5} \cline{6-8}  \vspace{-0.2cm}   \\
   &  & $\alpha=$5\% & 10\% & 20\%  & 5\%  &  10\%   & 20\%	  \\
 \hline  \vspace{-0.2cm}  \\
\multirow{3}{*}{FDR} & RD-A$_{{\rm B}}$ & 8.40\%  &  13.33\%  &  22.04\% & 8.02\% & 12.98\%  & 21.99\%   \\
&RD-A$_{{\rm N}}$ & 10.29\%  &  16.02\%  &  25.87\% & 9.16\% & 14.65\%  & 24.61\%   \\
 & OD-A &  12.18\%  &  18.46\%  &  29.24\% &  10.28\%  & 16.17\% & 26.79\%    \\
&Naive& 7.79\%  & 12.18\%  &  18.49\% & 7.99\%  & 11.77\%  & 17.93\%   \\
\hline \vspace{-0.2cm} \\
\multirow{3}{*}{FNR} &  RD-A$_{{\rm B}}$  &  0.32\% & 0.09\% & 0.02\% & 0.29\% & 0.11\% & 0.03\%    \\
&RD-A$_{{\rm N}}$ & 0.57\%  &  0.25\%  & 0.07\% & 0.56\% & 0.23\%  & 0.06\%   \\
 & OD-A &  0.26\% & 0.08\% & 0.02\% & 0.29\% & 0.10\%  & 0.03\%     \\
&Naive&  18.66\% & 16.99\% & 15.26\% & 18.51\% & 16.67\%  & 14.78\% \\
\hline \vspace{-0.2cm} \\
\multirow{3}{*}{TPR} &  RD-A$_{{\rm B}}$ & 99.07\% & 99.73\% & 99.95\% & 99.15\%  & 99.70\%  & 99.92\%  \\
& RD-A$_{{\rm N}}$ & 98.36\%  &  99.31\%  &  99.81\% & 98.38\% & 99.36\%  & 99.82\%   \\
& OD-A & 99.26\% & 99.76\% & 99.96\% & 99.16\%  & 99.72\% & 99.93\%   \\
&Naive& 28.39\% & 37.37\% & 50.14\% & 29.72\%  & 39.34\% & 51.95\%   \\
\hline \vspace{-0.25cm} \\
&   &  \multicolumn{6}{c}{Mixture Student's $t$/Weibull }       \\
 &  & \multicolumn{3}{c}{$n=80$} & \multicolumn{3}{c}{$n=120$}    \\
  \cline{3-5}\cline{6-8}    \\
   & & $\alpha=$5\% & 10\% & 20\%  & 5\%  &  10\%   & 20\%	  \\
 \hline  \vspace{-0.2cm}  \\
\multirow{3}{*}{FDR} &  RD-A$_{{\rm B}}$ & 7.51\%  &  12.05\%  &  20.43\% & 7.03\% & 11.66\% & 20.49\%    \\
 &RD-A$_{{\rm N}}$ & 8.91\%  &  14.34\%  &  24.11\% & 7.85\% & 13.17\%  & 23.11\%   \\
& OD-A &  11.00\%  &  17.11\%  &  27.67\% &  9.09\%  & 14.80\% & 25.37\%       \\
&Naive& 9.33\%  & 13.48\%  &  20.30\% & 8.14\%  & 12.52\% & 18.85\%   \\
\hline \vspace{-0.2cm} \\
\multirow{3}{*}{FNR} & RD-A$_{{\rm B}}$  &  0.62\% & 0.23\% & 0.06\% & 0.62\% & 0.25\% & 0.08\%  \\
& RD-A$_{{\rm N}}$ & 0.62\%  &  0.24\%  &  0.06\% &0.60\% & 0.24\%  & 0.07\%   \\
& OD-A & 0.30\% & 0.11\% & 0.03\% & 0.42\% & 0.16\%  & 0.05\%    \\
&Naive& 19.91\% & 18.78\% & 18.08\% & 20.71\% & 20.02\% & 19.36\%   \\
\hline \vspace{-0.2cm} \\
\multirow{3}{*}{TPR} &  RD-A$_{{\rm B}}$  &  98.16\% & 99.34\% & 99.83\% & 98.16\%  & 99.29\%  & 99.78\%  \\
&RD-A$_{{\rm N}}$ & 98.18\%  &  99.31\%  &  99.82\% & 98.22\% & 99.30\%  & 99.80\%   \\
 & OD-A &99.14\% & 99.69\% & 99.92\% & 98.78\%  & 99.54\% & 99.86\%      \\
&Naive& 23.41\% & 30.24\% & 39.65\% & 20.49\%  & 25.90\% & 34.07\%  \\\hline
\end{tabular}
\vspace{0.2cm}
\caption{Empirical FDR, FNR and TPR based on a factor model with dependent errors following a mixture normal/lognormal distribution (Model 3) and a mixture Student's $t$/Weibull distribution (Model 4).  }
\label{tab2}}
\end{table}

For weakly dependent errors following the normal distribution and $t$-distribution, Table \ref{tab1} shows that the RD-A procedure consistently outperforms the OD-A method, in the sense that the RD-A method provides a much better control of the FDR at the expense of slight compromises of the FNR and TPR. Should the FDR being controlled at the same level, the robust method will be more powerful.  In Table \ref{tab2}, when the errors are both asymmetric and heavy-tailed, the RD-A procedure has the biggest advantage in that it significantly outperforms the OD-A on controlling the FDR at all levels while maintaining low FNR and high TPR. Together, these results show that the RD-A procedure is indeed robust to outliers and does not lose efficiency when the errors are symmetric and light-tailed. In terms of controlling FDR, Models 3 and 4 present more challenges than Models 1 and 2 due to being both heavy-tailed and asymmetric. In Table \ref{tab2} we see that although both the RD-A and OD-A methods achieve near-perfect power, the empirical FDR is higher than the desired level across all settings and much more higher for OD-A. Hence we compare the FDR of the RD-A and OD-A methods for various sample sizes in Figure \ref{fig:model4}. We see that the empirical FDR decreases with increase in sample size, while consistently outperforming the OD-A procedure. The difference between the two methods is greater for lower sample sizes, reinforcing the usefulness of our method for high dimensional heavy-tailed data with moderate sample sizes.

The naive procedure suffers from a significant loss in FNR and TPR. The reasons are twofold: (a) the naive procedure ignores the actual dependency structure among the variables; (b) the signal-to-noise ratio of $H_{0j}: \mu_j = 0$ for the naive procedure is $\sigma_{jj}^{-1/2} |\mu_j|$, which can be much smaller than $\sigma_{\nu,jj}^{-1/2} |\mu_j|$ for the dependence-adjusted procedure.

\begin{figure}        \centering
                \includegraphics[width=.49\textwidth]{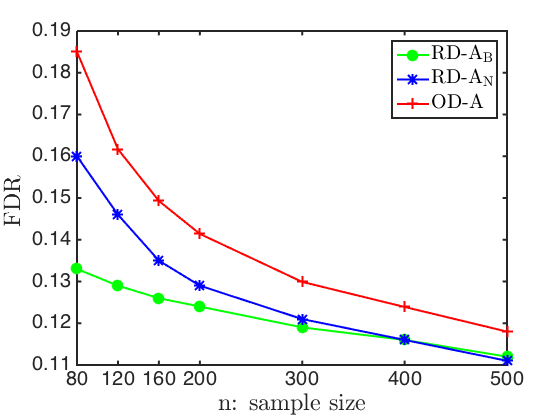}
                   \includegraphics[width=.49\textwidth]{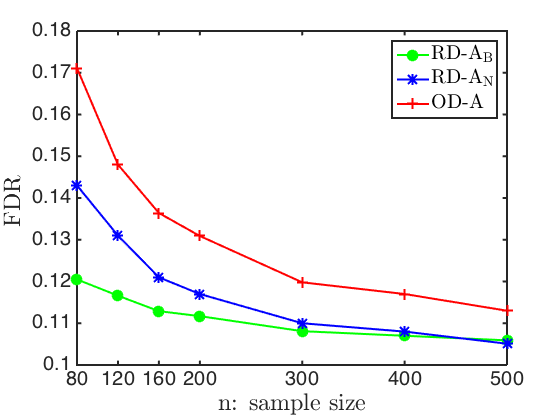}

                \caption{Empirical FDR of the testing problem at the $10\%$ significance level, when the data follows a mixture normal/lognormal distribution (Model 3) on the left panel and a mixture Student's $t$/Weibull distribution (Model 4) on the right panel, and the sample size varies. }

                \label{fig:model4}
\end{figure}

\subsection{Stock market data}

In this section, we apply our proposed robust dependence-adjusted multiple testing procedure to monthly stock market data. Consider Carhart's four-factor model [\cite{C1997}] on S\&P 500 index, where the excess return of a stock has the following representation:
\begin{align}
	r_{jt} & = \mu_j + \beta_{j,{\rm MKT}}({\rm MKT}_t-r_{ft})+\beta_{j,{\rm SMB}}{{\rm SMB}_t} \nn \\
	 & \quad  \quad\quad\quad\quad\quad\quad\quad\quad +\beta_{j,{\rm HML}}{{\rm HML}_t}+\beta_{j,{\rm UMD}}{{\rm UMD}_t}+u_{jt} , \label{carhart}
\end{align}
for  $j=1,\ldots, p$ and $t = 1, \ldots, T$. Here $r_{jt}$ is the excess returns of stock $j$ at month $t$, $r_{ft}$ is the risk free interest rate at month $t$, and MKT, HML, SMB and UMD represent the market, value, size, and momentum factors respectively. We are interested in the value $\mu_j$, which represents the {\it alpha} of stock $j$. A stock can be said to have excess returns if its alpha is positive, or in other words, the stock exhibits returns higher than those that can be accounted for by the four factors. If the alpha is negative, the stock is consistently underperforming, given the level of risk it undertakes. Detecting nonzero alpha is important since it is directly related to the efficient equity market hypothesis.   When the market is inefficient, we can conduct multiple hypothesis testing to identify those stocks in the market that have statistically significant alphas.
When the returns of mutual fund data are used, the test is related to test whether the fund manager has skills or not [\cite{BSW2010}].
All the data in this section was obtained from Kenneth French's website and the COMPUSTAT and CRSP databases.

We obtain monthly data for $393$ S\&P 500 constituents over the time period from January 2005 to December 2013, after removing those stocks that have missing values or have discontinuous inclusion in the index. The stock returns exhibit severely heavy-tails, as illustrated by the histogram of the excess kurtosis of the data in Figure~\ref{fig:stocks}. Among the $393$ series, $112$ have distributions whose tails are fatter than the $t$-distribution with $5$ degrees of freedom.
\begin{figure}        \centering
                \includegraphics[width=.9\textwidth]{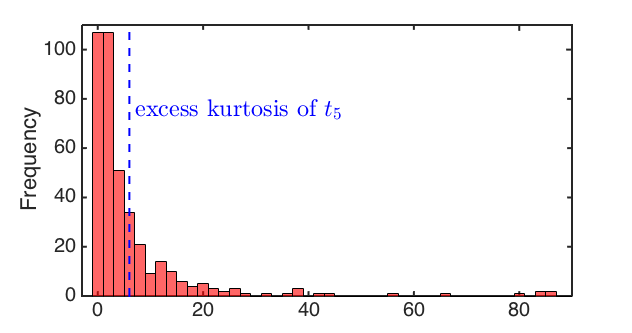}
                \caption{Histogram of excess kurtosises of monthly returns of $393$ S\&P 500 constituents from January 2005 to December 2013. The excess kurtosis of the $t_5$ distribution is shown for reference. }

                \label{fig:stocks}
\end{figure}

The regression in \eqref{carhart} is carried out over rolling windows: for each month, we evaluate the model using data from the preceding three years. For each rolling window, we simultaneously test the hypotheses $H_{0j} :  \mu_j = 0$ versus $H_{1j}: \mu_j \neq 0$ for $j=1,\ldots , p$, using the proposed robust dependence-adjusted procedure. We see that out of a portfolio of size  $393$, only a few stocks exhibit statistically significant nonzero alphas at the FDR threshold of $5\%$, $10\%$ and $20\%$. Table~\ref{tabreal} summarizes the results for the number of selected stocks and estimated alphas with the FDR controlled at 5\%. For the robust dependence-adjusted procedure, we see that no stock is selected on average, with maximum $2$ stocks selected over the entire time period. In particular, our method does not select any stocks from the third quarter of 2008 to the third quarter of 2010, coinciding with the financial crisis during which the market volatility is much higher. Moreover, the estimated alphas for the selected stocks are much higher than those not selected by the robust multiple testing procedure. This is represented as $|\hat{\mu}_j|$ in Table~\ref{tabreal}. The naive method, which directly performs multiple $t$-tests ignoring the common factors, appears to be unstable with the number of stocks selected being extremely variable.
Additionally, a tremendously  large number of stocks are selected in a few time periods, pointing towards false discoveries. In summary, the robust dependence-adjusted multiple testing procedure is particularly suited for the problem of finding a few stocks with nonzero alphas, which is explained by the focus on a balanced panel of highly traded stocks with large capitalizations, namely, the constituents of the S\&P 500.

\begin{table}[tbh]\centering
{\begin{tabular}{lcccccc}
\hline \vspace{-0.25cm} \\
 Variable  & Method&Mean & Std. Dev.&Median &Min &Max \vspace{0.1cm} \\
 \hline  \vspace{-0.2cm}  \\
 \multirow{2}{*}{Number of selected stocks} &RD-A$_{{\rm N}}$  &0.18 &0.42 &0&0&2\\
 & OD-A& 0.94 & 1.59 & 0 & 0 & 7\\
&Naive&4.39 & 22.46 &0&0&178\\
\hline \vspace{-0.2cm} \\
\multirow{2}{*}{$|\hat{\mu}_j|$ for selected stocks } &RD-A$_{{\rm N}}$&4.19\%&1.41\%&4.31\%&2.43\%&6.58\%\\
&OD-A& 3.37\% &1.00\% &3.24\%&1.70\%&6.34\%\\
&Naive& 3.76\% &1.20\% &3.47\%&2.77\%&8.01\%\\
\hline \vspace{-0.2cm} \\
\multirow{2}{*}{$|\hat{\mu}_j|$ for non-selected stocks }  &RD-A$_{{\rm N}}$&1.01\%&0.09\%&1.00\%&0.90\%&1.25\%\\
&OD-A& 1.09\% &0.11\% &1.08\%&0.91\%&1.32\%\\
&Naive& 1.75\% &0.35\% &1.63\%&1.35\%&2.45\%\\\hline
\end{tabular}
\vspace{0.1cm}
\caption{Summary of the three testing procedures based on $393$ stocks in S\&P 500 between January 2005 and December 2013. A rolling window of 3 years is used for estimation and selecting stocks with significant nonzero alpha. The stocks are selected at FDR level $5\%$.}
\label{tabreal}}
\end{table}

\subsection{Gene expression data}

In this section, we apply the proposed procedure to the analysis of a neuroblastoma data set reported in \cite{O2006} to identify differentially expressed genes between the group of patients who had 3-year event-free survival after the diagnosis of neuroblastoma and the group of patients who did not. This data set consists of 251 patients of the German Neuroblastoma Trials NB90-NB2004, diagnosed between 1989 and 2004. The complete data set, obtained via the MicroArray Quality Control phase-II (MAQC-II) project [\cite{Shi2012}], includes gene expression over 10,707 probe sites. There are 246 subjects with 3-year event-free survival information available (56 positive and 190 negative). See \cite{O2006} for more details about the data sets.

In the first stage, we use standard principal component analysis on the two samples to obtain the factors, based on which we construct dependence-adjusted $P$-values to conduct multiple testing in the second step. Note that the test statistic given in \eqref{test.stat} can be directly generalized to the two-sample case: Given two groups of $p$-dimensional ($p=10,707$) observations with sizes $n_1=56$ and $n_2=190$, we compute robust mean and variance estimators $(\hat{\mu}_{1j}, \hat{\mu}_{2j})$ and $( \hat{\sigma}_{1\nu,jj} , \hat{\sigma}_{2\nu,jj} )$ for $j=1,\ldots, p$. Define two-sample test statistics $T_j= (\hat{\mu}_{1j} - \hat{\mu}_{2j}) / (\hat{\sigma}_{1\nu,jj} /n_1 + \hat{\sigma}_{2\nu,jj}  / n_2)^{1/2}$ so that the corresponding $P$-values are $\{ 2\Phi(-|T_j|) \}_{j=1}^p$.

The number of factors is estimated by the eigenvalue ratio estimator proposed in \cite{AH2013}, which was also used in the context of factor-adjusted multiple testing in \cite{FH2017}. The estimator is defined as $\hat{K}= \text{argmax}_{1<k<k_{\max}}(\hat{\lambda}_k/\hat{\lambda}_{k+1})$, where $\hat{\lambda}_j$ is the $j$th eigenvalue of the sample covariance matrix and $k_{\max}$ is the maximum possible number of factors. Following this procedure, we use $K=2$ to model the latent structure in the data.

Next, we conduct multiple testing using the proposed robust dependence-adjusted procedure and the naive procedure based on two-sample $t$-tests. At FDR level 1\%, we detect $3779$ genes and the naive procedure detects $3236$ genes; while at FDR level 5\%, we discover $5223$ genes and the naive procedure discovers $4685$ genes. In general, taking the latent structure into account causes a visible increase in the number of genes that are declared statistically significant regardless of the prechosen FDR level, reflecting the improved power of our method. This phenomenon is in accord with that in \cite{DS2012}.  These results may serve as an exploratory step for more refined analyses regarding those significant genes.

\section{Summary and discussion}
\label{sec5}

This paper consists of two main parts with each one being of independent interest. In the first part, we study the conventional robust $M$-estimation [\cite{H1973}] from a new perspective by allowing the robustification parameter $\tau$ to diverge with the sample size to balance the bias and robustness of the estimator. Our main theoretical contribution (Theorem~\ref{RAthm1}) is a nonasymptotic Bahadur representation of the proposed robust estimator along with a sub-Gaussian-type deviation bound if the error variable has a finite second moment. As by-products, we prove the Berry-Esseen inequality and a Cram\'er-type moderate deviation theorem for the estimator. These probabilistic results are particularly useful in investigating robustness and accuracy of the $P$-values in multiple testing, among other high dimensional statistical inference problems [\cite{FHY2007}, \cite{DHJ2011}, \cite{CSZ2016}].

In the second part, we focus on large-scale multiple testing for dependent and heavy-tailed data. To characterize the dependence, we employ a multi-factor model similar to that used in \cite{DS2012}, \cite{FHG2012} and \cite{FH2017} but with an observable factor. To achieve robustness, we propose a Huber loss based approach to construct test statistics for testing the individual hypotheses. Under mild conditions, our procedure asymptotically controls the overall false discovery proportion at the nominal level. Thorough numerical results on both simulated and real world datasets are also provided to back up our theory. It is shown that the newly proposed robust dependence-adjusted method performs well numerically in terms of both the size and power. It significantly outperforms the multiple $t$-tests under strong dependence, and is applicable even when the true error distribution deviates wildly from the normal distribution. A more interesting and challenging problem is when the dependence structure is characterized by latent factors. In this case, robust estimators of the unobservable factors along with the loadings are required. Large-scale simultaneous inference for latent factor models with heavy-tailed errors is our ongoing work. We leave the details of the results elsewhere in the future.

\newpage

\begin{frontmatter}
\title{{Supplement to ``A New Perspective on Robust $M$-Estimation:  Finite Sample Theory and Applications to Dependence-Adjusted Multiple Testing''}}

\begin{aug}
\author{\fnms{Wen-Xin} \snm{Zhou}\thanksref{m0,m1}\ead[label=e1]{wez243@ucsd.edu}},
\author{\fnms{Koushiki} \snm{Bose}\thanksref{m1}\ead[label=e2]{bose@princeton.edu}},
\author{\fnms{Jianqing} \snm{Fan}\thanksref{m2,m1}\ead[label=e3]{jqfan@princeton.edu}} \\
\and
\author{\fnms{Han} \snm{Liu}\thanksref{m1}
\ead[label=e4]{hanliu@princeton.edu}
\ead[label=u1,url]{http://www.foo.com}}

\runauthor{Zhou, Bose, Fan and Liu}

\affiliation{University of California, San Diego\thanksmark{m0}, Princeton University\thanksmark{m1} \\
and Fudan University\thanksmark{m2}}

\address{W.-X. Zhou \\
Department of Mathematics \\
University of California, San Diego \\
La Jolla, California 92093 \\
USA \\
\printead{e1}}

\address{
K. Bose \\
H. Liu \\
Department of Operations Research \\
~~~and Financial Engineering \\
Princeton University \\
Princeton, New Jersey 08544 \\
USA \\
{E-mail: }\printead*{e2}  \\ \phantom{E-mail:} \printead*{e4}}

\address{
J. Fan\\
School of Data Science \\
Fudan University \\
Shanghai 200433 \\
China \\
and \\
Department of Operations Research \\
~~~and Financial Engineering \\
Princeton University \\
Princeton, New Jersey 08544 \\
USA \\
\printead{e3}}
\end{aug}

\begin{abstract}
This supplemental material contains the proofs for the theoretical results in the main text and additional simulation results.
\end{abstract}

\end{frontmatter}

\appendix

\section{Proofs of the results in Section~2}
\label{sec6}

In this section, we present the proofs to the theoretical results from Section~\ref{sec2}. Throughout, we use $C, C_1, C_2, \ldots$ and $c, c_1, c_2, \ldots$ to denote positive constants independent of $n$ and $p$, which may take different values at each occurrence. First, we collect two useful propositions in Section~\ref{pre.sec}.


\subsection{Preliminaries}
\label{pre.sec}

Proposition~\ref{RAlem1} reveals that the approximation error vanishes as $\tau$ diverges.

\begin{proposition}  \label{RAlem1}
Under Condition~\ref{cond2.1}, it holds as long as $\tau \geq  8K_1^2\, \sigma$ that
\begin{align}
\|  \bS^{1/2} ( \btheta^* - \btheta^*_\tau ) \|    \leq  2 \| \bS^{-1/2} \overline{\bS} \bS^{-1/2} \|^{1/2}  \frac{\sigma  }{\tau}  ,  \label{lem1.1}
\end{align}
where $K_1= (K_0^2+1)^{1/2}$, $\bS  = \e (\bZ \bZ^\T)$ and $\overline{\bS} = \e\{ {\sigma}^2(\bX) \bZ \bZ^\T   \}$.
\end{proposition}

\begin{proof}

Define functions $h(\btheta)=  n^{-1} \sn  \e \ell_\tau( Y_i  -  \bZ_i^\T \btheta)$ and $h_0(\btheta)= (2n)^{-1} \sn  \e  ( Y_i - \bZ_i^\T \btheta )^2 $ for $\btheta\in \bbr^{d+1}$, and put $\bdelta=\btheta^* - \btheta^*_\tau$. By the optimality of $\btheta^*_\tau$ and the mean value theorem, we have $\nabla h( \btheta^*_\tau)= \mo $ and 
\begin{align}
	\bdelta^\T \nabla^2 h( \wt \btheta^*_\tau) \bdelta & =  \big\langle \nabla h( \btheta^* ) - \nabla h( \btheta^*_\tau ) , \bdelta  \big\rangle  \nn \\
	&  = \big\langle  \nabla h( \btheta^* )  ,  \bdelta \big\rangle  =  -\frac{1}{n} \sn \e  \{ \ell_\tau'(\nu_i)   \bZ_i^\T \bdelta  \} , \label{FOC.ineq}
\end{align}
where $\wt{\btheta}_\tau =\lambda \btheta^* + (1-\lambda) \btheta^*_\tau$ for some $0\leq \lambda \leq 1$ and $\nu_i = \sigma(\bX_i) \varepsilon_i$. Moreover, write
$$
	\overline{\bZ} = \sigma(\bX) \bZ, \quad \overline{\bZ}_i = \sigma(\bX_i) \bZ_i ,  \, i=1\, \ldots, n, ~\mbox{ and }~ \overline{\bS} = \e (\overline{\bZ} \,\overline{\bZ}^\T ).
$$
 For the right-hand side of \eqref{FOC.ineq}, since $\e( \nu_i | \bZ_i ) = 0$, we have $- \e \{ \ell_\tau'( \nu_i) | \bZ_i \} = \e  [ \{  \nu_i  1(| \nu_i| > \tau) - \tau 1(\nu_i >\tau) + \tau 1 (\nu_i <-\tau)  \} | \bZ_i  ]$. This, together with H\"older's inequality, implies
\begin{align}
	& \frac{1}{n} \sn | \e \{ \ell_\tau'( \nu_i)   \bZ_i^\T \bdelta \} |   \leq   \frac{1}{n \tau} \sn \e | \nu_i^2 \bZ_i^\T \bdelta |   \nn \\
	& \leq  \frac{1}{n \tau} \sn \{ \e \sigma^2(\bX_i) \}^{1/2} \{\e (\overline \bZ_i^\T \bdelta)^2 \}^{1/2}   \leq     \| \overline \bS^{1/2} \bdelta \|   \frac{\sigma }{\tau} . \label{RHS.ubd}
\end{align}

Next, we deal with the left-hand side of \eqref{FOC.ineq}. Since $h$ is a convex function minimized at $\btheta^*_\tau$, $h(\wt{\btheta}_\tau)  \leq \lambda h(\btheta^*) +(1-\lambda) h(\btheta^*_\tau) \leq  h(\btheta^*) \leq  h_0(\btheta^*) = \sigma^2/2$. On the other hand, note that $h(\btheta) \geq n^{-1} \sn  \e \{  (\tau | Y_i - \bZ_i^\T\btheta| - \tau^2/2 ) 1(| Y_i  - \bZ_i^\T \btheta |>\tau)\}$ for all $\btheta \in \bbr^{d+1}$. Define $\wt \nu_i = Y_i - \bZ_i^\T \wt{\btheta}_\tau$ for $i=1,\ldots, n$. Combining these upper and lower bounds on $h(\wt \btheta^*_\tau)$ with Markov's inequality gives
\begin{align}
	\frac{\tau}{n} \sn \e  \big\{ | \wt \nu_i   |  1\big( | \wt \nu_i  |>\tau \big) \big\}  &  \leq \frac{\tau^2}{2n}  \sn  \mathbb{P} \big( | \wt \nu_i  |>\tau \big) +  \frac{ \sigma^2}{2}  \nn \\
	& \leq  \frac{\tau}{2n} \sn  \e  \big\{ | \wt \nu_i   |  1\big( |  \wt \nu_i | > \tau \big) \big\}  +  \frac{ \sigma^2}{2}  , \nn
\end{align}
which further implies
\begin{align}
	 \frac{1}{n} \sn \mathbb{P} \big( | \wt \nu_i | >\tau \big) \leq \frac{1}{n \tau} \sn \e  \big\{ | \wt \nu_i | 1\big( | \wt \nu_i |>\tau\big) \big\} \leq  \frac{\sigma^2}{\tau^2} . \label{lem1_p3}
\end{align}

Moreover, note that $\nabla^2 h( \wt \btheta^*_\tau) = \bS - n^{-1}\sn \e  \{ 1 (| \wt \nu_i | > \tau ) \bZ_i \bZ_i^\T \}$. Then, it follows from the Cauchy-Schwartz inequality and \eqref{lem1_p3} that
\begin{align}
	& 	\bdelta^\T \nabla^2 h( \wt \btheta^*_\tau) \bdelta   = \| \bS^{1/2} \bdelta \|^2 - \frac{1}{n} \sn   \e  \big\{ 1 \big( | \wt \nu_i | > \tau \big) (\bZ_i^\T \bdelta)^2 \big\}  \nn \\
		& \geq  \| \bS^{1/2} \bdelta \|^2 - \frac{1}{n} \sn \mathbb{P}\big( | \wt \nu_i | > \tau  \big)^{1/2} \big\{ \e \big(\bZ_i^\T \bdelta \big)^4 \big\}^{1/2} \nn \\
	& \geq  \| \bS^{1/2} \bdelta \|^2 -  \bigg\{  \frac{1}{n}  \sn \mathbb{P}\big( | \wt \nu_i | > \tau  \big) \bigg\}^{1/2} \bigg\{ \frac{1}{n}  \sn  \e (\bZ_i^\T \bdelta)^4 \bigg\}^{1/2} \nn \\ 
	& \geq  \| \bS^{1/2} \bdelta \|^2 - \frac{\sigma}{\tau}\bigg\{ \frac{1}{n}  \sn  \e (\bZ_i^\T \bdelta)^4 \bigg\}^{1/2} . \nn
\end{align}
Let $\bzeta = \bS^{-1/2} \bZ$. Under Condition~\ref{cond2.1}, $\bzeta$ is a sub-Gaussian random vector satisfying $\| \bzeta \|_{\psi_2} \leq K_1= (K_0^2 + 1)^{1/2}$ and therefore $\e (\bu^\T \bzeta)^4 \leq 16 K_1^4$ for all $\bu \in \mathbb{S}^d$. The preceding inequality can thus be further bounded from below by
\begin{align*}
	      \big(  1 - 4 K_1^2 \tau^{-1} \sigma \big)  \| \bS^{1/2} \bdelta \|^2    \geq \frac{1}{2 }\| \bS^{1/2} \bdelta \|^2 ,
\end{align*}
provided that $\tau \geq 8K_1^2 \,\sigma$.  This, together with \eqref{FOC.ineq} and \eqref{RHS.ubd}, proves \eqref{lem1.1}.
\end{proof}

Proposition~\ref{RAlem2} shows that the differences between the first two moments of $ \ell'_\tau(\nu )$ and $\nu$ vanish faster if higher moments of $\nu$ exist.

\begin{proposition}  \label{RAlem2}
Let $\nu$ be a real-valued random variable with $\e( \nu ) = 0$ and $\sigma^2 = \e (\nu^2)>0$. Assume that $v_\kappa = \e ( | \nu |^{\kappa} ) <\infty$ for some $\kappa > 2$. Then
\begin{align}
	|\e \ell'_\tau( \nu ) | \leq \min \bigg(  \frac{\sigma^2}{\tau} , \frac{ v_\kappa }{ \tau^{ \kappa -1} } \bigg)  ~\mbox{ and }~  | \e \{ \ell'_\tau( \nu )  \}^2 - \sigma^2 | \leq  \frac{2 v_\kappa  }{ ( \kappa -2 )  \tau^{\kappa -2 } } . \label{var.error}
\end{align}
\end{proposition}

\begin{proof}
Since $\e( \nu ) =0$, we have $\e \{ \ell_\tau'( \nu ) \}  = - \e \{ ( \nu - \tau ) 1 ( \nu > \tau) \} + \e \{ ( - \nu  - \tau ) 1( \nu < -\tau) \}$. Hence, for any $2\leq \iota \leq \kappa$, $|\e \ell_\tau'( \nu ) | \leq \e \{ | \nu | - \tau ) 1(| \nu |>\tau) \} \leq  \tau^{1-\iota} \,\e | \nu |^\iota$, which proves the first inequality in \eqref{var.error}. Next, letting $\eta = | \nu |$ be a nonnegative random variable, we have
\begin{align}
   \e\{ \eta^2 1(\eta > \tau ) \}   & = 2 \e \int_0^\infty  1(\eta>y) 1(\eta>\tau) y \, dy \nn \\
 & = 2 \mathbb{P}(\eta > \tau ) \int_0^\tau y\,dy + 2\int_\tau^\infty y \mathbb{P}(\eta>y) \, dy  \nn \\
 &  = \tau^2 \mathbb{P}(\eta>\tau) + 2 \int_\tau^\infty y \mathbb{P}(\eta > y) \, dy. \nn
\end{align}
By Markov's inequality, 
$$
 \int_\tau^\infty y \mathbb{P}(\eta > y) \, dy \leq \e (\eta^\kappa )  \int_\tau^\infty y^{1-\kappa} \, dy = (\kappa-2)^{-1}\tau^{2-\kappa}\, \e ( \eta^\kappa ) .
$$ 
This, together with the equality $\e \{ \ell'_\tau( \nu ) \}^2 = \e ( \nu^2)  - \{  \e \nu^2 1(| \nu |>\tau) - \tau^2 \mathbb{P}(| \nu |>\tau)\}$ proves the second inequality in \eqref{var.error}.
\end{proof}

The following lemma is borrowed from \cite{FLSZ2015}. It provides a localized analysis mechanism, which is a crucial element in the proof of Theorem~\ref{RAthm1}.

\begin{lemma}   \label{local.lemma}
For any positive integer $d\geq 1$ and convex loss function $L : \bbr^d \mapsto \bbr$, write $D_L(\bbeta_1,\bbeta_2)= L(\bbeta_1)- L(\bbeta_2)- \langle \nabla L(\bbeta_2),\bbeta_1-\bbeta_2  \rangle$ and $\overline D_L(\bbeta_1,\bbeta_2)=D_L(\bbeta_1,\bbeta_2)+D_L(\bbeta_2,\bbeta_1)$ for $\bbeta_1 , \bbeta_2 \in \bbr^d$. Then, $\overline D_L (\bbeta_\eta ,\bbeta^*)\leq \eta \overline D_L(\bbeta,\bbeta^*)$ for any $\bbeta_\eta  =\bbeta^*+ \eta (\bbeta-\bbeta^*)$ with $\eta \in (0,1]$.
\end{lemma}

\subsection{Proof of Theorem~\ref{RAthm1}}

Throughout, let $c_0, c_1, c_2, \ldots $ be positive constants depending only on $\tau_0$, $K_0$ and $\sigma^2$. Moreover, we write 
$$
	\nu_i = \sigma(\bX_i) \varepsilon_i, \quad \bZ_i = (1, \bX_i^\T)^\T ~\mbox{ and }~ \overline{\bZ}_i = \sigma(\bX_i) \bZ_i ~\mbox{ for }  i=1,\ldots, n.
$$

To begin with, we define an intermediate estimate $\hat{\btheta}_\eta = \btheta^* + \eta (\hat{\btheta} - \btheta^*)$ such that $\| \bS^{1/2} ( \hat{\btheta}_\eta  - \btheta^* ) \| \leq r$ for some $r>0$ to be specified. If $\| \bS^{1/2} ( \hat{\btheta} - \btheta^* ) \| \leq r$, we take $\eta =1$; otherwise, there always exists some $\eta \in (0,1)$ such that $\| \bS^{1/2}(  \hat{\btheta}_\eta - \btheta^* )\| = r$. Then, applying Lemma~\ref{local.lemma} to the loss function $L(\btheta) = n^{-1} \sn \ell_\tau(Y_i - \bZ_i^\T \btheta )$ gives
\begin{align}
	\big\langle \nabla L(\hat{\btheta}_\eta) - \nabla L(\btheta^*)  , \hat{\btheta}_\eta - \btheta^* \big\rangle &  \leq \eta \big\langle   \nabla L(\hat{\btheta}  ) - \nabla L(\btheta^*)  , \hat{\btheta} - \btheta^* \big\rangle \nn \\
	&  =  -\eta \big\langle   \nabla L(\btheta^*)  , \hat{\btheta} - \btheta^* \big\rangle , \label{FOC.cond}
\end{align}
where the last step follows from the first order condition that $ \nabla L(\hat{\btheta}  ) =0$. By the mean value theorem, $\nabla L(\hat{\btheta}_\eta) - \nabla L(\btheta^*)  = \nabla^2 L(\wt \btheta_\eta )(\hat{\btheta}_\eta - \btheta^*)$, where $\wt \btheta_\eta$ is a convex combination of $\btheta^*$ and $\hat{\btheta}_\eta$ and thus satisfies $\| \bS^{1/2}( \wt \btheta_\eta - \btheta^* ) \|\leq r$. Together with \eqref{FOC.cond}, this indicates that
\begin{align}
	\min_{\btheta: \| \bS^{1/2} ( \btheta  -\btheta^*) \|\leq r} \lambda_{\min} \big(\bS^{-1/2} \nabla^2 L( \btheta ) \bS^{-1/2}\big) \cdot  \| \bS^{1/2} ( \hat{\btheta}_\eta - \btheta^* ) \| \leq  \|   \bS^{-1/2} \nabla L(\btheta^*) \|.  \label{general.bound}
\end{align}

To bound the left-hand side of \eqref{general.bound} from below, note that 
$$
	\bS^{-1/2} \nabla^2 L(\btheta) \bS^{-1/2} = \frac{1}{n} \sn  1\big(  | Y_i - \bZ_i^\T \btheta | \leq \tau \big) \bzeta_i \bzeta_i^\T,
$$
where ${\bzeta}_i = (1, \zeta_{i1 }, \ldots, \zeta_{id})^\T = \bS^{-1/2} \bZ_i$. Using the inequality that $1 (| Y_i - \bZ_i^\T \btheta | \leq  \tau  )  \geq   1(  |\nu_i | + \| \bzeta_i \| r  \leq \tau)$, we have, for any $\bu \in \mathbb{S}^{d}$,
\begin{align}
 \bu^\T  \bS^{-1/2} \nabla^2 L(\btheta) \bS^{-1/2} \bu   \geq  \frac{1}{n} \sn (  \bu^\T \bzeta_i )^2 1\big( | \nu_i | + \| \bzeta_i \| r \leq \tau \big)  = \bu^\T  \frac{ \bA^\T \bA  }{n} \bu  ,  \nn
\end{align} 
where $\bA = (\ba_1, \ldots, \ba_n)^\T $ with $\ba_i = \bzeta_i 1(  | \nu_i | + \| \bzeta_i \| r  \leq \tau )$. For any $w>0$, applying Remark~5.40 in \cite{V2012} to $\bA$ yields that, with probability greater than $1-2e^{-w}$,
\begin{align}
	  \bigg\|  \frac{1}{n} \bA^\T \bA  -  \frac{1}{n}\sn \mathbb{E} \big\{ 1\big( | \nu_i | + \| \bzeta_i \| r  \leq \tau \big) \bzeta_i \bzeta_i^\T \big\}  \bigg\| \leq \max(\rho, \rho^2)   \nn 
\end{align}
with $\rho = c_0 \sqrt{(d+1+w)/n}$. Observe that $n^{-1}\sn \mathbb{E} \{ 1( | \nu_i | + \| \bzeta_i \| r  \leq \tau ) \bzeta_i \bzeta_i^\T \} =\bI_{d+1} - n^{-1}\sn \mathbb{E} \{ 1( | \nu_i | + \| \bzeta_i \| r  > \tau ) \bzeta_i \bzeta_i^\T\}$. Since $\bzeta$ is a sub-Gaussian random vector with $\| \bzeta \|_{\psi_2} \leq K_1 = (K_0^2+1)^{1/2}$, $\mathbb{E} ( \bu^\T \bzeta)^4 \leq 16 K_1^4$ for all $\bu\in\mathbb{S}^d$. For any $\bu \in \mathbb{S}^d$,
\begin{align}
  & \frac{1}{n}\sn \mathbb{E} \big\{ 1\big( | \nu_i | + \| \bzeta_i \| r > \tau \big) ( \bu^\T \bzeta_i )^2 \big\}  \nn \\
  & \leq  \frac{1}{n \tau  }\sn \mathbb{E} \big\{ \big( | \nu_i| + \|\bzeta_i \| r \big) ( \bu^\T \bzeta_i )^2 \big\}  \nn \\
  &  \leq  \frac{1}{n \tau } \sn   \{  \mathbb{E} (  \bu^\T \bzeta_i  )^4 \}^{1/2} \{ \sigma  +  (d+1)^{1/2} r \} . \nn
\end{align}
It then follows that with probability greater than $1-2 e^{-w }$, 
$$
	\lambda_{\min}\big( n^{-1} \bA^\T \bA \big)  \geq 1- 4 K_1^2  \tau^{-1} \big\{ \sigma +  (d+1)^{1/2} r \big\} - \max(\rho, \rho^2) .
$$
Putting the above calculations together, we conclude that with probability at least $1- 2e^{-w}$,
\bee
   \min_{\btheta: \| \bS^{1/2} ( \btheta  -\btheta^*) \|\leq r} \lambda_{\min} \big(\bS^{-1/2} \nabla^2 L( \btheta ) \bS^{-1/2}\big) \geq \frac{1}{4}  \label{min.eigenvalue.bound}
\eee
as long as $n \geq 16 c_0^2(d+1+w)$ and $\tau \geq 8K_1^2 \{ \sigma +  (d+1)^{1/2} r\}$.

Next we bound the quadratic form $ \|   \bS^{-1/2} \nabla L(\btheta^*) \|$. Define
$$
	\bxi = - \sqrt{n}  \, \bS^{-1/2} \big\{ \nabla L(\btheta^*) - \e \nabla L( \btheta^* ) \big\}  = \frac{1}{\sqrt{n}} \sn  \bzeta_i^* ,
$$
where $\bzeta_i^* = \ell'_\tau( \nu_i ) \bzeta_i - \e \{ \ell'_\tau( \nu_i) \bzeta_i\}$. Throughout the following, we write $\bDelta = \bS^{-1/2} \overline{\bS} \bS^{-1/2} $ with $\overline{\bS} = \e (\overline{\bZ}\, \overline{\bZ}^\T)$. For any $\bu  \in \mathbb{S}^d$, note that
$$
 \big| \e \big\{ \ell'_\tau(  \nu_i) \bu^\T \bzeta_i \big\} \big| \leq \tau^{-1} \e \big\{ \sigma^2(\bX_i ) |\bu^\T \bzeta_i| \big\} \leq  \| \bDelta \|^{1/2} \tau^{-1}  \sigma   .
$$
For any $\lambda \in \bbr$, using inequalities $e^t \leq 1 + t + t^2 e^{t \vee 0}/2$ and $1+t \leq e^t$ gives
\begin{align}
	& \e \exp(\lambda \bu^\T \bxi) =\prod_{i=1}^n \e \exp(  \lambda n^{-1/2}  \bu^\T \bzeta^*_i  ) \nn \\
	&\leq \prod_{i=1}^n    \bigg\{ 1   + \frac{\lambda^2}{2n} \e ( \bu^\T \bzeta^*_i )^2 e^{ \frac{ |\lambda |  }{\sqrt{n}}|\bu^\T \bzeta^*_i| }  \bigg\} \nn \\
	& \leq  \prod_{i=1}^n    \bigg[ 1   + \frac{\lambda^2}{  n }e^{ \| \bDelta \|^{1/2} \frac{ |\lambda | \sigma  }{\sqrt{n} \tau }   }   
\e \big\{	 \nu_i^2  (\bu^\T \bzeta_i )^2  +  \| \bDelta \|    \tau^{-2} \sigma^2 \big\}   e^{ \frac{ |\lambda | \tau}{\sqrt{n}}|\bu^\T \bzeta_i| }  \bigg] \nn \\
	& \leq  \prod_{i=1}^n    \bigg[ 1   + \frac{\lambda^2 } { n}  e^{  \| \bDelta \|^{1/2} \frac{ |\lambda |  \sigma }{\sqrt{n} \tau }  }   \big\{    \| \bDelta \| \tau^{-2}  \sigma^2\, \e e^{ \frac{ |\lambda | \tau}{\sqrt{n}}|\bu^\T \bzeta_i| }  +    \e (\bu^{\intercal}  \overline \bzeta_i)^2  e^{ \frac{|\lambda | \tau}{\sqrt{n}}|\bu^\T \bzeta_i| }   \big\} \bigg],  \nn
\end{align}
where $\overline{\bzeta}_i = \sigma(\bX_i) \bzeta_i$ and $\bDelta = \e(\overline{\bzeta}_i  \overline{\bzeta}_i^\T )  $. Recall that $\tau = \tau_0 \sqrt{n/(d+1+w)}$ for $\tau_0 \geq \sigma$. In addition, if $|\lambda| \leq \sqrt{2(d+1+w)} $, we have 
$$
  \frac{|\lambda | \tau}{\sqrt{n}} \leq  \sqrt{2}\tau_0 ~\mbox{ and }~	 \e \exp(\lambda \bu^\T \bxi) \leq \exp( \nu_0^2 \lambda^2/2),
$$
where $\nu_0 \geq 1$ is a constant depending on $\tau_0$, $K_0$ and $\| \bDelta \|$. Hence, condition (1.18) in the supplement of \cite{S2012} holds with $V_0 = \bI_{d+1}$ and $\mathbf{g} = \sqrt{2(d+1+w)}$. Moreover, put $D_0 = \bI_{d+1} $ so that $D_0^{-1} V_0^2 D_0^{-1} = \bI_{d+1}$. Applying Corollary~1.13 there implies that for any $\sqrt{2(d+1)} / 18 < x \leq x_c$,
\begin{align}
	 \mathbb{P} \big\{  \| \bxi \|^2  >   \nu_0 (d+1 + 6x)  \big\} \leq 2e^{-x} + 8.4 e^{ - x_c } , \nn 
\end{align}
where $x_c = (1-  0.5 \log 3 )(d+1) + 1.5w \geq  0.45 (d+1) + 1.5 w$. In particular, taking $x=(d+1)/3 + w$ in the preceding inequality we obtain that, with probability greater than $1- 5 e^{-w}$,
\begin{align}
 \| \bxi \|^2  \leq 3\nu_0 (d+1 + 2w) . \nn
\end{align}
For $\bS^{-1/2} \e \nabla L(\btheta^*) = - n^{-1} \sn \e\{ \ell'_\tau( \nu_i) \bzeta_i \}$, note that
\begin{align}
	\|  \bS^{-1/2} \e \nabla L(\btheta^*) \|  =  \sup_{\bu \in \mathbb{S}^d} \frac{1}{n} \sn \e  \{ \ell_\tau'(\nu_i) \bu^\T \bzeta_i \}    \leq  \| \bDelta \|^{1/2}  \frac{ \sigma }{\tau}  .  \nn
\end{align}
Together, the last two displays imply that, with probability at least $1- 5 e^{-w}$,
\begin{align}
  \big\|  \bS^{-1/2} \nabla L(\btheta^*) \big\| \leq r_0 :=  \| \bDelta \|^{1/2}\frac{ \sigma }{\tau} +  \sqrt{\frac{3\nu_0 (d+1 + 2w)}{n}}   .  \label{quadratic.bound}
\end{align}

Taking $r_1 = 4.1 r_0$, then it follows from \eqref{general.bound}--\eqref{quadratic.bound} that with probability at least $1-7e^{-w}$, $\| \bS^{1/2} ( \hat{\btheta}_\eta - \btheta^* ) \| \leq 4 r_0 < r_1$ whenever $n \geq c_1(d+w)^{3/2}$. By the definition of $\hat{\btheta}_\eta$ in the beginning of the proof, we must have $\eta=1$ and thus \eqref{concentration.MLE} follows.

Next we prove \eqref{Bahadur.representation}. Based on the above analysis, we only need to focus on a local vicinity of $\btheta^*$, and therefore a variant of Proposition~3.1 in \cite{S2013} can be applied. To this end, we need to check Conditions~($\mathcal{L}_0$) and ($ED_2$) there. For $\btheta \in \bbr^{d+1}$, define the matrix 
$$
	\bD^2(\btheta) = \nabla^2 \e L(\btheta) = \bS - \bS^{1/2} \frac{1}{n}\sn \e \big\{ 1\big(| Y_i - \bZ_i^\T \btheta | > \tau \big) \bzeta_i \bzeta_i^\T \big\} \bS^{1/2} .
$$
Define the parameter set $\Theta_0(r)  = \{ \btheta \in \bbr^{d+1} : \|  \bS^{1/2}( \btheta - \btheta^* ) \| \leq r  \}$ for $r>0$. By \eqref{concentration.MLE}, we see that $\hat{\btheta} \in  \Theta_0(4r_0)$ with probability greater than $1-7e^{-w}$ for $r_0$ given in \eqref{quadratic.bound}.

\medskip
\noindent
{\sc Condition ($\mathcal{L}_0$):}
For every $\btheta\in \Theta_0(r)$ and $\bu \in \mathbb{S}^d$, put $\bdelta = \bS^{1/2}( \btheta - \btheta^*)$ such that $\| \bdelta \| \leq r$, we have
\begin{align}
	&  \big|   \bu^\T \big\{ \bS^{-1/2} \bD^2(\btheta) \bS^{-1/2}  - \bI_{d+1} \big\} \bu  \big|  \nn\\
	&   \leq  \frac{1}{n} \sn \e \big\{ 1\big(| Y_i - \bZ_i^\T \btheta | > \tau \big) ( \bu^\T \bzeta_i )^2 \big\}   \nn \\
	& \leq  \frac{1}{n \tau^2} \sn \e \big\{   \nu_i^2 + (   \bdelta^\T \bzeta_i   )^2 \big\}  ( \bu^\T \bzeta_i )^2    \leq   \tau^{-2}  \big(  \| \bDelta \| + 16  K_1^4 r^2  \big) . \nn
\end{align}
This verifies Condition~$(\mathcal{L}_0)$ by taking
\bee  \label{def.delta}
	\delta(r) =  \tau^{-2}  \big(  \| \bDelta \| + 16  K_1^4 r^2  \big)   , \ \ r > 0.
\eee

\noindent
{\sc Condition ($ED_2$):} Set $\bzeta(\btheta) =  L(\btheta ) - \e L(\btheta)$, such that
$$
	\nabla^2 \bzeta(\btheta)  =  \frac{1}{n} \sn 1\big( |Y_i  -\bZ_i^\T \btheta | \leq \tau \big) \bZ_i \bZ_i^\T - \e \big\{ 1\big( |Y_i  -\bZ_i^\T \btheta | \leq \tau \big) \bZ_i \bZ_i^\T \big\} . 
$$
Following the same arguments as before, for every $\bgamma_1 , \bgamma_2 \in \bbr^{d+1}$ and $\lambda \in \bbr$, put $\wt \bgamma_1 = \bS^{1/2}\bgamma_1 / \| \bS^{1/2} \bgamma_1 \|, \wt \bgamma_2 = \bS^{1/2}\bgamma_2 / \| \bS^{1/2} \bgamma_2 \| $ we have
\begin{align}
	& \e \exp \bigg\{  \frac{\lambda}{n^{-1/2}} \frac{  \bgamma_1^\T  \nabla^2 \zeta(\btheta) \bgamma_2 }{ \| \bS^{1/2} \bgamma_1 \|   \| \bS^{1/2}  \bgamma_2 \| }  \bigg\} \nn \\
	& \leq \prod_{i=1}^n  \bigg( 1 +  \frac{\lambda^2 }{ n } \e \bigg[ \big\{
  ( \wt \bgamma_1^\T \bzeta_i  \wt \bgamma_2^\T \bzeta_i )^2   + \big( \e | \wt \bgamma_1^\T \bzeta_i  \wt \bgamma_2^\T \bzeta_i| \big)^2 \big\}  e^{    \frac{|\lambda|}{\sqrt{n}} \big( |\wt \bgamma_1^\T \bzeta_i  \wt \bgamma_2^\T \bzeta_i | + \e  | \wt \bgamma_1^\T \bzeta_i \wt \bgamma_2^\T \bzeta_i | \big)  }  \bigg]  \bigg) \nn \\
	& \leq  \prod_{i=1}^n  \bigg[ 1 +    e^{\frac{| \lambda | }{\sqrt{n}}}  \frac{\lambda^2}{n}   \e  \big( e^{\frac{ | \lambda| }{\sqrt{n}} |\wt \bgamma_1^\T \bzeta_i  \wt \bgamma_2^\T \bzeta_i | } \big) +    e^{\frac{| \lambda |}{\sqrt{n}}}  \frac{\lambda^2}{n} \e  \big\{    ( \wt \bgamma_1^\T \bzeta_i )^2 ( \wt \bgamma_2^\T \bzeta_i )^2  e^{\frac{ | \lambda | }{\sqrt{n}} |\wt \bgamma_1^\T \bzeta_i  \wt \bgamma_2^\T \bzeta_i | } \big\}  \bigg]  \nn \\
	& \leq  \prod_{i=1}^n \bigg[ 1 +  e^{\frac{ | \lambda| }{\sqrt{n}}}  \frac{\lambda^2}{n}    \max_{ \bu \in \mathbb{S}^d  } \e \big\{  e^{ \frac{ | \lambda| }{\sqrt{n}} (\bu^\T \bzeta )^2  } \big\}+  e^{\frac{| \lambda| }{\sqrt{n}}}  \frac{\lambda^2}{n}    \max_{ \bu \in \mathbb{S}^d  } \e \big\{  ( \bu^\T \bzeta )^4 e^{ \frac{| \lambda | }{\sqrt{n}} (\bu^\T \bzeta )^2  } \big\} \bigg]   \nn \\
	& \leq \exp\bigg[  e^{\frac{ | \lambda| }{\sqrt{n}}}   \lambda^2   \max_{ \bu \in \mathbb{S}^d  } \e \big\{  e^{ \frac{ | \lambda | }{\sqrt{n}} (\bu^\T \bzeta )^2  } \big\} +   e^{\frac{ | \lambda| }{\sqrt{n}}}  \lambda^2   \max_{ \bu \in \mathbb{S}^d  } \e \big\{  ( \bu^\T \bzeta )^4 e^{ \frac{ | \lambda | }{\sqrt{n}} (\bu^\T \bzeta )^2  } \big\} \bigg]. \nn
\end{align}
Under Condition~\ref{cond2.1}, there exists some constant $c_2>0$ depending only on $K_0$ such that, for all $|\lambda | \leq c_2 \sqrt{n}$ and $\btheta \in \bbr^{d+1}$, 
\begin{align}
  \sup_{ \bgamma_1, \bgamma_2 \in \bbr^{d+1}} \e \exp \bigg\{  \frac{\lambda}{n^{-1/2}} \frac{  \bgamma_1^\T  \nabla^2 \zeta(\btheta) \bgamma_2 }{ \| \bS^{1/2}  \bgamma_1 \|   \| \bS^{1/2}  \bgamma_2 \| }  \bigg\} \leq   \exp ( \wt \nu_0^2 \lambda^2 /2 ),	\nn
\end{align}
where $\wt \nu_0 >0$ is a constant depending only on $K_0$. This verifies Condition~($ED_2$) by taking $\omega = n^{-1/2}$, $\nu_0 = \wt \nu_0$ and $\mathbf{g}(r) = c_2\sqrt{n}$ for all $r>0$. Then, applying Proposition~3.1 in \cite{S2013} with $D_0$ replaced by $\bS^{1/2}$ yields that, as long as $n \geq  c_2^{-2}\{ 4(d+1)+2w\}$,
\begin{align}
	\Delta(r) := \sup_{\btheta \in \Theta_0(r)} \big\| \bS^{1/2} (\btheta - \btheta^*) - & \, \bS^{-1/2} \{ \nabla L(\btheta) - \nabla L(\btheta^*) \}   \big\| \nn \\
	&  \leq   \delta(r) r  + 6 \wt \nu_0 \{ 2w + 4(d+1)  \}^{1/2}  n^{-1/2} r \nn
\end{align}
with probability greater than $1-e^{-w}$, where $\delta(r)$ is as in \eqref{def.delta}. This, together with \eqref{concentration.MLE} and the fact $\nabla L(\hat{\btheta})=0$, proves \eqref{Bahadur.representation} by taking $r= 4r_0$. The proof of Theorem~\ref{RAthm1} is then complete.  \qed

\subsection{Proof of Theorem~\ref{RAthm2}}
Let $\xi_{ni} = n^{-1/2} \{ \ell_\tau'( \nu_i ) - m_\tau \}$ be i.i.d. random variables with mean zero, where $m_\tau = \e \{ \ell'_\tau(\nu ) \}$. Put $W_{0n} = \sn \xi_{ni}$ and define $\sigma_{\tau}^2 =  \var\{ \ell_\tau'(\nu ) \}$. Let $G$ and $G_{\tau}$ be two centered Gaussian random variables with variance 1 and $\sigma^{-2} \sigma_{\tau}^2$, respectively. The well-known Berry-Esseen inequality states that
\begin{align}
	 \sup_{ x\in \bbr }  | \mathbb{P}(   \sigma^{-1} W_{0n} \leq x) - \mathbb{P}(G_\tau \leq x ) | \leq
 c_1  \sigma_{\tau}^{-3} v_3  \, n^{-1/2}	  ,  \label{BE.1}
\end{align}
where $c_1>0$ is an absolute constant. By Proposition~\ref{RAlem2}, $\sigma_{\tau}^2 \geq  \sigma^2  -   \sigma^4 \tau^{-2}  - 2 (\kappa-2)^{-1}v_\kappa   \tau^{2-\kappa}$. Hence, the right-hand side of \eqref{BE.1} can be further bounded by $2\sqrt{2}\,c_1  \sigma^{-3} v_3  \, n^{-1/2}$, provided that $\tau  \geq     2\sigma \vee  \{ 8(\kappa-2)^{-1}\sigma^{-2} v_\kappa\}^{1/(\kappa-2)}$.

Next we prove that the distributions of Gaussian random variables $G_0$ and $G_\tau$ are also close. Again, using Proposition~\ref{RAlem2} we deduce that $ |  \sigma^{-2}\sigma_{\tau}^2  - 1  | \leq  \sigma^{-2} \{ 2(\kappa-2)^{-1} v_\kappa \tau^{2-\kappa} +   \sigma^4  \tau^{-2} \} \leq 1/2$ for sufficiently large $\tau$ as above. Then, applying Lemma~A.7 in the supplement of \cite{SZ2015} yields
\begin{align}
	 \sup_{ x\in \bbr }  | \mathbb{P}(G  \leq x) - \mathbb{P}(G_\tau \leq x ) | \leq  (\kappa-2)^{-1} \sigma^{-2} v_\kappa \, \tau^{2-\kappa}  +  \sigma^{2}   \tau^{-2}  /2  . \label{BE.2}
\end{align}

Finally, note that $\sigma^{-1} W_n = \sigma^{-1} W_{0n} + \sigma^{-1} m_\tau\sqrt{n} $. This, together with \eqref{BE.1}, \eqref{BE.2}, Proposition~\ref{RAlem2} and the inequality $\sup_{x\in \bbr} \mathbb{P}( x \leq G \leq x+\varepsilon ) \leq (2\pi)^{-1/2} \varepsilon$ completes the proof of the theorem.   \qed

\subsection{Proof of Theorem~\ref{RAthm3}}

Keeping the notations in the proof of Theorem~\ref{RAthm2}, we define $T_0  =  \sigma_\tau^{-1} W_{0n}$. First we prove that $T$ and $T_0$ are sufficiently close with overwhelmingly high probability. Note that
\begin{align}
	&  | T - T_0  |  \nn \\
	& \leq \frac{\sigma}{\sigma_\tau} \bigg\{  |\sigma^2  - \sigma_\tau^2|   \frac{\sqrt{n}}{\sigma^3}  | \hat{\mu} - \mu^* | + \frac{1}{\sigma} |  \sqrt{n} \, (\hat{\mu} - \mu^*) - W_n | + \frac{\sqrt{n}}{\sigma} |m_\tau| \bigg\}  . \nn
\end{align}
By Proposition~\ref{RAlem2}, we have $|\sigma^2  - \sigma_\tau^2| \leq 2 v_3 \tau^{-1}$ and $|m_\tau | \leq v_3\tau^{-2}$. This, together with Theorem~\ref{RAthm2} yields that
\bee
	\mathbb{P} \big( | T - T_0 | \geq \delta_n \big) \leq C  e^{-w_n}  \label{approxi.1}
\eee
where $\delta_n = c n^{-1/2}(d+w_n) $.

Next we establish a Berry-Esseen type bound for $T$. Define i.i.d. zero-mean random variables $\eta_{i} =  \sigma_\tau^{-1}  \{ \ell'_{\tau }( \nu_i ) - \e \ell'_\tau(\nu_i ) \}$, $i=1,\ldots, n$, such that $T_0 = n^{-1/2} \sn \eta_{i}$. Note that $\max_{1\leq i\leq n} |\eta_{i}| \leq  (2 \sigma_\tau^{-1} \tau \, n^{-1/2} ) \sqrt{n}$, $n^{-1}\cov(\eta_{1} + \cdots + \eta_{n}) = 1$ and $\beta_n :=  n^{-3/2} \sn \e |\eta_{i}|^3 \leq  \sigma_\tau^{-3}  v_3 \, n^{-1/2}$. Applying the Berry-Esseen inequality to $T_0$, and using \eqref{approxi.1}, we deduce that
\begin{align}
	\sup_{x \in \bbr}  | \mathbb{P} ( T \leq x ) - \Phi(x)  | \leq     C \big\{  n^{-1/2}(d+w_n) + e^{-w_n } \big\}   .   \nn
\end{align}
The conclusion \eqref{fs.GAR.hatT} for $0\leq z\leq 1$ thus follows immediately.

It suffices to prove \eqref{fs.GAR.hatT} for $z \in [1, o\{ \min(  \sqrt{w_n}  ,  \sqrt{n}w_n^{-1} ) \} )$. By \eqref{approxi.1}, we have for every $z\geq 1$ and for all sufficiently large $n$,
\begin{align}
	 \mathbb{P} \big( | T_0 |  \geq z +  \delta_n  \big) - C  e^{-w_n} \leq  \mathbb{P}  \big( | T | \geq z  \big)  \leq  \mathbb{P} \big( | T_0 |  \geq z -  \delta_n  \big) +  C e^{-w_n}.  \label{uni.bd}
\end{align}
For $|T_0|$, applying Lemma~3.1 in the supplement of \cite{LS2014} with $d=1$, $B_n =n$, $c_n \asymp w_n^{-1/2}$, $b_n = n^{-1}$, $d_n = n^{-3/10}  $, $t_n = (  C_{3,1}^{-1/2} \vee 4) ( \sqrt{\log n} + B_n^{-1/2} x )$ and $x = B_n^{1/2}t$, we deduce that for all sufficiently large $n$,
\begin{align}
 & \big| \mathbb{P}\big( | T_0 |  \geq t \big) - \mathbb{P}\big(|G| \geq t \big)   \big| \nn \\
 & \leq C   \bigg\{ \frac{  ( \sqrt{\log n} + t )^3 }{\sqrt{n}} + \frac{1+t}{n^{3/10}} \bigg\}  \mathbb{P}\big( | G | \geq t \big) + 7 n^{-1} e^{-t^2}  + 9 \exp\bigg( -   \frac{ c n^{2/5} }{  \log n }  \bigg)  \nn
\end{align}
uniformly for $0\leq t \leq c  \min( \sqrt{w_n} , n^{1/6} )$. As a direct consequence, we have
\begin{align}
	 \mathbb{P}\big( | T_0 |  \geq  t \big) = \big( 1 + C_{n,t}   \big)  \mathbb{P}\big(|G| \geq t \big)   \label{uni.GAR}
\end{align}
uniformly for $0  \leq  t \leq c  \min( \sqrt{w_n} , n^{1/6} )$, where $| C_{n,t}| \leq C   \{  (\sqrt{\log n} + t)^3     n^{-1/2}   +  (1+t) n^{-3/10}  \}$. Moreover, note that for any $t> 0$, $ t(1+t^2)^{-1} e^{-t^2/2} \leq \sqrt{2\pi }  \, \{ 1-\Phi(t) \} \leq  t^{-1} e^{-t^2/2} $. Therefore, for every $z\geq 1$ and all sufficiently large $n$ such that $z-\delta_n>0$,
\begin{align}
\begin{split}
	& \big| \mathbb{P}\big( |G| \geq z-\delta_n \big) - \mathbb{P}\big( |G| \geq z \big) \big|   \\
	  & \leq \frac{2\delta_n}{\sqrt{2\pi}}  \exp\{-(z-\delta_n)^2/2\}   \leq \mathbb{P}\big( |G| \geq z \big)   (1+z)\delta_n \exp(z\delta_n) ,  \\
	& \big| \mathbb{P}\big(|G| \geq z + \delta_n \big) - \mathbb{P}\big( |G| \geq z \big) \big|   \\
	 & \leq \frac{2}{\sqrt{2\pi}}\delta_n  \exp(-z^2/2)  \leq \mathbb{P}\big(|G| \geq z \big) (1+z) \delta_n.
\end{split}  \label{Gaussian.perturbation}
\end{align}

Combining \eqref{uni.bd}, \eqref{uni.GAR} and \eqref{Gaussian.perturbation} we deduce that the convergence in \eqref{fs.GAR.hatT} holds uniformly for $1\leq z \leq o\{ \min(  \sqrt{w_n}  ,  \sqrt{n}w_n^{-1} ) \}$, which completes the proof of \eqref{fs.GAR.hatT}.  \qed

\section{Proofs of the results in Section~3.4}

We present here the proofs to the main theorems in Section~\ref{secT}, starting with a few essential technical results stated as propositions and proved in Section~\ref{appB} below. Throughout, we use $C$ and $c$ to denote positive constants independent of $n$ and $p$, which may take different values at each occurrence.

\subsection{Technical tools}

\begin{proposition} \label{prop1}
Under Conditions~(C1) and (C2), it holds
\begin{align}
 \max_{1\leq j\leq p}  \sup_{ 0\leq z\leq o( \sqrt{w_n} \wedge \sqrt{n} w_n^{-1} ) }  \bigg| \frac{ \mathbb{P}(\sqrt{n} \, \sigma_{\nu,jj}^{-1/2}|\hat{\mu}_j - \mu_j |  \geq z ) }{ 2-2\Phi(z) } - 1 \bigg| \to 0    \label{unif.GAR}
\end{align}
as $(n,p)\to \infty$.
\end{proposition}

Proposition~\ref{prop1} is a direct consequence of Theorem~\ref{RAthm3}. The proof is thus omitted.

\begin{proposition} \label{prop2}
Under Conditions~(C1) and (C2), we have
\begin{align}
	 \max_{1\leq j\leq p}  \big|  \hat \sigma_{\nu,jj}^{-1} \sigma_{\nu,jj}  - 1  \big| \leq  C    n^{-1/2}\sqrt{w_n}  \label{var.consist}
\end{align}
with probability at least $1- C p e^{- w_n}$ for all sufficiently large $n$.
\end{proposition}

The next two propositions give an uniform law of large numbers for $p_0^{-1} \sum_{j \in \mathcal{H}_0} 1(| T_j| \geq z)$ under dependence and independence, respectively, where $\mathcal{H}_0 =\{ j: 1\leq j\leq p, \mu_j =0\}$ and $p_0={\rm Card}(\mathcal{H}_0)$.

\begin{proposition} \label{prop3}
Assume Conditions (C1)--(C3) hold. Then, for any sequence of positive numbers $m_p \leq p$ satisfying $ m_p \to \infty$, we have as $(n,p) \to \infty$,
\begin{align}
	\sup_{0\leq z\leq  \Phi^{-1} (1-  m_p /(2p) ) } \bigg| \frac{\sum_{j\in \mathcal{H}_0} 1(| {T}_j| \geq z) }{ 2p_0\Phi(-z) } - 1 \bigg| \rightarrow 0 \ \ \mbox{ in probability.}  \label{unif.lln}
\end{align}
\end{proposition}

\begin{proposition} \label{prop4}
Assume Conditions (C1), (C2) and (C5) hold. Then \eqref{unif.lln} remains valid for any sequence of positive numbers $m_p \leq  p$ satisfying $m_p \to \infty$.
\end{proposition}

Proofs of Propositions~\ref{prop2}--\ref{prop4} are provided in Section~\ref{appB}.

\subsection{Proof of Theorem~\ref{thm0}}

First, using Propositions~\ref{prop1}, \ref{prop2} and the inequality $t(1+t^2)^{-1} e^{-t^2/2}  \leq \int_t^\infty e^{-u^2/2} \, du \leq t^{-1} e^{-t^2/2}$ for $t>0$, we deduce that as $n\to \infty$,
\begin{align}
	\max_{1\leq j\leq p} \max_{ 0 \leq z \leq o(\sqrt{w_n} \wedge \sqrt{n} w_n^{-1} )} \max\bigg\{ \bigg| \frac{1-F_{j,n}(z)}{1-\Phi(z)} - 1 \bigg|  , \bigg| \frac{F_{j,n}(-z)}{1-\Phi(z)} -1  \bigg|  \bigg\} = o(1).  \label{thm0.1}
\end{align}
Let $z_n>0$ satisfy $1-\Phi(z_n) = \alpha/(3p)$. Then it is easy to see that $z_n = \{1+ o(1)\} \sqrt{2\log p}$. This, together with \eqref{thm0.1} yields $F_{j,n}(-z_n) + 1- F_{j,n}(z_n) = \{1+o(1)\}\{ \Phi(-z_n) + 1 - \Phi(z_n) \} = \{1+o(1)\}  2\alpha / (3p) $. On the event $\mathcal{S}_{j}$, we have $P_j^{{\rm true}}  > \alpha/p \geq F_{j,n}(-z_n) + 1- F_{j,n}(z_n)$ and hence $|T_j| \leq z_n$ for all sufficiently large $n$. This, together with \eqref{thm0.1} proves \eqref{unif.consistency}.  \qed

\subsection{Proof of Theorem~\ref{thm1}}

Noting that
\bee
	\bigg| \frac{\FDP(z)}{\FDP_{{\rm N}}(z)} - 1 \bigg| \leq  \frac{p_0}{p} \bigg| \frac{ \sum_{j \in \mathcal{H}_0} 1(|T_j| \geq z)}{ 2 p_0  \Phi(-z)} - 1  \bigg| +  p^{-1} (p-p_0),  \nn
\eee
the conclusion \eqref{FDP.AFDP} follows immediately from Proposition~\ref{prop3}. \qed

\subsection{Proof of Theorem~\ref{thm2}}

Recall the definition of $\hat z_{ {\rm N} ,0}$ in \eqref{zBH.def}. As in \cite{LS2014}, using the monotonicity of the indicator function and the continuity of $\Phi$, it can be shown that
\begin{align}
	2 \Phi(- \hat z_{{\rm N},0}) =   \frac{\alpha}{p} \max\bigg\{ \sum_{j=1}^p 1\big( | T_j | \geq \hat z_{{\rm N},0} \big) , 1 \bigg\} \geq \frac{\alpha}{p} . \label{zBH.eqn}
\end{align}
Note that $ \Phi(- \sqrt{ 2\log p } ) \sim  (2\pi)^{-1/2} p^{-1} (2\log p)^{-1/2}  = o(p^{-1})$ as $p\to \infty$. This, together with \eqref{zBH.def} and \eqref{zBH.eqn} implies that $\mathbb{P}( \hat z_{{\rm N},0} \leq \sqrt{ 2\log p } \, ) \to 1$. This upper bound on $\hat z_{{\rm N},0}$ can be refined by utilizing Condition~(C4). To see this, note that
\begin{align}
  &  \sum_{j=1}^p 1\big( | T_j | \geq \sqrt{2\log p} \, \big)  \nn \\
   & \geq  \sum_{j=1}^p 1\bigg(  \sqrt{n} \, \hat{\sigma}_{\nu,jj}^{-1/2} | {\mu}_j| \geq \sqrt{ 2\log p } + \sqrt{n}  \max_{1\leq j\leq p}  \hat{\sigma}_{\nu,jj}^{-1/2} |  \hat{\mu}_j - \mu_j |   \bigg) \nn \\
 & \geq \sum_{j=1}^p 1 \bigg(  \sqrt{n} \,  \sigma_{\nu,jj}^{-1/2} | {\mu}_j|     \geq    \frac{ \sqrt{ 2\log p } +  \max_{1\leq j\leq p}  \hat{\sigma}_{\nu,jj}^{-1/2} \sigma_{\nu,jj}^{1/2} \cdot \sqrt{n}  \max_{1\leq j\leq p} \sigma_{\nu,jj}^{-1/2} | \hat{\mu}_j - \mu_j |  }{\min_{1\leq j\leq p}  \hat{\sigma}_{\nu,jj}^{-1/2} \sigma_{\nu,jj}^{1/2}  }   \bigg).  \label{rej.lbd1}
\end{align}

To proceed, we need to derive concentration inequalities for $\hat{\mu}_j$'s and $\hat{\sigma}_{\nu,jj}$'s, which follow from Propositions~\ref{prop1} and \ref{prop2}. For any $\varepsilon >0$, define the event
$$
	F_{1n}(\varepsilon) = \bigg\{  \sqrt{n} \max_{1\leq j\leq p} \sigma_{\nu,jj}^{-1/2}| \hat{\mu}_j - \mu_j | \leq \sqrt{(2+\varepsilon) \log p}  \bigg\}.
$$
Since $w_n \asymp n^{1/5}$ and $\log p = o( n^{1/5} )$, it follows immediately from \eqref{unif.GAR} that
\begin{align}
	\mathbb{P}\big\{ F_{1n}(\varepsilon)^{{\rm c}} \big\} \leq p \max_{1\leq j\leq p} \mathbb{P}\Big\{   \sqrt{n} \, \sigma_{\nu,jj}^{-1/2} | \hat{\mu}_j - \mu_j |   \geq \sqrt{(2+\varepsilon) \log p}   \Big\} \leq C   p^{-\varepsilon/2} .
\end{align}
Moreover, denote by $F_{2n}$ the event that \eqref{var.consist} holds so that $\mathbb{P}(F_{2n}^{{\rm c}}) \leq C  p e^{-c w_n}  \to 0$. On the event $F_{1n}(\varepsilon) \cap F_{2n}$ with $\varepsilon>0$ small enough, note that for all sufficiently large $n$, the right-hand side of \eqref{rej.lbd1} is bounded from below by $c_p := {\rm Card} \{  j: 1\leq j\leq p , \sigma_{\nu,jj}^{-1/2} |\mu_j|  \geq \lambda n^{-1/2} \sqrt{\log p}  \} $ for $\lambda>2\sqrt{2}$ as in Condition~(C4). Hence, by \eqref{zBH.eqn}, with probability converging to $1$ as $(n,p) \to \infty$, $2\Phi(- \hat z_{{\rm N},0} )  \geq  \alpha p^{-1} \sum_{j=1}^p 1(| T_j |\geq  \sqrt{2\log p} \, )  \geq  \alpha  c_p p^{-1}$. This implies that
\begin{align}
	\mathbb{P}\big\{ \hat z_{{\rm N},0}  \leq  \Phi^{-1} (1-   \alpha  c_p / (2p)    ) \big\} \to 1. \label{zBH.ubd}
\end{align}

Together, \eqref{zBH.eqn}, \eqref{zBH.ubd} and Proposition~\ref{prop3} with $m_p = \alpha c_p $ imply that with probability converging to 1 as $(n,p) \to \infty$,
$$
	\frac{ {\rm FDP}( \hat z_{{\rm N},0} )}{  ( p_0 / p ) \alpha   }= \frac{ p_0^{-1}\sum_{j\in \mathcal{H}_0} 1(| T_j | \geq \hat z_{{\rm N},0} ) }{ \alpha p^{-1} \max\{ \sum_{j=1}^p 1(| T_j| \geq \hat z_{{\rm N},0} ) , 1 \} }  =  \frac{\sum_{j\in \mathcal{H}_0} 1(| T_j | \geq \hat z_{{\rm N},0} ) }{2p_0 \Phi(-\hat z_{{\rm N},0} )} \to 1.
$$
This completes the proof of \eqref{oracle.FDP.converge}. \qed

\subsection{Proof of Theorem~\ref{thm3}}  Theorem~\ref{thm3} is proved similarly to Theorem~\ref{thm2}, and so is not derived in detail here. The only difference in the argument is to use Proposition~\ref{prop4} instead of Proposition~\ref{prop3}.  \qed

\subsection{Proof of Theorem~\ref{thm4}}

The proof follows an argument similar to that in the proof of Theorem~\ref{thm2}. It suffices to show that Propositions~\ref{prop1}, \ref{prop2} and \ref{prop3} remain valid for test statistics $S_j$'s and variance estimators $\wt \sigma_{\nu,jj}$'s. Note that the only difference between $T_j$ and $S_j$ is on the variance estimation. The former uses $\hat \sigma_{\nu,jj}$ as an estimator of $\sigma_{\nu,jj} = \sigma_{jj} - \bb_j^\T \bSigma_f \bb_j$, while the latter uses $\wt \sigma_{\nu,jj}$ defined in \eqref{mom-var}.

Reviewing the proof of Propositions~\ref{prop1}--\ref{prop3}, it remains to show that a similar concentration result to \eqref{var.unif.consist1} holds for $\wt \sigma_{jj}(V)$'s for a well-chosen $V$. To see this, using Theorem~1 in \cite{JL2015} with $h(x,y)= (x-y)^2/2$, $m=2$, $q=1$ and $\delta=e^{-w_n}$ for $w_n$ as in Condition~(C2) and the union bound, we deduce that with probability greater than $1-2pe^{-w_n}$, $\max_{1\leq j\leq p}  |\sigma_{jj}^{-1} \wt \sigma_{jj}(V) - 1   | \leq C \kappa_j^{1/2} n^{-1/2} \sqrt{w_n}$ as long as $n\geq 128 \ceil{w_n}$, where $\wt \sigma_{jj}(V)$ is the median-of-means estimator of $\sigma_{jj}$ with $V=V_n = 64 \ceil{w_n}$ and $\kappa_j = \sigma_{jj}^{-2}\, \e(X_j-\mu_j)^4$ is the kurtosis of $X_j$. Keep all other statements the same, we then get the desired result. \qed

\section{Proofs of Propositions~B.2--B.4}
\label{appB}

\subsection{Proof of Proposition~\ref{prop2}}

The proof is based on combining exponential bounds for $\hat{\theta}_j$'s and $\hat{\bb}_j^\T \hat{\bSigma}_f \hat{\bb}_j$'s. For $\hat{\theta}_j$'s, using Theorem~5 in \cite{FLW2014} we deduce that, with $\gamma = \gamma_n$ as in Condition~(C2),
\begin{align}
 	\max_{1\leq j\leq p}  \mathbb{P}\big( | \hat{\theta}_j  - \theta_j | \geq 4 \gamma_0 \,  n^{-1/2} \sqrt{w_n} \big) \leq 2 e^{-w_n}  \nn
\end{align}
as long as $n\geq 8 w_n$, where $ \theta_j = \e ( X_j^2 )$. Then, it follows from the union bound that
\begin{align}
  \max_{1\leq j\leq p} | \hat{\theta}_j  - \theta_j | \leq 4 \gamma_0 \, n^{-1/2} \sqrt{w_n}   \label{var.unif.consist1}
\end{align}
with probability greater than $1-2p e^{-w_n}$ whenever $n\geq 8w_n$. Next, note that for each $j$,
\begin{align}
 &   \big| \| \hat{\bSigma}_f^{1/2} \hat{\bb}_j \|^2 - \|  {\bSigma}_f^{1/2} {\bb}_j \|^2 \big|  \nn \\
    &  \leq \| \bSigma_f^{-1/2} \hat{\bSigma}_f \bSigma_f^{-1/2} -  \bI_K  \| \| \bSigma_f^{1/2}\hat{\bb}_j \|^2  \nn \\
    & \quad +   \big( \| \bSigma_f^{1/2} \hat{\bb}_j   \|+\| \bSigma_f^{1/2} \bb_j \| \big) \| \bSigma_f^{1/2}(\hat{\bb}_j - \bb_j) \|   .  \label{var.unif.ineq}
\end{align}
Applying Theorem~5.39 in \cite{V2012} to i.i.d. random vectors $\bff_{0i} = \bSigma_f^{-1/2}\bff_i$ yields that for every $t >0$, $\| \bSigma_f^{-1/2}\hat{\bSigma}_f \bSigma_f^{ - 1/2}  - \bI_{K}  \| \leq \max(\delta, \delta^2)$ with probability at least $1-2 e^{-t}$, where $\delta = C (K+t)^{1/2} n^{-1/2}$. Taking $t=w_n$, we have
\bee 
	 \|  \bSigma_f^{-1/2}\hat{\bSigma}_f \bSigma_f^{ -1/2}  - \bI_{K} \|  \leq   C n^{-1/2}\sqrt{w_n }   \label{cov.concentration}
\eee
with probability greater than $1-2e^{-w_n}$ for all sufficiently large $n$.

Together, \eqref{var.unif.consist1}--\eqref{cov.concentration} and Theorem~\ref{RAthm1} prove \eqref{var.consist}. \qed

\subsection{Proof of Proposition~\ref{prop3}}

The proof is based on a discretization technique used to prove Theorem~2.1 in \cite{LS2014} and the following results on joint Gaussian approximations.

\begin{lemma}  \label{lemB1}
Under the Conditions (C1)--(C3), we have for any $0<\rho \leq 1$ and $0<\delta <1$,
\begin{align}
	\mathbb{P}\big( | T_j| \geq z, | T_k | \geq z  \big) \leq C \exp\{ - (1-\delta) z^2 / (1+ \rho ) \}  \label{joint.tail}
\end{align}
uniformly for $z\in [0, o(\sqrt{ w_n}))$ and all all $j,k\in \mathcal{H}_0$ satisfying $  j\neq k$ and $|\rho_{\nu , jk}| \leq \rho$. In addition, we have for any $A >0$,
\begin{align}
	\mathbb{P}\big( | T_j| \geq z, | T_k | \geq z  \big) =   (1+ C_{n,z}) \mathbb{P}\big(|G| \geq z\big)^2   \label{joint.md}
\end{align}
uniformly for $0\leq z\leq A \sqrt{\log p}$ and all $j,k\in \mathcal{H}_0$ satisfying $  j\neq k$ and $| \rho_{\nu ,jk} | \leq (\log p)^{-2-\kappa }$, where $G\sim N(0,1)$ and $ |C_{n,z}| \leq C \{   (\log p)^{1/2}   n^{-3/10}  +  (\log p)^{-1-\kappa /2} \}$.
\end{lemma}

Define the function $\Psi(z) = 2\Phi(-z)$, $z\in \bbr$, which is strictly decreasing and continuous. Let $t_0= m_p / p $ and $z_0 = \Psi^{-1}(t_0)$. We claim that $z_0 = \{ 1+ o(1) \} \sqrt{2\log(p / m_p)}$ as $p \to \infty$. To see this, put 
$$
	a_p = \sqrt{2\log(p / m_p) - \log \{ 4 \log(p/m_p) \} },
$$ 
$b_p = \sqrt{2\log(p / m_p)}$, and note that, as $p\to \infty$,
$$
	\Psi(a_p) \sim \sqrt{\frac{2}{\pi}} \frac{2\sqrt{ \log(p/m_p)}}{\sqrt{2\log(p / m_p) - \log \{ 4\log(p/m_p) \} }}  \frac{m_p}{p} , 
$$
$$
	\Psi(b_p ) \sim \frac{1}{\sqrt{\pi \log(p/m_p)}} \frac{m_p}{p} . 
$$
Therefore, $\Psi( b_p) \leq \Psi(z_0) = t_0 \leq \Psi(a_p)$ for all sufficiently large $p$, which proves the claim.

Let $\delta \in (0,1)$ be a constant to be specified. Starting with $t_0$, define $t_\ell = t_0  + e^{\ell} m_p^{1-\delta}/p $ and $z_\ell = \Psi^{-1}(t_\ell)$ for $\ell = 1, \ldots, d_p$ such that $0< t_0 < t_1 < \cdots < t_{d_p} = 1$ and $0 =  z_{d_p} <   \cdots < z_1  < z_0$, where $d_p$ equals either $\floor{ \log \{ (p-m_p)/m_p^{1-\delta}\} }$ or $\floor{ \log \{ (p-m_p)/m_p^{1-\delta}\} } +1$. Observe that, for every $z_{\ell-1} \leq z  \leq z_\ell$,
\begin{align}
	   \frac{\sum_{j\in \mathcal{H}_0} 1(| T_j| \geq z_{\ell  } ) }{ p_0 \Psi(z_{\ell  }) } \frac{\Psi(z_{\ell })}{\Psi(z_{\ell - 1})}   & \leq 
	\frac{\sum_{j\in \mathcal{H}_0} 1(|  T_j| \geq z) }{ p_0 \Psi(z) }   \nn \\
	&  \leq \frac{\sum_{j\in \mathcal{H}_0} 1(| T_j| \geq z_{\ell - 1} ) }{ p_0 \Psi(z_{\ell -1}) } \frac{\Psi(z_{\ell-1})}{\Psi(z_\ell)} . \nn 
\end{align}
Note too that, uniformly in $\ell$,
$$
	\frac{\Psi(z_{\ell-1})}{\Psi(z_\ell)} = \frac{t_0 + p^{-1} m_p^{1-\delta} e^{\ell -1} }{t_0 + p^{-1} m_p^{1-\delta}e^{\ell} } = \frac{1 + m_p^{- \delta}e^{-1}}{ 1 + m_p^{- \delta}} \to 1 .
$$
In view of the last two displays, it is enough to prove that
\begin{align}
	\max_{0\leq \ell \leq d_p} \bigg| \frac{\sum_{j\in \mathcal{H}_0} 1(| T_j| \geq z_\ell ) }{p_0 \Psi(z_\ell)} - 1 \bigg| \to 0   \label{discret.lln}
\end{align}
in probability. For any $\varepsilon>0$, applying Boole's and Markov's inequalities we deduce that
\begin{align}
	& \mathbb{P}\left\{ \max_{0\leq \ell \leq d_p} \bigg| \frac{\sum_{j\in \mathcal{H}_0} 1(| T_j| \geq z_\ell ) }{p_0 \Psi(z_\ell)} - 1 \bigg|  \geq \varepsilon \right\}  \nn \\
	& \leq \sum_{\ell = 0 }^{d_p}  \mathbb{P}\left[  \bigg| \frac{ \sum_{j\in \mathcal{H}_0} \{  1(|T_j| \geq z_\ell ) - \mathbb{P}(|G| \geq z_\ell) \} }{p_0 \Psi(z_\ell)}   \bigg|  \geq \varepsilon \right] \nn \\
	& \leq   \frac{1}{  ( \varepsilon p_0 )^2 }   \sum_{\ell=0}^{d_p} \frac{1}{ \Psi^2(z_\ell)}  \sum_{j\in \mathcal{H}_0}  \bigg( \sum_{k\in \mathcal{I}_j} +  \sum_{k\in \mathcal{I}_j^{{\rm c}}}  \bigg)  Q_{jk}(z_\ell),  \label{ulln.ubd1}
\end{align}
where $\mathcal{I}_j = \{ k \in \mathcal{H}_0 :  | \rho_{\nu ,jk} | \leq (\log p)^{-2- \kappa} \}$ and 
\begin{align}
Q_{jk}(z) & = \mathbb{P}\big( | T_j| \geq z , | T_k| \geq z \big) \nn \\
& \quad - \mathbb{P}\big(|G| \geq z \big) \big\{ \mathbb{P}\big( | T_j| \geq z\big) + \mathbb{P}\big( | T_k| \geq z \big) \big\} + \mathbb{P}\big(|G| \geq z \big)^2 . \nn
\end{align}
Recall that $z_0 \sim \sqrt{2\log(p/m_p)}$. For each $j \in \mathcal{H}_0$ and $k\in \mathcal{I}_j$, from \eqref{fs.GAR.hatT} and \eqref{joint.md} it can be derived that, uniformly in $\ell$,
\begin{align}
	Q_{jk}(z_\ell) \leq C \big\{   (\log p)^{1/2}   n^{-3/10}   + (\log p)^{-1-\kappa/2 }  \big\}    \Psi^2(z_\ell)  .   \label{ulln.ubd2}
\end{align}
On the other hand, for $j \in \mathcal{H}_0$ and $k\in \mathcal{I}_j^{{\rm c}} \setminus \{ j \}$, using \eqref{joint.tail} and the inequality $\Psi(z) \geq (2/\pi)^{1/2} (1+z^2)^{-1} z e^{-z^2/2} $ for $z\geq 0$, we deduce that
\begin{align}
 Q_{jk}(z_\ell)  \leq C ( 1 \vee z_\ell ) \exp\bigg( z_\ell^2 - \frac{1-\delta}{1+ \rho} z_\ell^2  \bigg)\Psi^2(z_\ell)   \label{ulln.ubd3}
\end{align}
uniformly in $\ell$. Together, \eqref{ulln.ubd1}, \eqref{ulln.ubd2} and \eqref{ulln.ubd3} imply that, as long as $0<\delta < \{ 1- \rho - (1+\rho) r \}/2$,
\begin{align}
	&  \mathbb{P}\left\{ \max_{0\leq \ell \leq d_p} \bigg| \frac{\sum_{j\in \mathcal{H}_0} 1(|T_j| \geq z_\ell ) }{p_0 \Psi(z_\ell)} - 1 \bigg|  \geq \varepsilon \right\}  \nn \\
	& \leq C \bigg\{ \frac{1}{p_0}\sum_{\ell =0}^{d_p}  \frac{ 1}{ \Psi(z_\ell)}   +   (\log p)^{3/2}   n^{-3/10}   + (\log p)^{-\kappa /2} + \frac{d_p s_p}{p_0} \exp\bigg( \frac{\rho +\delta}{\rho+1} z_0^2 \bigg)   \bigg\} \nn \\
	& \leq C  \bigg\{   m_p^{-1} + m_p^{-1+\delta} \sum_{\ell =1}^{d_p} e^{-\ell}  +  (\log p)^{3/2}   n^{-3/10}   + (\log p)^{- \kappa /2 } \nn \\
	 & \quad \quad  \quad     +  m_p^{-(2\rho+2\delta)/(1+\rho)}   p^{-1 + r +  (2\rho+ 2\delta) / (1+\rho)  } \sqrt{\log p} \bigg\} \to 0 \nn
\end{align}
as $(n,p) \to \infty$, where $s_p$, $r$ and $\rho$ are defined in Conditions (C3). This proves \eqref{discret.lln}, and hence completes the proof of Proposition~\ref{prop3}. \qed

\subsection{Proof of Proposition~\ref{prop4}}
For $j=1,\ldots, p$, define 
$$
	T_{0j} = \frac{1}{ \sigma_{j,\tau}  \sqrt{n} } \sn \{ \ell'_\tau(\nu_j) - m_{j,\tau} \},
$$
where $\sigma_{j,\tau}^2 = \var\{ \ell'_\tau( \nu_j) \}$ and $m_{j,\tau} = \e  \{ \ell'_\tau(\nu_j)\}$. By \eqref{approx.3} in the proof of Lemma~\ref{lemB1} below and the union bound, we have
\bee
	\mathbb{P}\bigg( \max_{ j \in \mathcal{H}_0}  |T_j - T_{0j}| \geq \delta_n \bigg) \leq C  p e^{-w_n}  \label{prop4-1}
\eee
for all sufficiently large $n$, where $\delta_n = c n^{-1/2} w_n$. Moreover, note that 
\bee
	\Phi(-z) = \Phi(-z \pm \delta_n ) \big\{ 1 +  O\big( n^{-1/2}  w_n  \sqrt{\log p} \big) \big\}  \label{prop4-2}
\eee
uniformly in $0\leq z\leq \Phi^{-1}(1-m_p/(2p))$. In view of \eqref{prop4-1} and \eqref{prop4-2}, and by the argument leading to \eqref{discret.lln}, we only need to prove that as $(n,p) \to \infty$,
\bee
	\sup_{ 0\leq \ell \leq d_p } \bigg|  \frac{\sum_{j\in \mathcal{H}_0} 1(|T_{0j} | \geq z_\ell ) }{2p_0 \Phi(-z_\ell)} - 1 \bigg| \rightarrow 0 \ \ \mbox{ in probability} ,    \label{prop4-3}
\eee
where $0= z_{d_p} < \cdots < z_1 <z_0 = \{1+o(1)\}\sqrt{2\log(p/m_p)} $ are as in the proof of Proposition~\ref{prop3}.

First, it follows from \eqref{uni.GAR} that $\mathbb{P}(|T_{0j}| \geq z) = 2\Phi(-z)\{ 1+o(1) \}$ uniformly in $j \in \mathcal{H}_0$ and $0\leq z \leq  o\{ \min(\sqrt{w_n} , n^{1/6} ) \}$. Also, note that under Condition~(C5), $T_{01}, \ldots, T_{0p}$ are independent random variables. Hence, for any $\varepsilon>0$, using the union bound and Markov's inequality we deduce that for all sufficiently large $n$,
\begin{align}
	& \mathbb{P}\left\{ \max_{0\leq \ell \leq d_p} \bigg| \frac{\sum_{j\in \mathcal{H}_0} 1(| T_{0j} | \geq z_\ell ) }{ 2p_0 \Phi(-z_\ell)} - 1 \bigg|  \geq \varepsilon \right\}  \nn \\
	& \leq \sum_{\ell = 0 }^{d_p}  \mathbb{P}\left[  \bigg| \frac{ \sum_{j\in \mathcal{H}_0} \{  1( | T_{0j} | \geq z_\ell ) - \mathbb{P}(|T_{0j} | \geq z_\ell) \} }{ 2p_0 \Phi(-z_\ell) }   \bigg|  \geq  \frac{ \varepsilon }{2} \right]\nn \\
	& \leq   \frac{1}{  ( \varepsilon p_0 )^2 }   \sum_{\ell=0}^{d_p} \frac{1}{ \Phi^2(-z_\ell)}  \sum_{j\in \mathcal{H}_0}   \mathbb{P}\big( |T_{0j}| \geq z_\ell \big) \big\{ 1-  \mathbb{P}\big( |T_{0j}| \geq z_\ell \big) \big\}  \nn \\
	& \leq    \frac{C}{   \varepsilon^2 p}    \sum_{\ell=0}^{d_p} \frac{1}{ \Phi(-z_\ell)} \leq \frac{C}{\varepsilon^2 } \bigg( m_p^{-1} +  m_p^{-1+\delta} \sum_{\ell=1}^{d_p}  e^{-\ell} \bigg) .   \nn
\end{align}
This proves \eqref{prop4-3}, and thus completes the proof of \eqref{unif.lln}. \qed

\subsection{Proof of Lemma~\ref{lemB1}}

First we prove \eqref{joint.tail}. The conclusion is obvious when $0\leq z\leq 1$, so we only need to focus on the case of $z\geq 1$. For $j =1,\ldots, p$, define $\sigma_{j,\tau}^2 = \var\{ \ell'_\tau(\nu_j) \}$, $m_{j,\tau} = \e \{ \ell'_\tau(\nu_j) \}$ and $ T_{0j} =  \sigma_{j,\tau}^{-1} \, n^{-1/2} \sn \{ \ell'_\tau(\nu_j) - m_{j,\tau}\}$. By Proposition~\ref{prop2}, and using an argument similar to that leads to \eqref{approxi.1}, it can be shown that
\begin{align}
	\max_{ j \in \mathcal{H}_0 } \mathbb{P} \big(   | T_j - T_{0j} | \geq \delta_n  \big) \leq C  e^{-w_n}   \label{approx.3}
\end{align}
for all sufficiently large $n$, where $\delta_n = c  n^{-1/2} w_n$.

For $z\geq 1$ and $  j \neq k \in \mathcal{H}_0$, it follows from \eqref{approx.3} that for all sufficiently large $n$,
\begin{align}
 & \mathbb{P}\big( | T_j | \geq z,   | T_k | \geq z  \big) \nn \\
 & \leq \mathbb{P}\big( | T_{0j} | \geq z - \delta_n ,   | T_{0k} | \geq z - \delta_n \big) + C e^{-w_n} \nn \\
 & \leq \mathbb{P}\big\{ | T_{0j} + T_{0k} | \geq 2(z-\delta_n ) \big\} + \mathbb{P}\big\{ | T_{0j} - T_{0k} | \geq 2(z-\delta_n )  \big\} + C e^{-w_n}.   \label{joint.tail.prob}
\end{align}
For $T_{0j} + T_{0k} = n^{-1/2} \sn [   \sigma_{j,\tau}^{-1} \{ \ell'_\tau( \nu_j ) - m_{j,\tau} \} +  \sigma_{k ,\tau}^{-1} \{ \ell'_\tau( \nu_k )  - m_{k,\tau} \} ]$, note that
\begin{align}
 & \e  \big[   \sigma_{j,\tau}^{-1} \big\{ \ell'_\tau( \nu_j ) - m_{j,\tau} \big\} +  \sigma_{k ,\tau}^{-1} \big\{ \ell'_\tau( \nu_k )  - m_{k,\tau} \big\} \big]^2  \nn \\
 & = 2 + \frac{2}{ \sigma_{j,\tau} \sigma_{k,\tau} } \big[ \e \big\{ \ell'_\tau( \nu_j )  \ell'_\tau( \nu_k ) \big\} - m_{j, \tau} m_{k, \tau}   \big] ,  \label{var.dec}
\end{align}
where
\begin{align}
 &  \e \big\{ \ell'_\tau( \nu_j )\ell'_\tau( \nu_k ) \big\}  \nn \\
 & = \e \big\{ \nu_j  \nu_k 1\big(| \nu_j | \leq \tau,  | \nu_k | \leq \tau \big)  \big\} + \tau \e \big\{ \nu_j 1\big(| \nu_j| \leq \tau , | \nu_k| > \tau \big) \big\} \nn \\
 & \quad + \tau \e \big\{  \nu_k 1\big( | \nu_k| \leq \tau , | \nu_j | > \tau \big) \big\}  + \tau^2 \e \big\{ \sgn( \nu_j \nu_k) 1 \big(| \nu_j | >\tau , | \nu_k | >\tau \big) \big\} \nn \\
 & =  \sigma_{\nu, jk } + \tau \e \big\{  \nu_j 1\big(| \nu_j| \leq \tau , | \nu_k| > \tau \big) \big\} + \tau \e \big\{  \nu_k 1\big( | \nu_k| \leq \tau , | \nu_j | > \tau \big) \big\} \nn \\
 & \quad  - \e \big[ \nu_j \nu_k\big\{ 1\big( | \nu_j|>\tau\big) + 1\big( | \nu_k | > \tau \big) \big\}  \big]  \nn \\
 & \quad + \e\big[ \big\{  \nu_j \nu_k  + \tau^2 \sgn( \nu_j \nu_k ) \big\}  1\big( | \nu_j | >\tau , | \nu_k | >\tau \big) \big] . \nn
\end{align}
Under the condition that $\max_{1\leq j\leq p} \e ( \nu_j^4 )  \leq C_\nu$, this implies
\begin{align}
 \big| \e \big\{ \ell'_\tau(\nu_j)\ell'_\tau(\nu_k) \big\}  -  \sigma_{\nu, jk } \big|  \leq C  n^{-1} w_n .  \label{cov.approxi}
\end{align}
Also, it follows from Lemma~\ref{RAlem2} that $\max_{ j \in \mathcal{H}_0} \max(  |m_{j,\tau}| ,  | \sigma_{\nu ,jj} -  \sigma_{j,\tau}^2   | )  \leq C  n^{-1} w_n$. Substituting this into \eqref{var.dec} gives
\begin{align}
	 \e  \big[  \sigma_{j,\tau}^{-1} \big\{ \ell'_\tau( \nu_j ) - m_{j,\tau} \big\} +  \sigma_{k ,\tau}^{-1} \big\{ \ell'_\tau( \nu_k )  - m_{k,\tau} \big\} \big]^2  \leq 2\big(1 + | \rho_{\nu , jk}  | \big) + C n^{-1} w_n. \nn
\end{align}
Then, applying Bernstein's inequality to $T_{0j}+T_{0k}$ we deduce that
\begin{align}
	&	\mathbb{P}\big\{   |T_{0j} + T_{0k}| \geq 2(z-\delta_n) \big\} \nn \\
	& \leq  2\exp \left\{  -  \frac{ (z-\delta_n)^2}{ 1 + | \rho_{\nu,jk} |  + C( n^{-1} w_n  +  w_n^{-1/2}   z ) }  \right\}  \nn \\
	& \leq  2\exp\left\{  -  \frac{ z^2}{ 1 + | \rho_{\nu,jk} |  + C( n^{-1} w_n  +  w_n^{-1/2}   z ) }  \right\}  \label{tail.prob.Tj+Tk}
\end{align}
for all $z \geq 1$ and sufficiently large $n$. A similar argument can be used to show that the same type of bound holds for $\mathbb{P}\{   |T_{0j} - T_{0k}| \geq 2(z-\delta_n)  \}$. Together, \eqref{joint.tail.prob} and \eqref{tail.prob.Tj+Tk} prove \eqref{joint.tail} for $z\in [1 , o( \sqrt{w_n} ) )$.

Next we prove \eqref{joint.md} when $| \rho_{\nu,jk} | \leq (\log p)^{-2-\kappa}$ for some $\kappa >0$. Without loss of generality, we assume that $0<\kappa<1$. Write $ {\bT} = ( T_j, T_k)^\T$  and $\bT_0 = ( T_{0j} , T_{0k})^\T$. Let $\bxi_{ni} = \big( \sigma_{j,\tau}^{-1} (\Id - \e) \ell'_\tau( \nu_k )  ,  \sigma_{k,\tau}^{-1} (\Id - \e) \ell'_\tau( \nu_k)   \big)^\T$, $i=1,\ldots, n$ be i.i.d. bivariate random vectors with mean zero and covariance matrix $\bA = (a_{\ell m})_{1\leq \ell, m\leq 2}$, where $a_{11} = a_{22} =1 $ and $a_{12} = a_{21} = \rho :=(  \sigma_{j,\tau} \sigma_{k,\tau} )^{-1} \e \{  (\Id - \e) \ell'_\tau( \nu_j ) (\Id - \e) \ell'_\tau( \nu_k) \}$. For every $t>0$, we have
\begin{align}
  \mathbb{P}\big( \| \bT \|_{\min} \geq t \big)
  & = \mathbb{P}\big(  T_{0j} \leq - t, T_{0k} \leq -t \big)  + \mathbb{P}\big(  T_{0j}^- \leq - t,  T_{0k}^- \leq -t \big)   \label{joint.dec}   \\
  & \quad \ \ + \mathbb{P}\big(  T_{0j}^- \leq - t, T_{0k} \leq -t \big) + \mathbb{P}\big(  T_{0j} \leq - t,  T_{0k}^- \leq -t \big) , \nn
\end{align}
where $T_{0j}^- = -T_{0j}$ and $T_{0k}^- = -T_{0k}$. Applying Theorem~1.1 in \cite{B2003} to $\bT_0 = n^{-1/2} \sn \bxi_{ni}$ and $(T_{0j}^-, T_{0k}^-)^\T = - n^{-1/2} \sn \bxi_{ni}$, we deduce that
\begin{align}
\begin{split}
	\sup_{t_1 , t_2 \in \bbr} \big| \mathbb{P}\big(T_{0j} \leq t_1 , T_{0k} \leq t_2 \big) - \mathbb{P}\big( W_1 \leq t_1, W_2 \leq t_2\big) \big| & \leq C  \, \e \| \bxi_{n1} \|_2^3 \, n^{-1/2} ,  \\
	\sup_{t_1 , t_2 \in \bbr} \big| \mathbb{P}\big(T_{0j}^- \leq t_1 , T_{0k}^- \leq t_2 \big) - \mathbb{P}\big( W_1 \leq t_1, W_2 \leq t_2\big) \big| & \leq C  \, \e \| \bxi_{n1} \|_2^3 \, n^{-1/2} , 
\end{split}  \label{2d.BE1}
\end{align}
where $\bW= ( W_1, W_2)^\T \sim N(\mo, \bA)$. Further, it can similarly shown that
\begin{align}
\begin{split}
\sup_{t_1 , t_2 \in \bbr} \big| \mathbb{P}\big( T_{0j}^- \leq t_1 , T_{0k} \leq t_2 \big) - \mathbb{P}\big( W_1^- \leq t_1, W_2 \leq t_2\big) \big| & \leq C  \, \e \| \bxi_{n1} \|_2^3 \, n^{-1/2} ,    \\
\sup_{t_1 , t_2 \in \bbr} \big| \mathbb{P}\big( T_{0j} \leq t_1 , T_{0k}^- \leq t_2 \big) - \mathbb{P}\big( W_1 \leq t_1, W_2^- \leq t_2\big) \big| & \leq C  \, \e \| \bxi_{n1} \|_2^3 \, n^{-1/2} , 
\end{split}   \label{2d.BE2}
\end{align}
where $W_1^-=-W_1$ and $W_{2}^- = -W_2$. For the Gaussian random vector $ ( W_1, W_2)^\T$, the anti-concentration inequality states that for every $t_1, t_2\in \bbr$ and $\varepsilon>0$, $\mathbb{P}(W_1 \leq t_1 + \varepsilon ,  W_2 \leq t_2  + \varepsilon ) - \mathbb{P}(W_1 \leq t_1 , W_2 \leq t_2 ) \leq C  \varepsilon$. This, combined with \eqref{approx.3}, \eqref{joint.dec}, \eqref{2d.BE1} and \eqref{2d.BE2} yields that
\begin{align}
  \sup_{t \geq 0 }  \big|  \mathbb{P}\big( \|  {\bT} \|_{\min}  \geq t \big) - \mathbb{P}\big(  \| \bW \|_{\min} \geq t \big) \big| \leq  C \big(  n^{-1/2} w_n +  e^{-w_n} \big) .   \label{min.BE}
\end{align}
By \eqref{cov.approxi} and the assumption that $w_n \asymp n^{-1/5}$ and $\log p=o(n^{1/5})$, $W_1$ and $W_2$ are weakly correlated with $\rho = \cov(W_1, W_2) \leq  C (\log p)^{-2- \kappa}$. Therefore, for every $0\leq t\leq 1$ and all sufficiently large $n$, we deduce that
\begin{align}
 &  \mathbb{P}\big(  \| \bW \|_{\min} \geq t \big)  \nn \\
 & = \frac{1}{ 2\pi \sqrt{1-\rho^2}} \int_{  |x| \geq t }  \int_{  |y|  \geq t } \exp\bigg\{ - \frac{x^2 + y^2 - 2\rho xy}{2(1-\rho^2)} \bigg\} \, dx \, dy  \nn \\
 & \leq   \frac{1}{2\pi  } \int_{|x| \geq t} \int_{|y| \geq t}\exp\bigg( - \frac{x^2 + y^2 }{2 } \bigg) \, dx \, dy  \big\{ 1 + C (\log p)^{-1-\kappa} \big\}  \nn \\
 & \leq  \mathbb{P}\big(  \| \bG \|_{\min}  \geq t \big) \big\{ 1 + C (\log p)^{-1-\kappa} \big\}  \label{GAR.perturb1}
\end{align}
and similarly,
\begin{align}
   \mathbb{P}\big(  \| \bW \|_{\min} \geq t \big)   \geq  \mathbb{P}\big(  \| \bG \|_{\min}  \geq t  \big) \big\{ 1  - C (\log p)^{-1- \kappa} \big\} ,   \label{GAR.perturb2}
\end{align}
where $\bG = ( G_1, G_2)^\T \sim N(\mo, \bI_2)$. Consequently, the conclusion \eqref{joint.md} for $0\leq z\leq 1$ follows from \eqref{min.BE}, \eqref{GAR.perturb1} and \eqref{GAR.perturb2}.

Now it remains to consider the case of $z\geq 1$. By \eqref{approx.3},
\begin{align}
& \mathbb{P}\big(  \| \bT_0 \|_{\min}    \geq z + \delta_n  \big) -  C e^{-w_n} \nn \\
& \leq  \mathbb{P}\big( \|  {\bT } \|_{\min} \geq z  \big)  \leq \mathbb{P}\big( \| \bT_0 \|_{\min} \geq z - \delta_n \big) + C  e^{-w_n} . \label{joint.bd}
\end{align}
Again, we shall use Lemma~3.1 in the supplement of \cite{LS2014} to prove the Gaussian approximation for $\bT_0$. Recall that $\bT_0 = n^{-1} \sn \bxi_{ni}$, where $\bxi_{ni}$'s are i.i.d. centered  random vectors satisfying $\max_{1\leq i\leq n}\| \bxi_{ni} \|_2 \leq c   (  n^{-1/2}\tau ) \sqrt{n}$ and $\beta_n := n^{-3/2} \sn \e \| \bxi_{ni} \|_2^3 \leq C n^{-1/2} $. Moreover, using Lemma~\ref{RAlem2} and \eqref{cov.approxi},
$$
	 \| n^{-1} \cov( \bxi_{n1} + \cdots + \bxi_{nn} ) - \bI_2  \|  = \|  \bA - \bI_2 \| \leq C    (\log p)^{-2 -\kappa}  .
$$
Hence, it follows from Lemma~3.1 in the supplement of \cite{LS2014} by taking $d=2$, $B_n =n$, $c_n \asymp w_n^{-1/2}$, $b_n \asymp   (\log p)^{-2-\kappa}$, $d_n = (\log p)^{-3/2 - \kappa/2 }$, $t_n=  (   C_{3,2}^{-1/2}\vee 4) ( \sqrt{\log p} + B_n^{-1/2} x)$ and $x=B_n^{1/2} t$ that
\begin{align}
	&  \big| \mathbb{P}\big( \| \bT_0 \|_{\min} \geq t  \big)   -  \mathbb{P}\big(  \| \bG \|_{\min} \geq t \big)  | \nn       \\
	&  \leq  C   \bigg\{  \frac{( \sqrt{\log p} +t )^3 }{\sqrt{n}} + \frac{1+t}{ (\log p)^{(3+ \kappa )/2} }  \bigg\}  \mathbb{P} \big(  \| \bG \|_{\min} \geq t \big) +  7 n^{-1}  e^{-t^2} \nn \\
	& \quad   \quad \quad \ \   + 9  \exp\bigg[  - c  \min\bigg\{  \frac{n}{(\log p)^{3+\kappa} \log n } ,  (\log p)^{(1+\kappa)/2}  \sqrt{w_n}    ,  (\log p)^{1+\kappa} \bigg\} \bigg]    \label{joint.GAR}
\end{align}
for all $0\leq t \leq c \min\{ \sqrt{w_n}  , (\log p)^{(3+\kappa)/2} \}$ with $n$ sufficiently large.

Together, \eqref{joint.bd}, \eqref{joint.GAR} and \eqref{Gaussian.perturbation} prove \eqref{joint.md} for $1\leq z \leq A \sqrt{\log p}$.  \qed

\section{Additional simulation results}
\label{appC}

In this section, we present numerical results comparing the performance of the RD-A$_{{\rm N}}$ procedure and the OD-A procedure. Again, we consider the three factor model $ X_{ij} = \mu_j + \bb_j^\T \bff_i +  u_{ij}$ for $i=1,\ldots, n$, where $\bu_i = (u_{i1}, \ldots, u_{ip})^\T$ are i.i.d. copies of $\bu = (u_1,\ldots, u_p)^\T$. We simulate $\{\bff_i\}_{ i =1}^n$ from $N_3( \mo ,\bSigma_f)$ as in the main text; independently, we generate the loadings $\{ \bb_j = (b_{j1}, b_{j2}, b_{j3})^\T \}_{j=1}^p$ according to $b_{j1} , b_{j3} \overset{{\rm i.i.d.}}{\sim}\mbox{Uniform}(0.5,1.5)$ and $b_{j2}\overset{{\rm i.i.d.}}{\sim}\mbox{Uniform}(-2,-1)$. The errors $\{\bu_i\}_{i=1}^n$ are generated independently from the following distributions:
\begin{itemize}
\item \textbf{Model 5}.  $\bu \sim (1/\sqrt{2}) \, t_{4}(\mo, \bSigma_u)$;

\item \textbf{Model 6}.  $\bu \sim \pi \bu_1   +  (1-\pi) (\bu_2 - \bu_3 ) $, where $\bu_1, \bu_2, \bu_3$ and $\pi$ are independent and satisfy $\bu_1 \sim (1/\sqrt{2}) \, t_{4}(\mo, \bSigma_u)$, $\bu_2 , \bu_3 \sim (1/4) \exp\{ N(\mo, \bSigma_u)\}$, and $\mathbb{P}(\pi=1)=0.6$, $\mathbb{P}(\pi=0)=0.4$;

\item \textbf{Model 7}.  $\bu \sim \pi \bu_1   + (1-\pi) (\bu_2 - \bu_3 ) $, where $\bu_1, \bu_2=(u_{21}, \ldots, u_{2p})^\T, \bu_3=(u_{31}, \ldots, u_{3p})^\T$ and $\pi$ are independent and satisfy $\bu_1 \sim (1/\sqrt{2}) \,  t_{4}(\mo, \bSigma_u)$, $u_{2j}, u_{3j} \overset{{\rm i.i.d.}}{\sim}   (1/2) \mbox{Weibull}(0.75, 0.75)$, and $\mathbb{P}(\pi=1)=0.25$, $\mathbb{P}(\pi=0)=0.75$;

\item \textbf{Model 8}.  $\bu \sim \pi \bu_1   + (1-\pi) \bu_2 $, where $\bu_1, \bu_2$ and $\pi$ are independent and satisfy $\bu_1 \sim (1/\sqrt{2}) \,  t_{4}(\mo, \bSigma_u)$, $\bu_2 \sim N(\mo, \bSigma_u)$ and $\mathbb{P}(\pi=1)=0.9$, $\mathbb{P}(\pi=0)=0.1$.

\end{itemize}

To use normal calibration, we need to estimate the variance $\sigma_{jj}=\var(X_j)$. We compute the median-of-means estimator $\wt \sigma_{jj}(V)$ with a universal parameter $V= \ceil{0.5 \log(pn)}$ for all $j$. Since we need to subtract the common variance estimator according to \eqref{mom-var}, the final estimate of $\var(u_j)$ may sometimes be too close to zero. To slightly improve numerical performance, we proceed as follows: (i) For each $j$, compute $\wt \sigma_{jj}(v)$ for $v=1,\ldots,V$; (ii) remove those $\wt \sigma_{jj}(v)$'s that are smaller than $\hat \bb_j^\T \hat{\bSigma}_f \hat \bb_j$; (iii) taking the 0.75 quantile of the remaining $\wt \sigma_{jj}(v)$'s as the final estimate of $\sigma_{jj}$.  The numerical results indicate that this modified procedure is numerically stable.

In the simulations reported here, we take $p = 2000$, $n = 80, 120$, $\mu_j= \mu$ for $1\leq j\leq \pi_1p$ and $\mu_j=0$ otherwise, where $\mu=\sqrt{2(\log p)/n}$ and $\pi_1=0.25$. For simplicity, we set $\lambda = 0.5$ in our procedure and use the Matlab package \texttt{mafdr} to compute the estimate $\hat{\pi}_0(\lambda)$ of $\pi_0=1-\pi_1$. All the results are based on 500 simulation rounds. From Tables~\ref{tab3} and \ref{tab4} we see that, in the presence of heavy-tailed errors, the RD-A procedure provides consistently better controls of the FDR at the expense of slight compromises of the FNR and TPR. Figure~\ref{fig:model6} compares the performance 
of the three methods, for various sample sizes and signal strengths, in the cases of Model 2 and Model 6. Wee see that the RD-A consistently outperforms the two other methods, across all sample sizes and even for low signal strengths when all the methods exhibit higher errors. 

\begin{table}[tbh]\centering
{\begin{tabular}{ccrrrrrrrrr}
\hline \vspace{-0.25cm} \\
&   &  \multicolumn{6}{c}{Student's $t$}      \\
 &  & \multicolumn{3}{c}{$n=80$} & \multicolumn{3}{c}{$n=120$}    \\
  \cline{3-5} \cline{6-8}    \vspace{-0.2cm}   \\
   &  & $\alpha=$5\% & 10\% & 20\%  & 5\%  &  10\%   & 20\%	\vspace{0.1cm} \\
 \hline  \vspace{-0.2cm}  \\
\multirow{2}{*}{FDR} & RD-A$_{{\rm N}}$ & 4.94\%  &  9.67\%  &  19.21\% &  4.73\% & 9.49\%  & 19.14\%   \\
 & OD-A & 7.15\%  &  13.27\%  &  24.90\% & 6.29\%  & 12.11\% & 23.27\%    \\
\hline \vspace{-0.2cm} \\
\multirow{2}{*}{FNR} &  RD-A$_{{\rm N}}$  &  3.20\% & 1.96\% & 1.00\% & 2.92\% & 1.76\% & 0.87\%     \\
 & OD-A &  2.39\% & 1.46\% & 0.75\% & 2.37\% & 1.42\%  & 0.72\%   \\
\hline \vspace{-0.2cm} \\
\multirow{2}{*}{TPR} &  RD-A$_{{\rm N}}$ & 90.08\% & 94.13\% & 97.15\% & 91.01\%  & 94.75\%  & 97.53\% \\
 & OD-A &  92.74\% &  95.73\% & 97.95\% & 92.79\%  & 95.81\% & 98.01\%     \\
\hline \vspace{-0.25cm} \\
&   &  \multicolumn{6}{c}{Mixture Student's $t$/Lognormal}       \\
 &  & \multicolumn{3}{c}{$n=80$} & \multicolumn{3}{c}{$n=120$}    \\
  \cline{3-5}\cline{6-8} \cline{9-11}   \vspace{-0.2cm}   \\
   & & $\alpha=$5\% & 10\% & 20\%  & 5\%  &  10\%   & 20\%   \vspace{0.1cm} \\
 \hline  \vspace{-0.2cm}  \\
\multirow{2}{*}{FDR} &  RD-A$_{{\rm N}}$ & 4.43\%  &  8.91\%  &  18.16\% & 4.24\% & 8.80\% & 18.26\%    \\
 & OD-A &  6.96\%  &  13.22\%  &  25.07\% & 6.04\%  & 11.93\% & 23.34\%    \\
\hline \vspace{-0.2cm} \\
\multirow{2}{*}{FNR} & RD-A$_{{\rm N}}$  &  1.86\% & 1.10\% & 0.55\% & 1.79\% & 1.05\% & 0.52\%    \\
 & OD-A & 1.35\% & 0.82\% & 0.42\% & 1.44\% & 0.86\%  & 0.43\%  \\
\hline \vspace{-0.2cm} \\
\multirow{2}{*}{TPR} &  RD-A$_{{\rm N}}$  &  94.34\% & 96.73\% & 98.46\% & 94.54\%  & 96.88\%  & 98.54\%  \\
 & OD-A &  95.96\%  & 97.63\%  & 98.87\% & 95.67\%  & 97.48\% & 98.82\%  \\\hline
\end{tabular}
 \vspace{0.2cm}
\caption{Empirical FDR, FNR and TPR based on a factor model with dependent errors following a Student's $t$ distribution (Model 5) and a mixture Student's $t$/lognormal distribution (Model 6).  }
\label{tab3}}
\end{table}

\begin{table}[tbh]\centering
{\begin{tabular}{ccrrrrrrrrr}
\hline \vspace{-0.25cm} \\
&   &  \multicolumn{6}{c}{Mixture Student's $t$/Weibull}      \\
 &  & \multicolumn{3}{c}{$n=80$} & \multicolumn{3}{c}{$n=120$}    \\
  \cline{3-5} \cline{6-8}    \vspace{-0.2cm}   \\
   &  & $\alpha=$5\% & 10\% & 20\%  & 5\%  &  10\%   & 20\%	\vspace{0.1cm} \\
 \hline  \vspace{-0.2cm}  \\
\multirow{2}{*}{FDR} & RD-A$_{{\rm N}}$ & 4.44\%  &  9.01\%  &  18.24\% & 4.45\% & 9.09\%  & 18.46\%   \\
 & OD-A &  6.84\%  &  13.04\%  &  24.66\% & 6.18\%  & 12.04\% & 23.20\%    \\
\hline \vspace{-0.2cm} \\
\multirow{2}{*}{FNR} &  RD-A$_{{\rm N}}$  &  1.66\% & 0.94\% & 0.44\% & 1.43\% & 0.79\% & 0.36\%     \\
 & OD-A &  1.12\% & 0.64\% & 0.30\% & 1.09\% & 0.60\%  & 0.27\%   \\
\hline \vspace{-0.2cm} \\
\multirow{2}{*}{TPR} &  RD-A$_{{\rm N}}$ &94.97\% & 97.21\% & 98.75\% & 95.70\%  & 97.67\%  & 99.00\% \\
 & OD-A & 96.65\% & 98.15\% & 99.20\% & 96.76\%  & 98.25\% & 99.25\%   \\
\hline \vspace{-0.25cm} \\
&   &  \multicolumn{6}{c}{Mixture Student's $t$/Normal }       \\
 &  & \multicolumn{3}{c}{$n=80$} & \multicolumn{3}{c}{$n=120$}    \\
  \cline{3-5}\cline{6-8} \cline{9-11}   \vspace{-0.2cm}   \\
   & & $\alpha=$5\% & 10\% & 20\%  & 5\%  &  10\%   & 20\%   \vspace{0.1cm} \\
 \hline  \vspace{-0.2cm}  \\
\multirow{2}{*}{FDR} &  RD-A$_{{\rm N}}$ & 5.04\%  &  9.89\%  &  19.57\% & 4.76\% & 9.49\% & 18.98\%    \\
 & OD-A &  7.32\%  &  13.60\%  &  25.34\% & 6.29\%  & 12.10\% & 23.07\%    \\
\hline \vspace{-0.2cm} \\
\multirow{2}{*}{FNR} & RD-A$_{{\rm N}}$  &  3.16\% & 1.92\% & 0.98\% & 2.91\% & 1.75\% & 0.89\%    \\
 & OD-A & 2.34\% & 1.43\% & 0.73\% & 2.37\% & 1.43\%  & 0.73\%     \\
\hline \vspace{-0.2cm} \\
\multirow{2}{*}{TPR} &  RD-A$_{{\rm N}}$  &  90.25\% & 94.25\% & 97.23\% & 91.02\%  & 94.75\%  & 97.43\%  \\
 & OD-A &  92.90\% & 95.83\% & 98.04\% & 92.78\%  & 95.76\% & 97.97\%      \\\hline
\end{tabular}
 \vspace{0.2cm}
\caption{Empirical FDR, FNR and TPR based on a factor model with dependent errors following a mixutre Student's $t$/Weibull distribution (Model 7) and a mixture Student's $t$/normal distribution (Model 8).  }
\label{tab4}}
\end{table}

\begin{figure}      
         \centering
               \includegraphics[width=0.49\textwidth,]{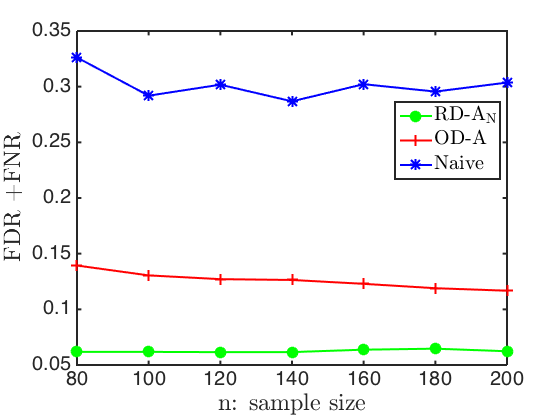}
               \includegraphics[width=0.49\textwidth,]{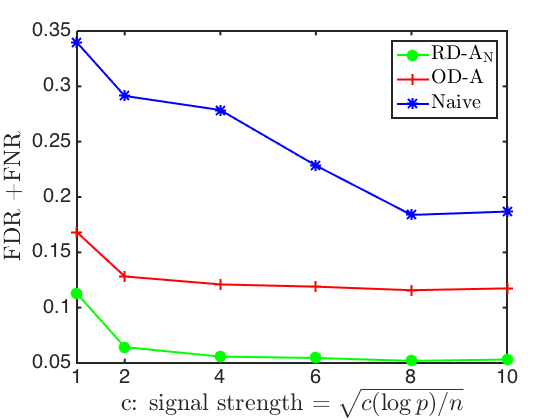}\\
                \includegraphics[width=0.49\textwidth,]{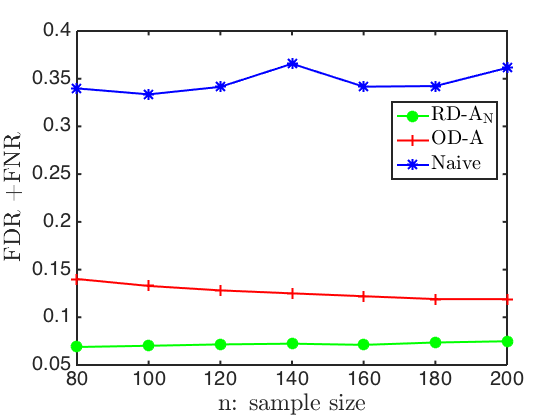}
               \includegraphics[width=0.49\textwidth,]{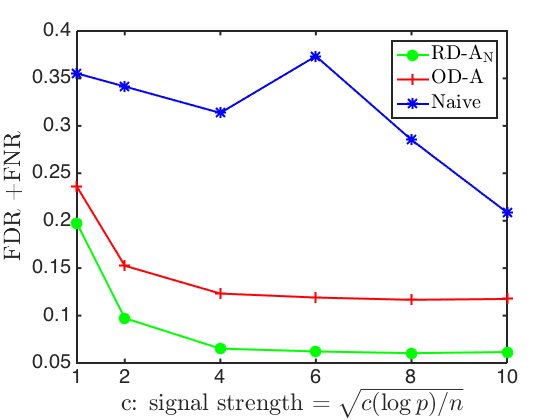}             
           \caption{Empirical FDR and FNR of the testing problem at the $10\%$ significance level, when the data follows the distribution in Model 2 (Student's $t_{2.5}$) in the top panel and Model 6 (mixture of Student's $t$/Lognormal) in the bottom panel. On the left panel, the sample size varies, with the signal strength fixed at $\mu=\sqrt{2(\log p)/n}$. On the right panel, $n$ is fixed at $120$ and we vary the value of $c$, where $\mu=\sqrt{c(\log p)/n}$. Here $p=2000$, the proportion of true signals is $0.25$ and tuning parameter $\tau_j = \,\hat{\sigma}_j \sqrt{n/\log(pn)}$.}
            \label{fig:model6}
\end{figure}

\end{document}